\documentclass[11pt,twoside,a4paper]{article}
\usepackage{color,amscd,amsmath,amssymb,amsthm,latexsym,stmaryrd}
\usepackage{mathabx}
\usepackage{shuffle}
\usepackage[latin1,utf8]{inputenc}
\usepackage[T1]{fontenc}   
\usepackage[english]{babel}
\usepackage{tikz,tkz-tab}
\usepackage{hyperref}
\usepackage{appendix}
\usepackage{enumitem}
\newcommand{\sy}[1]{{\color{purple} #1}}

\input{xy}
\xyoption{all}

\definecolor{vert}{rgb}{0.,0.5,0.}

\newcommand{\Gr}{\mathbf{Gr^{\circlearrowright}}}
\newcommand{\PGr}{\mathbf{ {CGr}^{\circlearrowright}}}
\newcommand{\rGr}{\mathbf{{sol}Gr^{\circlearrowright}}}
\newcommand{\rPGr}{\mathbf{ {solCGr}^{\circlearrowright}}}
\newcommand{\uPGr}{\mathbf{uPGr^{\circlearrowright}}}

\newcommand{\Gacirc}{\Gamma^{\circlearrowright}}
\newcommand{\rGacirc}{{\bf sol}\Gamma^{\circlearrowright}}
\newcommand{\Ga}{\Gamma^{\uparrow}}
\newcommand{\Hom}{\mathbf{Hom}}

\newcommand{\reg}{\mathrm{sol}}
\newcommand{\sh}{\mathrm{sh}}
\newcommand{\val}{\mathrm{val}}

\newcommand{\Trap}{\mathbf{TRAP}}
\newcommand{\uTrap}{\mathbf{uTRAP}}

\newcommand{\K}{\mathbb{K}}
\newcommand{\Z}{\mathbb{Z}}
\newcommand{\C}{\mathbb{C}}
\newcommand{\R}{\mathbb{R}}
\newcommand{\N}{\mathbb{N}}

\newcommand{\rond}[1]{*++[o][F-]{#1}}

\renewcommand{\geq}{\geqslant}
\renewcommand{\leq}{\leqslant}

\newcommand{\sym}{\mathfrak{S}}

\newcommand{\calC}{\mathcal{C}}

\newcommand{\grapheo}{\mathcal{O}}
\newcommand{\End}{\mathrm{End}}
\newcommand{\Obj}{\mathrm{Obj}}
\newcommand{\Mor}{\mathrm{Mor}}

\newcommand{\cat}{\mathcal{C}}
\newcommand{\catssm}{\mathbf{Mod}_\sym }

\theoremstyle{plain}
\newtheorem{defi}{Definition}[subsection]
\newtheorem{theo}[defi]{Theorem}
\newtheorem{lemma}[defi]{Lemma}
\newtheorem{cor}[defi]{Corollary}
\newtheorem{prop}[defi]{Proposition}

\theoremstyle{remark}
\newtheorem{remark}{Remark}[subsection]
\newtheorem{notation}{Notations}[subsection]
\newtheorem{example}{Example}[subsection]

\begin{document}
	
	\title{From non-unitary wheeled PROPs to smooth amplitudes and  generalised convolutions}
	
	\author{Pierre J. Clavier${}^{1,2,3}$, Lo\"ic Foissy${}^4$, Sylvie Paycha${}^{3,5}$\\
		~\\
		{\small \it $^1$ Université de Haute Alsace, IRIMAS,}
		{\small \it 12 rue des Frères Lumière,}\\
		{\small \it 68 093 MULHOUSE Cedex, France}\\ 
		~\\
		{\small \it $^2$ Technische Universit\"at, Institut f\"ur Mathematik,}\\
		{\small \it Str. des 17. Juni 136,}\\
		{\small \it 10587 Berlin, Germany}\\
		~\\
        {\small \it $^3$ Universit\"at Potsdam, Institut f\"ur Mathematik,} \\
		{\small \it Campus II - Golm, Haus 9} \\
		{\small \it Karl-Liebknecht-Stra\ss e 24-25} \\
		{\small \it D-14476 Potsdam, Germany}\\ 
		~\\
		{\small \it $^4$ Univ. Littoral Côte d'Opale, UR 2597}\\
		{\small \it LMPA, Laboratoire de Mathématiques Pures et Appliquées Joseph Liouville}\\
		{\small \it F-62100 Calais, France}\\ 
		~\\
		{\small \it $^5$ On leave from the University of Clermont Auvergne, Clermont-Ferrand, France}\\
		~\\
		{\small emails: \it pierre.clavier@uha.fr, foissy@univ-littoral.fr, paycha@math.uni-potsdam.de}}
	
	\date{}
	
	\maketitle

	\begin{abstract}  We introduce the concept of TRAP  (Traces and Permutations), which can roughly be viewed as a wheeled PROP (Products and Permutations) without  unit. TRAPs are equipped with a horizontal concatenation and partial trace maps. {Continuous morphisms on an infinite dimensional topological space and smooth kernels   (resp. smoothing operators) on a closed manifold   form a TRAP but not a wheeled PROP}. We build the free objects in the category of TRAPs   as TRAPs of graphs  and  show that a TRAP can be completed to a unitary TRAP (or wheeled PROP). We further show that it can be equipped with a   vertical concatenation, which 
				on the TRAP of linear homomorphisms of a vector space, amounts to the usual composition.
			 {	The vertical concatenation in the TRAP of smooth kernels  gives  rise to generalised convolutions. 
					 Graphs  whose vertices are decorated by smooth kernels (resp. smoothing operators) on a closed manifold form a TRAP}. 
			From their universal properties we build  smooth  amplitudes associated with the graph.
	\end{abstract}  
	
	\noindent
	{\bf Classification:} 18M85, 46E99, 47G30
	
	\tableofcontents
	
	\addcontentsline{toc}{section}{Introduction}
	
\section*{Introduction}

	\subsection*{State of the art}

	PROPs (\underline{Pro}ducts and \underline{P}ermutations)\footnote{The traditional notation  is PROP, more recently prop.} 
	 provide an 
	algebraic structure that allows to deal with an arbitrary number of inputs and outputs. They generalise many other algebraic structures such as 
	operads, which have one output and multiple inputs. PROPs appeared in \cite{McLane} and later in the book \cite{BV73} in the context of cartesian 
	categories. Although operads stemmed from  the study of iterated loop spaces in algebraic topology, see e.g. \cite{May72}, their origin can also be traced back to the earlier work \cite{BV68}\footnote{We 
		thank B. Vallette for his enlightening comments on these historical aspects.}.
	
	An important asset of PROPs over operads is that they encompass algebraic structures such as bialgebras and Hopf algebras that lie outside the realm of 
	operads or co-operads. This very fact is a motivation to consider PROPs in the context of renormalisation in quantum field theory. We refer the reader to \cite{Pirashvili01} for the study of bialgebras in the PROPs framework and 
	\cite{Markl,JY15,McLane} for other 
	classical examples of PROPs.  In recent years, wheeled PROPs  \cite{Merkulov2004,Merkulov2009}, which allow for loops, have played an important role in the context of deformation quantisation.
	
	A central example  of PROP is the PROP $\Hom_V$ of homomorphisms  of a finite dimensional vector space $V$ which we generalise 
	to the PROP $\Hom^c_V$ of continuous homomorphisms  of a  nuclear Fréchet space $V$. Whereas the first is a wheeled PROP   (Proposition \ref{prop:HomV}), the latter is not unless the space $V$ is  finite dimensional (Theorem \ref{thm:Hom_V_generalised}). It can nevertheless be interpreted as a TRAP (Definition \ref{defn:trap}), which roughly speaking, amounts to a wheeled PROP without unit\footnote{To our knowledge, wheeled PROPs without units do not appear in the literature, which is why we allow ourselves to give them a shorter name.}. TRAPs  introduced in this paper,  offer natural structures to host morphisms of infinite dimensional spaces (see Proposition \ref{prop:Homfr}) and are therefore expected to play a role in the context of renormalisation in quantum field theory. 
	
	Another class of important examples we consider, are  TRAPs    of graphs  (Proposition \ref{prop:PROPgraphs})
	 of various types.
	 In the context of deformation quantisation,  the complex of oriented graphs whether 
	directed or wheeled,  plays an important role in the construction of a free PROP   generated by a $\sym\times \sym^{op}$-module  (see e.g  
	\cite[Paragraph 2.1.3]{Merkulov2004}). 

	Our long term   goal is to use the TRAP structure of graphs decorated by distribution (e.g. Green) kernels in order to build amplitudes  as generalised convolutions (called $P$-amplitudes, see Definition \ref{defi:generalised-P-convolution}) of kernels associated with the 
	decorated graph. The  expected singularities of the resulting amplitudes are immediate obstacles in defining such generalised convolutions.  In this paper, we  focus on the smooth setup, in which case  the amplitudes are smooth.
	
	\subsection*{Feynman rules and TRAPs}	
	
	In space-time variables, a Feynman rule is expected to assign to a graph $G$ with $k$ incoming and $l$ outgoing edges, an amplitude
	(it is actually a distribution) $K_G$ in $k+l$ variables. Our long term goal is  to derive the existence and the properties of the map $G\mapsto K_G$ from a universal property 
	of the PROP  structure on graphs.
	
	By  means of  blow-up   methods, generalised convolutions of Green functions were built on a closed Riemannian manifold  in    \cite{DangZhang17}, with the goal of renormalising multiple loop 
	amplitudes for Euclidean QFT on Riemannian manifolds.
	We hope to be able to simplify    the intricate analytic aspects of the renormalisation procedure for   multiple loop amplitudes,  by adopting an algebraic point of view on  amplitudes 
	using TRAPs.  There were    earlier attempts to describe QFT theories in 
	terms of PROPs (see e.g. \cite{Ionescu,Ionescu07}), yet to our knowledge, none with  the focus we are putting on generalised convolutions to 
	describe amplitudes. 	
	
We therefore expect the PROP of oriented graphs, briefly  mentioned  in \cite{Ionescu07}, and more specifically TRAPs, their non-unitary counterparts which naturally arise in the infinite-dimensional set-up, to have  concrete applications in the 
	perturbative approach to quantum field theory. {To our knowledge, this is yet an unexplored aspect of the theory.  Filling in this gap is  a long term goal we  have in mind.
	A first step towards this goal is the study of the TRAP of smoothing symbols (Theorem \ref{theo:Kinfty}), which like 
	$\Hom_V^c$ is not a wheeled PROP due to the  infinite dimensional spaces it involves.
	
	\subsection*{TRAPs of  graphs}
	
	TRAPs and unitary TRAPs entail two operations, the horizontal concatenation, and the  partial trace map}. The difference between TRAPs and unitary TRAPs is the existence of a unit for the trace in the latter.
		We define a TRAP structure on various families of graphs, which can be {corolla ordered} \sy(Definition \ref{defi:graph}) or decorated (Proposition \ref{prop:PROPgraphs}).
		The horizontal concatenation of this TRAP is the natural concatenation of 
		graphs and the partial trace map consists in gluing together
		one of the input with one of the output, and therefore assigns to a graph $G$ with $k$ incoming and $l$ outgoing edges
		a graph with $k-1$ incoming and $l-1$ outgoing edge. The  set of {corolla ordered} graphs $\PGr$ equipped with the partial trace map builds
		a unitary TRAP, and we prove that it is a free unitary TRAP (Theorem \ref{freetraps}):
		this is the TRAP counterpart of a similar statement for free PROPs, described in terms of graphs without loops 	\cite[Proposition 57]{Markl} and   \cite{Vallette1,Vallette3,JY15}.
		More generally,
		the set of {corolla ordered} graphs $\PGr(X)$ decorated by a set $X$ on their vertices is the free unitary TRAP generated by $X$.
		These unitary TRAPs contain free nonunitary TRAPs, which are  combinatorially described by particular graphs,
		which we call solar\footnote{ {In \cite{JY15} such graphs are called ordinary.}}.

	\subsection*{From  TRAPs to unitary TRAPs}
	
	If $V$ is a finite-dimensional vector space,
		then the PROP $\Hom_V$  of homomorphisms of $V$ is a  unitary TRAP, with the usual trace of endomorphisms.
		Its unit as a TRAP 	is the identity map of $V$. When $V$ is not finite-dimensional, one cannot equip the whole PROP $\Hom_V$
		with	a structure of unitary TRAP. 
In this case,  one has to restrict to smaller classes of homomorphisms, such as that of the 
		homomorphisms of finite rank. This class no longer contains the identity, and we only obtain a TRAP and not   
		a unitary TRAP (Proposition \ref{prop:Homfr}). 
		To circumvent this difficulty, we construct for any TRAP $P$ a unitary TRAP $\uPGr(P)$ which contains $P$ (Theorem \ref{theo:uPGr}).
		This object is characterised by a universal property (Proposition \ref{prop:univpropPGr}), which amounts to  applying the left adjoint to
		the forgetful functor from the category of unitary TRAPs to the category of TRAPs.
		The existence of this functor comes from the inclusion of {corolla ordered} { solar} graphs, describing nonunitary free TRAPs,
		in  the set of {corolla ordered} graphs, describing unitary free TRAPs.	
		In particular, in $\uPGr(P)$ an identity $I$ is added, as well as its trace, symbolised by an abstract element $\grapheo$,
		which is no longer an element of the base field $\K$.

	\subsection*{The vertical concatenation and generalised traces}	
	
	In the same way as a vertical concatenation can be built on wheeled PROPs to turn them into PROPs, any TRAP comes with a   vertical concatenation (Proposition \ref{propverticalconcatenation}). 
		The vertical concatenation on TRAPs 
		previously considered in \cite[Definition 11.33]{JY15}   generalises  the composition of morphisms.
		When applied to unitary TRAPs, this construction yields a functor from unitary TRAPs to PROPs (Proposition \ref{prop:morphismPROPfromTRAP}).		
	Roughly speaking, 	on graphs, if $G$ is a graph with $k$ inputs and $l$ outputs,
		if $G'$ is a graph with $l$ inputs and $m$ outputs, 
		$G'\circ G$ is obtained by gluing together the  outgoing edges of $G$ and the  incoming edges of $G'$ according to their indexation,
		giving a graph with $k$ inputs and $m$ outputs. The universal property of graphs allows to generalise the construction of the vertical concatenation to arbitrary traces.
		If $V$ is a finite-dimensional space, the vertical concatenation of the TRAP $\Hom_V$ coincides with the usual composition of 
		linear maps $f:V^{\otimes k}\longrightarrow V^{\otimes l}$.

		Extending this to the infinite dimensional setup requires the use of a completed tensor product $\widehat \otimes$ in order to have an 
		isomorphism 
		\[\Hom_V^c(k,l)\simeq\left(V'\right)^{\widehat\otimes  k}\widehat \otimes V^{\widehat\otimes  l},\]  
		where $\Hom_V^c(k,l)$ stands for the algebra of continuous morphisms from $V^{\widehat\otimes  l}$ to $V^{\widehat\otimes  k}$ (see 
		Definition \ref{defi:Hom_V_generalised}) and
		$V'$ for the topological dual of a topological space $V$. This holds in the framework of Fr\'echet nuclear spaces which form a monoidal 
		category under
		the completed tensor product   (Lemma \ref{lem:prod_Frechet_nuclear}). On Fr\'echet nuclear spaces, the composition can 
		indeed be described as a dual pairing,   so it  comes as no surprise that 
		for a Fr\'echet nuclear space $V$, the vertical concatenation obtained from the nonunitary TRAP structure
		is the usual composition.
		
	Any TRAP also inherits a generalised trace defined on its elements with the same number of inputs and outputs. Roughly speaking, these generalised traces are obtained by grafting the outputs to the inputs according to their indexation.
	These traces on TRAPs generalise the usual trace of morphisms, and they also enjoy a cyclicity property (Proposition \ref{prop:generalised_traces}).
		
When $V$ is a space of smooth functions on a closed Riemannian manifold $M$, the associativity of the vertical concatenation amounts to the Fubini property (Theorem \ref{thm:generalised_convolution}, 2.) and the generalised trace of a generalised kernel $K$ with $k$ inputs and $k$ outputs is given by the integration of $K$ along the small diagonal of $M^k$ (Theorem \ref{thm:generalised_convolution}, 3.).

	 \subsection*{Amplitude of a graph decorated by a TRAP}

As mentioned above,    our goal   in the present paper is to 
	provide an adequate algebraic and analytic framework in which we build generalised convolution   functions associated with graphs decorated  with smooth kernels.
	We show that these form a TRAP (Theorem \ref{theo:Kinfty}), whose partial trace maps are given by a partial convolution. 

When $P$ is a TRAP, the universal property of the TRAP of {corolla ordered} graphs decorated by $P$
	 gives rise to a canonical TRAP map, which associates to any such graph $G$ an element of $P$ which we call the $P$-amplitude associated with $G$
	 (Definition \ref{defi:generalised-P-convolution}).
	 The $P$-amplitude commutes with both horizontal and vertical concatenation of $P$ (Proposition \ref{prop:gen_conv_vert_conc}). It   turns out that the vertical concatenation is in fact a particular example 
	 of $P$-amplitude, where the considered graph has two vertices. 
	 	  When applied to the TRAP of smooth generalised kernels, this construction generalises the usual convolution of kernels (Remark \ref{rk:generalised_conv}) and gives rise to smooth amplitudes.

	\subsection*{Unitary TRAPs and wheeled PROPs}
	
	A unitary TRAP is known in the literature under the name of wheeled PROP. 
		In order to prove that the two notions coincide, we describe TRAPs and unitary TRAPS as algebras over a monad (see Definition \ref{defimonades})
        which generalises the notion of monoid to  the frame of category theory.
		We state that  unitary TRAPS are algebras over a monad $\Gacirc$ of  graphs,
		described as an  endofunctor of a category of modules over symmetric groups
		sending an object  $X$ to the free  unitary TRAP of graphs  $\rGacirc(X)$ generated by $X$.
		When $X$ is a unitary TRAP, $\Gacirc(X)$ inherits a contraction operation to $X$, which induces the monadic structure 
		(Theorem  \ref{thm:equivalence_TRAP_wPROP}).
		This monad $\Gacirc$ turns out to be the monad used to defined wheeled PROPs in the literature \cite[Corollary 11.35]{JY15},
		thus relating our presentation of unitary TRAPs in terms of 
		family of sets with maps satisfying a set of axioms and the categorical presentation of wheeled PROPs
		in terms of algebras over a particular monad. A similar result holds for (nonunitary) TRAPs, replacing graphs by 
		solar graphs  introduced in Definition   \ref{defi:graph}.

	\subsection*{Openings}
	
To sum up,  by means of a  TRAP structure, we were able to build  generalised convolutions (resp. traces) associated with   graphs decorated with smooth 
	kernels. As announced at the beginning of the introduction, we expect this algebraic approach  to enable us to tackle non smooth kernels and   thus  to describe  (non necessarily smooth) amplitudes as generalised 
	convolutions of distribution kernels associated with graphs. At this stage these are open questions that we hope to address in future work.
	
	{There are other possible natural  generalisations of the framework presented here, that  are more algebraic in nature\footnote{{We thank Mark Johnson for pointing   out the subsequent interesting questions to us.}}.   { One could consider} coloured TRAPs,  whose input  and output  edges are coloured, and whose  partial trace maps relate inputs and outputs of the same colour. Such structures are expected to play a role in QFTs with more than one type of particles (for example QED and QCD). { Coloured TRAPs could also be relevent in the more geometric context of maps between different manifolds or in the context of} modules over an algebra.
	
	 There are also potential generalisations of Theorem \ref{theo:Kinfty}, which only requires that there be enough integral-like objects to define the partial trace maps, thereby hinting  to the fact that more general spaces    than the ones considered here should also carry TRAPs structures.
	 Weakened versions of $C^*$-algebras such  as inverse limits of $C^*$-algebras \cite{Phillips} and locally multiplicative convex $C^*$-algebras \cite{JoachimJohnson}  would be worth investigating in that context.   }
	
	\section*{Acknowledgements} 
	
	The first and last authors would like to acknowledge the Perimeter Institute for hosting them during the initial phase of this paper.
	They are also grateful for inspiring discussions with Matilde Marcolli. The three authors most warmly  thank  Bruno 
	Vallette and Dominique Manchon for pointing the wheeled PROP literature to us. The article was  restructured in the 
	aftermath of their suggestions and other remarks we got on a preliminary version of this manuscript, all of which 
			we are very grateful for.  { Our warmest thanks to Christian Brouder, Daniel Grieser and Mark Johnson for their very enlightening comments.}

\section{The category of TRAPs}

\subsection{Definition}\label{defn:trap}
\begin{notation}{For any $k\in\N_0$, we write $[k]:=\{1,\cdots,k\}$. In particular, $[0]=\emptyset$.
$\sym_k$ denotes the symmetric group on $k$ elements. An element $\sigma\in \sym_k$  sends $i\in [k]$ to $ \sigma(i)\in [k]$.}\end{notation}

\begin{defi} \label{defi:Trap}
A \textbf{TRAP} is a family $P=(P(k,l))_{k,l\geqslant 0}$ of sets, equipped with the following structures:
\begin{enumerate}
\item $P$ is a $\sym\times \sym^{op}$-module, that is to say, 
for any $k,l\in  \N_0$, $P(k,l)$ is a $\sym_l\times \sym_k^{op}$-module.  In other words, there exist maps
\begin{align*}
&\left\{\begin{array}{rcl}
\sym_l\times P(k,l)&\longrightarrow&P(k,l)\\
(\sigma,p)&\mapsto&\sigma\cdot p,
\end{array}\right.&
&\left\{\begin{array}{rcl}
P(k,l)\times \sym_k&\longrightarrow&P(k,l)\\
(p,\tau)&\mapsto&p\cdot \tau,
\end{array}\right.
\end{align*}
such that for any $(k,l)\in  \N_0^2$, for any $(\sigma,\sigma',\tau,\tau')\in \sym_l^2\times \sym_k^2$, for any $p\in P(k,l)$,
\begin{align*}
&&\mathrm{Id}_{[l]}\cdot p&=p\cdot \mathrm{Id}_{[k]}=p,\\
\sigma\cdot (\sigma'\cdot p)&=(\sigma\sigma')\cdot p,&
\sigma \cdot (p\cdot \tau)&=(\sigma\cdot p)\cdot \tau,&
(p\cdot\tau)\cdot \tau'&=p\cdot(\tau\tau').
\end{align*}
\item For any $k,l,k',l'\in  \N_0$, there is a map
\begin{align*}
*:&\left\{\begin{array}{rcl}
P(k,l)\times P(k',l')&\longrightarrow&P(k+k',l+l')\\
(p, p')&\longrightarrow&p*p',
\end{array}\right.
\end{align*}
called the  \textbf{horizontal concatenation}, such that:
\begin{enumerate}
\item (Associativity). For any $(k,l,k',l',k'',l'')\in  \N_0^6$, for any $(p,p',p'')\in P(k,l)\times P(k',l')\times P(k'',l'')$,
\[(p*p')*p''=p*(p'*p'').\]
\item (Unity). There exists $I_0\in P(0,0)$ such that for any $(k,l)\in  \N_0^2$, for any $p\in P(k,l)$,
\[I_0*p=p*I_0=p.\]
\item (Compatibility with the symmetric actions).
For any $(k,l,k',l')\in  \N_0^4$, for any $(p,p')\in P(k,l)\times P(k',l')$,
for any $(\sigma,\tau,\sigma',\tau')\in \sym_l\times \sym_k\times \sym_{l'} \times \sym_{k'}$,
\[(\sigma\cdot p\cdot \tau)*(\sigma'\cdot p'\cdot \tau')
=(\sigma \otimes \sigma')\cdot (p*p')\cdot (\tau\otimes \tau'),\]
where, for any $(\alpha,\beta)\in \sym_m\otimes \sym_n$, $\alpha\otimes \beta \in \sym_{m+n}$
is defined by
\[\alpha\otimes \beta(i)=\begin{cases}
\alpha(i)\mbox{ if }i\leqslant m,\\
\beta(i-m)+m\mbox{ if }i>m.
\end{cases}\]
\item (Commutativity). For any $(k,l,k',l')\in  \N_0^4$, For any $p\in P(k,l)$, $p'\in P(k',l')$,
\[c_{l,l'}\cdot (p*p')=(p'*p)\cdot c_{k,k'},\]
where for any $(m,n)\in  \N_0^2$, $c_{m,n}\in \sym_{m+n}$ is defined by:
\begin{align} 
\label{defcmn} c_{m,n}(i)&=\begin{cases}
i+n\mbox{ if }i\leq m,\\
i-m\mbox{ if }i>m.
\end{cases}
\end{align}
\end{enumerate}
\item For any $k,l\geqslant 1$, for any $i\in [k]$, $j\in [l]$, there is a map
\begin{equation}\label{eq:partialtrace}
t_{i,j}:\left\{\begin{array}{rcl}
P(k,l)&\longrightarrow&P(k-1,l-1) \\
p&\mapsto&t_{i,j}(p),
\end{array}\right.
\end{equation}
called the \textbf{partial trace map}, such that:
\begin{enumerate}
\item (Commutativity). For any $k,l\geqslant 2$, for any $i\in [k]$, $j\in [l]$, $i'\in [k-1]$, $j'\in [l-1]$,
\begin{align*}
t_{i',j'}\circ t_{i,j}&=\begin{cases}
t_{i-1,j-1}\circ t_{i',j'}\mbox{ if }i'<i,\: j'<j,\\
t_{i,j-1}\circ t_{i'+1,j'}\mbox{ if }i'\geqslant i,\: j'<j,\\
t_{i-1,j}\circ t_{i',j'+1}\mbox{ if }i'<i,\: j'\geqslant j,\\
t_{i,j}\circ t_{i'+1,j'+1}\mbox{ if }i'\geqslant i,\: j'\geqslant j.
\end{cases}
\end{align*}
\item (Compatibility with the symmetric actions). For any $k,l\geqslant 1$, for any $i\in [k]$, $j\in [l]$,
$\sigma \in \sym_l$, $\tau\in \sym_k$, for any $p\in P(k,l)$,
\[t_{i,j}(\sigma\cdot p\cdot \tau)=\sigma_j\cdot (t_{\tau(i),\sigma^{-1}(j)}(p))\cdot \tau_i,\]
with the following notation: if $\alpha\in \sym_n$ and $p\in [n]$, then $\alpha_p\in \sym_{n-1}$ is defined by
\begin{equation}
\label{defalphak} \alpha_p(k)=\begin{cases}
\alpha(k)\mbox{ if }k<\alpha^{-1}(p) \mbox{ and }\alpha(k)<p,\\
\alpha(k)-1\mbox{ if }k<\alpha^{-1}(p) \mbox{ and }\alpha(k)>p,\\
\alpha(k+1)\mbox{ if }k\geqslant \alpha^{-1}(p) \mbox{ and }\alpha(k)<p,\\
\alpha(k+1)-1\mbox{ if }k\geqslant \alpha^{-1}(p) \mbox{ and }\alpha(k)>p.
\end{cases}
\end{equation}
In other words, if we represent $\alpha$ by a word $\alpha_1\ldots \alpha_n$, then $\alpha_p$ is represented
by the word obtained by deleting the letter $p$ in $\alpha_1\ldots \alpha_n$ 
and subtracting   $1$ to all the letters $>p$.
\item (Compatibility with the horizontal concatenation). 
For any $k,l,k',l'\geq 1$, for any $i\in [k+l]$, $j\in [k'+l']$, for any $p\in P(k,l)$, $p'\in P(k',l')$:
\[t_{i,j}(p*p')=\begin{cases}
t_{i,j }(p)*p'\mbox{ if }i\leqslant k,\: j\leqslant l,\\
p* t_{i-k,j-l}(p')\mbox{ if }i>k,\: j>l.
\end{cases}\]
\end{enumerate}\end{enumerate}
The TRAP is \textbf{unitary} if moreover
there exists $I\in P(1,1)$ such that for any $k,l\geqslant 1$, for any $i\in [k+1]$, $j\in [l+1]$,
for any $p\in P(k,l)$:
\begin{align*}
t_{1,j}(I*p)&= (1,2,\ldots,j-1)\cdot p \mbox{ if }j\geqslant 2,\\
t_{i,1}(I*p)&= p\cdot (1,2,\ldots, i-1)^{-1}\mbox{ if }i\geqslant 2,\\
t_{k+1,j}(p*I)&=(j,j+1,\ldots,{k})^{-1}\cdot p\mbox{ if }j\leqslant {k},\\
t_{i,l+1}(p*I)&=p\cdot (i,i+1,\ldots,{l})\mbox{ if }i\leqslant {l}.
\end{align*}
\end{defi}

\begin{remark}
By commutativity of $*$, for any $p\in P(0,0)$, for any $(k,l)\in  \N_0^2$, for any $p'\in P(k,l)$:
\[p*p'=p'*p,\]
since $c_{0,k}=\mathrm{Id}_{[k]}$.
\end{remark}

\begin{remark}
The abuse of notation $t_{i,j}$ is legitimate since a full notation such as  $t_{i,j}^{k,l}$ is not necessary in practice. Indeed the indices  $
k$ and $l$ in $t_{i,j}(p)$ are entirely determined by  the element $p$ to which  $t_{i,j}$ is applied.

 More so,  if $P$ is unitary, $t_{i,j}{(p)}$ does not strongly depend on $k$ and $l$ {determined by $p$}: indeed, let  $f:P(k,l)\longrightarrow P(k+1,l+1)$ be the map that sends $p$ to  $p*I$ 
 then for  $i\in [k]$ and $j\in [l]$, we have
 \[t_{i,j} \circ f(p)=f\circ t_{i,j}(p),\] which is the axiom 3.(c).
\end{remark}

\begin{remark}
 Notice that our TRAPs can be seen as wheeled PROPs without units, or equivalently a  wheeled PROP is a TRAP with unit (see Theorem \ref{thm:equivalence_TRAP_wPROP} below). {As for wheeled PROPs, TRAPs can be defined in a much more concise way {as an} algebra over a given monad (see Section \ref{section:monad}). We choose {a}  more pedestrian definition of TRAPs {which} will allow us in Section \ref{section:Examples} to prove that known analytic and geometric spaces carry TRAP  structures. }
\end{remark}

\begin{defi} \label{defn:subtrap}
 { We call {\bf sub-TRAP} of a TRAP   $P=(P(k,l))_{k,l\geqslant 0}$, a  $\sym\times \sym^{op}$-submodule   $Q=(Q(k,l))_{k,l\geqslant 0}$ of $P$ which contains the unit $I_0\in P(0,0)$ and is stable under the partial trace map of $P$. {If the TRAP $P$ is unitary, then the} sub-trap $Q$ is unitary if it contains the unit $I\in P(1,1)$.}
\end{defi}

\begin{defi}       \label{defn:trap_morphism}
     Let $P=(P(k,l))_{k,l{\in \N_0}}$ and $Q=(Q(k,l))_{k,l\in \N_0}$ be two TRAPs with partial trace maps $(t_{i,j}^P)_{i,j{\in \N_0}}$ and $(t_{i,j}^Q)_{i,j{\in \N_0}}$ 
     respectively. A \textbf{morphism of TRAPs} is a family $\phi=(\phi(k,l))_{k,l{\in \N_0}}$ 
     of morphisms of $\sym\times \sym^{op}$-modules  $\phi(k,l):P(k,l)\longrightarrow Q(k,l)$ which are morphisms for the horizontal concatenation, 
     and the partial trace maps. More precisely, for any $(k,l,m,n)\in \N_0^4$:
     \begin{itemize}
     \item For any $(\sigma,p,\tau)$ in $\sym_l\times P(k,l)\times\sym_k$, $\phi(k,l)(\sigma\cdot p\cdot \tau)
     =\sigma\cdot \phi(k,l)(p)\cdot\tau$.
     \item $\phi(0,0)(I_0)=I_0$.
      \item $\forall (p,q)\in P(k,l)\times P(n,m),~\phi(k+n,l+m)(p* q) = \phi(k,l)(p)* \phi(n,m)(q)$,
      \item $\forall (p,i,j)\in P(k,l)\times [k]\times [l]$, $\phi(k-1,l-1)\circ t^P_{i,j}(p)=t^Q_{i,j}\circ \phi(k,l)(p)$.
     \end{itemize}
     With a slight  abuse of notations, we write $\phi(p)$ instead of $\phi(k,l)(p)$ for $p\in P(k,l)$.
     In particular, TRAPs form a category, which we denote by $\Trap$.\\
     If $P$ and $Q$ are unitary TRAPs with units $I_P$ and $I_Q$ and $\phi:P\longrightarrow Q$ is a morphism of TRAPs, this morphism is unitary if
     $\phi(1,1)(I_P)=I_Q$. Unitary TRAPs form a subcategory of $\Trap$, denoted by $\uTrap$.
  \end{defi}
  \begin{remark}
    In Theorem \ref{thm:equivalence_TRAP_wPROP} we shall see that $\uTrap$ coincides with the category of wheeled PROPs.
      \end{remark}
Let us simplify the axioms:

\begin{lemma}\label{lemmeaxiomessimples}
Let $P=(P(k,l))_{k,l\in  \N_0}$ be a $\sym\times \sym^{op}$-module, equipped with a horizontal concatenation $*$
satisfying axioms 2. (a)-(d), and with maps $t_{i,j}$ satisfying axioms 3. (a)-(b). 
\begin{enumerate}
\item We assume that for any $k,l,k',l'\geqslant 1$, for any $p\in P(k,l)$, $p'\in P(k',l')$,
\[t_{1,1}(p*p')=t_{1,1}(p)*p'.\]
Then Axiom 3.(c) is satisfied.
\item Let $I\in P(1,1)$. We assume for   any $k,l\geqslant 1$, for any $p\in P(k,l)$,
\[t_{1,2}(I*p)=p.\]
Then $I$ is a unit of $P$.
\end{enumerate}
\end{lemma}

\begin{proof} 
1. Let $p\in P(k,l)$ and $p'\in P(k',l')$. Let us take $i\in [k+l]$, $j\in [k'+l']$, consider the transpositions $\sigma=(1,j)$ and $\tau=(1,i)$, with the convention $(1,1)=\mathrm{Id}$.
If $i\leqslant k$ and $j\leqslant l$,  then:
\begin{align*}
t_{i,j}(p*p')&=t_{i,j}(\sigma^2\cdot( p*p')\cdot \tau^2)\\
&=\sigma_j\cdot t_{1,1}(\sigma\cdot (p*p')\cdot \tau)\cdot \tau_i\\
&=\sigma_j\cdot(t_{1,1}((\sigma\cdot p\cdot \tau)*p')\cdot \tau_i\\
&=\sigma_j \cdot (t_{1,1}(\sigma\cdot p\cdot \tau)*p')\cdot \tau_i\\
&=(\sigma_j\cdot (t_{1,1}(\sigma\cdot p\cdot \tau)\cdot \tau_i)*p'\\
&=t_{i,j}(p)*p'.
\end{align*}
If $i>k$ and $j>l$, using $c_{m,n}^{-1}=c_{n,m}$:
\begin{align*}
t_{i,j}(p*p')&=t_{i,j}(c_{l',l}\cdot (p'*p)\cdot c_{k,k'})\\
&=(c_{l',l})_j\cdot t_{i-k,j-l}(p'*p)\cdot (c_{k,k'})_i\\
&=c_{l'-1,l}\cdot (t_{i-k,j-l}(p')*p)\cdot c_{k,k'-1}\\
&=p*t_{i-k,j-l}(p').
\end{align*}

2. Let us take $j\geqslant 2$.
\begin{align*}
t_{1,j}(I*p)&=t_{1,j}((2,j)^2\cdot(I*p))\\
&=(2,\ldots,j-1)\cdot t_{1,2}((2,j)\cdot (I*p))\\
&=(2,\ldots,j-1)\cdot t_{1,2}(I*(1,j-1)\cdot p))\\
&=(2,\ldots,j-1)\cdot((1,j-1)\cdot p)\\
&=(2,\ldots,j-1)(1,j-1)\cdot p\\
&=(1,\ldots,j-1)\cdot p.
\end{align*}
The  other three relations are proved in the same way. 
\end{proof}
We can also simplify the axioms for morphisms of TRAPs.
\begin{lemma}\label{lemmemorphismes}
Let $P$ and $Q$ be two TRAPs and $\phi:P\longrightarrow Q$ be a map. We assume that:
\begin{enumerate}
\item For any $(k,l)\in  \N_0^2$, for any $(\sigma,\tau) \in \sym_l\times\sym_k$, for any $x\in P(k,l)$,
\[\phi(\sigma\cdot x\cdot \tau)=\sigma\cdot \phi(x)\cdot \tau.\]
\item For any $k,l\geqslant 1$, for any $x\in P(k,l)$,
\[t_{1,1}\circ \phi(x)=\phi\circ t_{1,1}(x).\]
\end{enumerate}
Then $\phi$ is a map of TRAPs.
\end{lemma}

\begin{proof} If $i\in [k]$, $j\in [l]$, and $x\in P(k,l)$:
\begin{align*}
\phi\circ t_{i,j}(x)&=\phi\circ t_{i,j}((1,j)^2\cdot x\cdot (1,i)^2)\\
&=\phi((1,j)\cdot t_{1,1}((1,j)\cdot x\cdot (1,i))\cdot (1,i))\\
&=(1,j)\cdot  \phi\circ t_{1,1}((1,j)\cdot x\cdot (1,i))\cdot (1,i)\\
&=(1,j)\cdot  t_{1,1}\circ \phi ((1,j)\cdot x\cdot (1,i))\cdot (1,i)\\
&=t_{i,j}((1,j)\cdot \phi((1,j)\cdot x\cdot (1,i))\cdot (1,i))\\
&=t_{i,j}\circ \phi(x),
\end{align*}
with the convention $(1,1)=\mathrm{Id}$. So $\Phi$ is a morphism of TRAPs. \end{proof}

In particular, to show that a collection of linear maps between two TRAPs preserving 
the horizontal concatenation and the actions of the symmetry group is a morphism of TRAPs, it is enough to check 
the properties of Lemma \ref{lemmemorphismes}.

\subsection{Quotient of TRAPs} \label{subsec:taking_quotients}

{This paragraph prepares for the construction of an  embedding of   a TRAP $P$ in a unitary TRAP.}
 \begin{lemma} \label{lem:relation_respection_TRAP}
 Let $Q$ be a TRAP and let $\sim $ be an equivalence relation on  $Q$ which is compatible with the TRAP-structure  on $Q$ in the following sense.
 	 For any two elements $x,x'\in Q$, such that $x\sim x'$: 
 
 	\begin{itemize}
 		\item (Compatibility with the module-structure). For any $(\sigma,\tau)\in \sym_k\times \sym_l$,
 		$\tau\cdot x\cdot \sigma\sim \tau \cdot x'\cdot \sigma$.
 		\item (Compatibility with the horizontal concatenation). For any  $y\in Q$, $x*y\sim x'*y$ and $y*x\sim y*x'$.
 		\item (Compatibility with the partial trace maps). For any $(i,j)\in [k]\times [l]$, $t_{i,j}(x)\sim t_{i,j}(x')$. 
 		 	\end{itemize}
 Then the quotient $Q/\sim$ is a TRAP with $[I_0]$ as unit for the concatenation product. If $Q$ is unitary with unit $I\in Q(1,1)$ for the partial trace maps, then $Q/\sim$ is also unitary with $[I]$ as unit for the partial trace maps.
 \end{lemma}
 \begin{proof}
  
  1. By compatibility with the module structure, $Q/\sim$ is a $\sym\times\sym^{op}$-modules.
  
  2. Using twice the compatibility with partial trace  maps, we find that if $x\sim x'$ and $y\sim y'$, then $x*y \sim x'*y'$
   by transitivity of $\sim$. Thus the horizontal concatenation $[x]*[y]:=[x*y]$ on $Q/\sim$ is well-defined. It fulfils properties 2.(a) to 2.(d) of Definition \ref{defi:Trap} by construction.
  
  3. We defined the partial trace maps on the quotient to be $t_{i,j}([x]):=[t_{i,j}(x)]$. It is well defined by compatibility with the partial trace maps  and has properties 3.(a) to 3.(c) of Definition \ref{defi:Trap} by construction.
  
  Thus $Q/\sim$ is a TRAP. Finally, if $Q$ is unitary with unit $I$, then $[I]$ endows the quotient $Q/\sim$ with a unit by construction and $Q/\sim$ is then a unitary TRAP.
 \end{proof}
The following statement is a direct consequence.
\begin{prop} \label{prop:quotient_from_TRAP_morph}
 Let  $P$ be a TRAP and  $\Phi:Q\to P$ a TRAP-morphism. The relation
\[x\sim x'\Longleftrightarrow  \Phi(x)=\Phi(x')\]
defines an equivalence relation compatible with the TRAP-structure on $Q$ and $Q/\sim$ defines a TRAP.
\end{prop} 
\begin{proof}
 Let $x$, $x'$ in $Q$, such that $x\sim x'$.
 \begin{itemize}
  \item For any $(\sigma,\tau)\in \sym_k\times \sym_l$,
  \begin{equation*}
   \Phi(\sigma\cdot x\cdot\tau) = \sigma\cdot \Phi(x)\cdot \tau = \sigma\cdot \Phi(x')\cdot \tau = \Phi(\sigma\cdot x\cdot \tau),
  \end{equation*}
  where the first and last identities follow from the fact that $\Phi$  is a morphism of $\sym\times\sym^{op}$ -modules.
  Thus $\sym$ is compatible with the module structure.
  \item For any $y$ in $Q$ we have
  \begin{equation*}
   \Phi(x*y)=\Phi(x)*\Phi(y)=\Phi(x')*\Phi(y)=\Phi(x'*y)
  \end{equation*}
  where we have used the fact that $\Phi$ is a morphism for the horizontal concatenation product (since $\Phi$ is a morphism of TRAPs) and the fact that $\Phi(x)=\Phi(x')$. Thus $x*y\sim x'*y$. Similarly we show that $y*x\sim y'*x'$ and $\sim$ is compatible with the horizontal concatenation.
  \item For any $(i,j)\in[k]\times[l]$ we have 
  \begin{equation*}
   \Phi(t_{i,j}(x)) = t_{i,j}(\Phi(x)) = t_{i,j}(\Phi(x'))= \Phi(t_{i,j}(x')),
  \end{equation*}
 where the first and last identities follow from the fact that $\Phi$  is a morphism of TRAPs.
 
  Thus $\sim$ is compatible with the partial trace maps.\qedhere
 \end{itemize}
\end{proof}

\section{Fundamental examples} \label{section:Examples}

\subsection{The Hom TRAP} \label{subsection:HomTRAP}

Let us give a fundamental example of unitary TRAP:

\begin{example}\label{ex221}
	Let $V$ be a finite dimensional vector space and $V^*$ its algebraic dual. We consider the family \[ \Hom_V=\left(\Hom_V(k, l)\right)_{k, l\in \N_0}:= { (\Hom(V^{\otimes k},V^{\otimes l}))_{(k,l)\in \N_0^2}},\] where for any $(k,l)\in  \N_0^2$,  {$\Hom(V^{\otimes k},V^{\otimes l})$ is  the vector space of linear maps from $V^{\otimes k}$ to $V^{\otimes l}$}. 
	 {We shall} identify $\Hom_V(k,l)$  and $  {V^{*\otimes k}\otimes V^{\otimes l}}$ through the isomorphism
	\[\theta_{k,l}:\left\{\begin{array}{rcl}
	{V^{*\otimes k}\otimes V^{\otimes l}}&\longrightarrow&\Hom_V(k,l)\\
	 { f_1\ldots f_k\otimes v_1\ldots v_l}&\longmapsto&\left\{\begin{array}{rcl}
	V^{\otimes k}&\longrightarrow&V^{\otimes l}\\
	x_1\ldots x_k&\longmapsto&f_1(x_1)\ldots f_k(x_k)\,v_1\ldots v_l,
	\end{array}\right.
	\end{array}\right.\]
 {	where with  some abuse of notation, we have set} $f_1\cdots f_k:= f_1\otimes\cdots \otimes f_k\in V^{*\otimes k}$ and 
	$v_1\cdots v_l:= v_1\otimes \cdots \otimes v_l\in V^{\otimes l}$. For any vector space $W$, the tensor power $W^{\otimes k}$ is a left $\sym_k$-module
	with the action defined by
	\[\sigma \cdot w_1\ldots w_k=w_{\sigma^{-1}(1)}\ldots w_{\sigma^{-1}(k)}.\]
	Via the identification $\theta:=\left(\theta_{k, l}\right)_{(k, l)\in \N_0^2}$,  we equip the family $\Hom_V=(\Hom(V^{\otimes k},V^{\otimes l}))_{(k,l)\in \N_0^2}$ with a structure
	of $\sym_l\times \sym_k^{op}$ by putting, for an $f\in \Hom(V^{\otimes k},V^{\otimes l})$,
	for any $(\sigma,\tau)\in \sym_k\times \sym_l$:
	\begin{align}
	&\forall v_1\ldots v_k\in V^{\otimes l},&
	\tau\cdot f\cdot \sigma(v_1\ldots v_k)=\tau\cdot f(\sigma\cdot v_1\ldots v_k). \label{EQaction}
	\end{align}
	The horizontal concatenation is the usual tensor product of linear maps:
	if $f\in \Hom_V(k,l)$ and $g\in \Hom_V(k',l')$, then
	\[f\otimes g:\left\{\begin{array}{rcl}
	V^{\otimes (k+k')}&\longrightarrow&V^{\otimes (l+l')}\\
	v_1\ldots v_{k+k'}&\longmapsto&f(v_1\ldots v_k)\otimes g(v_{k+1}\ldots v_{k+k'}).
	\end{array}\right.\]
	We define  the following partial trace maps:
	\begin{align}
	\label{EQtrace} t_{i,j}(\theta_{k,l}(f_1\ldots f_k\otimes v_1\ldots v_l))&=f_i(v_j)\,\theta_{k-1,l-1}(f_1\ldots f_{i-1}f_{i+1}\ldots f_k
	\otimes v_1\ldots v_{j-1}v_{j+1}\ldots v_l).
	\end{align}
\end{example}
\begin{remark}
     Notice that for $p\in\Hom_V(1,1)$, $t_{1,1}(p)$ coincides with the usual trace of a linear map on $V$.  Proposition \ref{prop:generalised_traces}  generalises the notion of trace to any element of $\Hom_V$ and to any TRAP.
    \end{remark}

\begin{prop}\label{prop:HomV} 
	For a finite dimensional vector space $V$, the above construction equips $\Hom_V$ with  a TRAP structure which is unitary, with unit given by the identity of $V$.  In other words,   $\Hom_V$  is a wheeled PROP.
\end{prop}

\begin{proof}
	Properties 2.(a)-(d) are trivially satisfied, with $I_0=1\in \K$.
	Property 3.(a) is direct.  Let us prove Property 3.(b). 
	\begin{align*}
	t_{i,j}(\sigma\cdot\theta_{k,l}(f_1\ldots f_k\otimes v_1\ldots v_l) \cdot \tau)
	&=t_{i,j}\circ \theta_{k,l}(f_{\tau(1)}\ldots f_{\tau(k)}\otimes v_{\sigma^{-1}(1)}\ldots v_{\sigma^{-1}(l)})\\
	&=f_{\tau(i)}(v_{\sigma^{-1}(j)})
	\theta_{k-1,l-1}(f_{\tau(1)}\ldots f_{\tau(i-1)} f_{\tau(i+1)}\ldots f_{\tau(k)}\\
	&\otimes v_{\sigma^{-1}(1)}\ldots  v_{\sigma^{-1}(j-1)} v_{\sigma^{-1}(j+1)}\ldots v_{\sigma^{-1}(l)})\\
	&=\sigma_j\cdot t_{\tau(i),\sigma^{-1}(j)}\theta_{k,l}(f_1\ldots f_k\otimes v_1\ldots v_l)\cdot \tau_i.
	\end{align*}
	Property 3.(c) is straightforward.  Let us prove that $\Hom_V$ is unitary with the help of Lemma \ref{lemmeaxiomessimples}. Let us fix $(e_i)_{i\in I}$ a basis of $V$, then $(e_i^*)_{i\in I}$ is a basis of $V^*$ and  the identity map of $V$ is 
	\begin{equation} \label{eq:identitymap}
	\displaystyle \mathrm{Id}_V=\theta_{1, 1}\left(\sum_{i\in I} e_i \otimes e_i^*\right).\end{equation}
	Then for any $p=\theta_{k,l}(f_1\ldots f_k\otimes v_1\ldots v_l)\in\Hom_V(k,l)$
	\begin{align*}
	t_{1,2}(\mathrm{Id}_V*\theta_{k,l}(f_1\ldots f_k\otimes v_1\ldots v_l))
	&=\sum_{i\in I} t_{1,2}\circ \theta_{k+1,l+1}(e_i^*f_1\ldots f_k\otimes e_iv_1\ldots v_l)\\
	&=\sum_{i\in I} \theta_{k,l}(f_1\ldots f_k\otimes e_i\, e_i^*(v_1)v_2\ldots v_l)\\
	&= \theta_{k,l}(f_1\ldots f_k\otimes \mathrm{Id}_V(v_1)v_2\ldots v_l)\\
	&=\theta_{k,l}(f_1\ldots f_k\otimes v_1\ldots v_l). 
	\end{align*}
	So $\Hom_V$ is a unitary TRAP.
\end{proof}

When $V$ is not finite-dimensional, $\theta$ is an injective, non surjective map. Its range is the subpsace $\Hom_V^{fr}$ of linear
maps from $V^{\otimes k}$ to $V^{\otimes l}$ of finite rank. We can equip $\Hom_V^{fr}$ with a similar TRAP structure:

\begin{prop}\label{prop:Homfr}
	With the $\sym\times \sym^{op}$ action defined  by (\ref{EQaction}), the usual tensor product of maps
	and the partial trace maps defined by (\ref{EQtrace}), $\Hom_V^{fr}$ is a TRAP. It is unitary if, and only if, $V$
	is finite-dimensional.
\end{prop}

\begin{proof} {We skip the proof that $\Hom_V^{fr} $ can be equipped with a TRAP structure since it goes as for $\Hom_V$  when $V$ is finite dimensional. Note that when $V$ is finite-dimensional, then $\Hom_V^{fr}=\Hom_V$ is a unitary TRAP. We show the second part of the statement.}
	 
	Let us assume that $\Hom_V^{fr}$ has a unit $I$. Then $I$ has finite rank, let us fix a basis $(e_1,\ldots,e_k)$ of $\mathrm{Im}({I})$. {There exist  $\lambda_1,\ldots,\lambda_k\in V^*$ such that for any $v\in V$,
	\[I(v)=\sum_{i=1}^k \lambda_i(v)e_i.\]
	In other words, }
		\[I=\theta_{1,1}\left({\sum_{i=1}^k e_i\otimes \lambda_i}\right).\]
	Let $v\in V$, nonzero, and let $\lambda \in V^*$ such that $\lambda(v)=1$. We consider $f=\theta_{1,1}(v\otimes \lambda)$.
	Then $f(v)=\lambda(v)v=v$. Moreover:
	\begin{align*}
	v&=f(v)=t_{1,2}(I*f)(v)\\
	&=t_{1,2}\circ \theta_{2,2}\left(\sum_{i=1}^k {e_i}\,  v\otimes \lambda_i \lambda\right)(v)\\
	&=\theta_{1,1}\left(\sum_{i=1}^k { \lambda_i}{ (v)\, e_i} \otimes {\lambda}\right)(v)\\
	&=\sum_{i=1}^k \lambda(v)\lambda_i(v)\, {e_i} \\
	&=\sum_{i=1}^k \lambda_i(v)\, {e_i} ,
	\end{align*}
	so $v\in \mathrm{Vect}(e_1,\ldots,e_k)$. Hence, $V\subseteq \mathrm{Im}(I)$, so $V$ is finite-dimensional. 
\end{proof}
 We end this paragraph with an example of a TRAP similar to Example \ref{ex221} but of a more geometric nature.
\begin{example} [The TRAP of tensors]
		Given a finite dimensional smooth  manifold $M$ and a point $x\in M$, we build the $\Hom$-TRAP $\Hom_{T_xM}$ where $T_xM$  is the tangent space to $M$ at the point $x$.  
		Given a pair $(p, q)\in \Z_{\geq 0}^2$ we have 
		\[ \Hom_{T_xM}(p, q)\simeq (T_x^*M)^{\otimes p}\otimes T_xM^{\otimes q},\]
		where we have set   $V^{\otimes 0}=\R$. The partial trace maps   are built by pairing cotangent  and tangent vectors. 
		We note that if $M$ is equipped with a Riemannian metric, thanks to  the  musical isomorphisms  $T_x^*M\ni \alpha  \longmapsto \alpha^\sharp \in T_x M$ and  $T_xM\ni v\longmapsto v^\flat \in  T_x^*M$ between  $T_x^*M$ and $ T_xM$, these dual pairings can be seen as contractions via the metric tensor.
		
		This yields a smooth fibration  $ \Hom_{TM}:=\{ \Hom_{T_xM}, x\in M\} $ of TRAPs parametrised by $M$. For any $(p, q)\in \Z_{\geq 0}^2$,  a smooth section of $\Hom_{TM}(p,q)$ defines a smooth $(p,q)$ tensor on $M$.   
\end{example}

\subsection{The TRAP of continuous morphisms} \label{subsection:Frechet_nuc} 

We generalise the constructions of the previous paragraph, replacing  the finite dimensional spaces $V^{\otimes k}$ in $ \Hom_V$ by  nuclear spaces. 
These nuclear spaces were defined in the seminal work \cite{Gr54}. Most of the results stated here can  be found in \cite{Gr52,Gr54}. We also 
refer to the more recent presentation \cite{Treves67}.

We recall that: \begin{itemize}
	\item A topological vector space is \textbf{Fr\'echet} if it is Hausdorff, has its topology induced by a  {countable} family of semi-norms and is complete with respect 
	to this family of semi-norms. 
	\item  
 {	The topological dual $E'$ of  a locally convex topological vector space $E$  can be endowed with various topologies, one of which is 
	the \textbf{strong topology}, namely the topology of uniform convergence on the bounded domains of $E$. It is  generated by the family of semi-norms of $E'$ defined on any $f\in E'$ by $||f||_B:=\sup_{x\in B}|f(x)|$
	for any bounded set $B$ of $E$.} The topological dual $E'$ endowed with this topology is called the \textbf{strong dual}.
	\item A topological vector space is called \textbf{reflexive} if $E''=(E')'=E$, where $E'$ is the   topological  dual of $E$ {endowed with the strong topology}.
\end{itemize}
In the following $E$ and $F$ are two topological vector spaces and $\Hom^c(E,F)$ is the set of continuous  linear maps from $E$ to $F$.
\begin{remark}
\begin{itemize}
	\item 	When $E$ and $F$ are finite dimensional, we have $\Hom^c(E,F)$=$\Hom(E,F)$.
	\item {As pointed out to us by Mark Johnson, a natural generalisation to consider in the context of Fr\'echet spaces are  $\sigma$ $C^*$-algebras, defined as inverse limits of $C^*$-algebras \cite{Phillips}, which however lie out of the scope of the present article.}
\end{itemize} 
\end{remark}
In order to build the TRAP $\Hom_V^c$ for nuclear spaces, we need Grothendieck's   completion of the tensor product, a notion we recall 
here in the set-up of locally convex topological $\K$-vector spaces.

Let $E$ and $F$ be two vector spaces. Recall that there exists a   vector space $E\otimes F$, and a bilinear map 
$\phi:E\times F\longrightarrow E\otimes F$ such that for any vector space $V$ and bilinear map $f:E\times F\longrightarrow V$, there is a unique 
linear map $\tilde f:E\otimes F\to V$ satisfying  $f=\tilde f\circ \phi$. The space $E\otimes F$ is unique modulo isomorphism and is called the \textbf{tensor product} of 
$E$ and $F$.

Given two     topological vector spaces $E$ and $F$, one can a priori equip   $E\otimes F$ with several topologies, among which the 
\textbf{$\epsilon$-topology} and the \textbf{projective topology}  whose construction are recalled in Appendix \ref{section:topologies_tens_prod}. We 
denote by $E\otimes_\epsilon F$ (resp. $E\otimes_\pi F$) the space $E\otimes F$ endowed with 
the $\epsilon$-topology (resp. the projective topology)  and by
$E\widehat\otimes _\epsilon F$ (resp. $E\widehat\otimes _\epsilon F$) of $E\otimes_\epsilon F$ (resp. $E\otimes_\epsilon F$) their completion with respect to the 
$\epsilon$-topology (resp. projective topology). These two spaces differ in general but coincide for nuclear spaces.
\begin{defi} \label{defi:nuclear_spaces} \cite{Gr54}
	A locally convex topological vector space $E$ is \textbf{nuclear} if, and only if, for any locally convex topological vector space $F$,
	\begin{equation*}
	E\widehat\otimes _\epsilon F = E\widehat\otimes _\pi F =: E\widehat\otimes  F
	\end{equation*}
	holds, in which case $E\widehat\otimes  F $ is called the \textbf{completed tensor product} of $E$ and $F$.
\end{defi} 
There are other equivalent definitions of nuclearity, see for example \cite{gelfand1964,hida2008}. 
{\begin{remark} It was pointed out to us by Mark Johnson that such minimal and maximal tensor products, much used in the context of $C^*$-algebras, further extend to l.m.c. $C^*$-algebras, where l.m.c stands for locally multiplicative convex (see \cite{JoachimJohnson} and references therein).   
\end{remark}} 
For Fr\'echet spaces, nuclearity is preserved under strong duality.
\begin{prop} \begin{itemize}
		\item 
		\cite[Proposition 50.6]{Treves67} \label{prop:Frechet_nuclear}
		A Fr\'echet space is nuclear if and only if its strong dual is nuclear.
		\item 
		\cite[Proposition 36.5]{Treves67} A Fr\'echet nuclear space is reflexive.
	\end{itemize}
\end{prop}
Many spaces relevant to renormalisation issues are Fr\'echet and nuclear. We list here some examples.
\begin{example}\label{ex:findimtensor1}  
	Any finite dimensional vector space  can be equipped with a norm and for any of these norms, they are trivially  Banach, hence Fr\'echet and nuclear. 
	If $E$ and $F$ are finite dimensional vector spaces  we have
	$\Hom^c (E, F)=  \Hom (E, F)\simeq  E^* \otimes F$, where $  \Hom (E, F)$ stands for the space of $F$-valued linear maps on 
	$E$ and where the dual $E^*$ is the { \bf algebraic dual}.
\end{example} 
\begin{example}\label{ex:infindimtensor1} 
	Let $U$  be an open subset of $\R^n$.
	Take $E= C^\infty(U)=:{\mathcal E}(U)$. The topological dual 
	is the space ${E'}={\mathcal E}^\prime(U)$ of distributions 
	on $U$ with compact support. 
		Then $E$ is Fr\'echet (\cite{Treves67}, pp. 86-89), and $E'$ is nuclear (\cite{Treves67}, Corollary p. 530). By Proposition 
	\ref{prop:Frechet_nuclear}, $E$ is also nuclear. 
\end{example}
\begin{remark}\label{rk:dualnotfrechet}  Note that the dual $E' $ of a Fr\'echet space  $E$ is  never a Fr\'echet space (for any of the natural topologies on $E'$), unless $E$  is actually a Banach space (see for example 
	\cite{kothe1969}).   In particular, ${\mathcal E}^\prime(U)$ is generally not Fr\'echet.
\end{remark} 
We now sum up various results of \cite{Treves67} of importance for  later purposes.
\begin{prop}\cite[Equations (50.17)--(50.19)]{Treves67}
	Let $E$ and $F$ be two Fr\'echet spaces, with $E$ nuclear. The following isomorphisms of topological vector spaces hold.
	\begin{align}
	& E'\widehat\otimes  F \simeq \Hom^c(E,F) \label{eq:E_prime_otimes_F} \\
	& E\widehat\otimes  F \simeq \Hom^c(E',F) \label{eq:E_otimes_F} \\
	& E'\widehat\otimes  F' \simeq (E\widehat\otimes  F)' \simeq {\mathcal B}^c(E\times F, \K). \label{eq:E_prime_otimes_F_prime}
	\end{align}
	with ${\mathcal B}^c(E\times F, \K)$ the set of continuous bilinear maps 
	$K:E\times F\longrightarrow\K$. Here the duals are endowed with the strong dual topology, 
 {	$\Hom^c(E,F)\simeq E'\otimes F$ } with the  topology of uniform convergence on the bounded subsets of $E$\footnote{ {It is defined by a family of  semi-norms $p_{B,\pi_i}$ of the form $ p_{B,\pi_i}(f)=\sup_{x\in B} \pi_i(f(x))$ when applied to some $f$ in ${\rm Hom}^c(E,F)$, where $B$ runs through the sets of all bounded subsets of $E$ and $\pi_i$ runs through a countable family of semi-norms which generate the topology of $F$. It gives back the strong topology on $E'$ if $F$ is the underlying field.  It also carries other names such as  "topology of bounded convergence" \cite[page III.14]{Bourbaki},  \cite[p.81]{Schaefer}}} and ${\mathcal B}^c(E\times F, \K)$ with the topology of uniform convergence on products of bounded sets.
\end{prop}
The stability of Fr\'echet nuclear spaces under completed tensor products,  {It follows from combining and   {the definition of the completed tensor product with the fact that if $E$ and $F$} are two nuclear spaces then $E\widehat\otimes  F$ is a nuclear space (\cite[Equation (50.9)]{Treves67}). A stronger version of this Lemma is mentioned in \cite[\S 6.f, page 182]{BB03} which quotes \cite{Gr54}.}
\begin{lemma} \label{lem:prod_Frechet_nuclear}
	The completed tensor product  $E\widehat\otimes  F$ of two Fr\'echet nuclear spaces is a Fr\'echet nuclear space.
\end{lemma}
\begin{prop}
	Let $V$ be a Fr\'echet nuclear space. Then 
	\begin{equation} \label{eq:echange_dual_prod}
	\left(V^{ \widehat\otimes  k}\right)' \simeq\left(V'\right)^{ \widehat\otimes  k}
	\end{equation} 
	holds for any $k\geq1$, where the duals are endowed with their strong topologies.
\end{prop}
\begin{proof}
	Let $V$ be a Fr\'echet nuclear space. The case $k=1$ is trivial. Then Equation \eqref{eq:echange_dual_prod} with $k=2$ holds by Equation \eqref{eq:E_prime_otimes_F_prime} 
	with $E=F=V$. The cases $k\geq2$ are proved by induction, using $E=V^{\widehat\otimes  k-1}$ and $F=V$. The induction holds by Lemma 
	\ref{lem:prod_Frechet_nuclear}.
\end{proof}
We  denote by $\mathcal{D}'({M})$  the set of distributions on ${X}$  and by $\mathcal{E}'(M)$, the set of distributions with 
compact support on a  finite dimensional smooth manifold $M$, see e.g. \cite[Definition 6.3.3]{Ho89}.
It is well-known   (see e.g.  \cite[Exercise 2.3.2]{BaCr13}, \cite[p. 4]{BrDaLGRe17}) that $\mathcal{E}(M)$ is a Fr\'echet nuclear space.
It then follows from Proposition \ref{prop:Frechet_nuclear}, that the space $\mathcal{E}'(M)$ is   also nuclear. 
\begin{remark} (Compare with Remark \ref{rk:dualnotfrechet}). Note that the 
	space $\mathcal{E}'(M)$ is \emph{not} Fr\'echet  since the dual of a Fr\'echet space $F$ is Fr\'echet if and only if $F$ is Banach (see for example 
	\cite{kothe1969}) which is not the 
	case of $\mathcal{E}(M)$.
\end{remark}
One further useful result is:
\begin{prop} \label{prop:prod_function}
	Let $M$ and ${N}$ be two finite dimensional smooth manifolds. Then 
	\begin{equation*}
	\Hom^c(\mathcal{E}'(M),\mathcal{E}(N))\simeq \,  \mathcal{E}(M)\,\widehat\otimes\, \mathcal{E}(N) \simeq \mathcal{E}(M\times N)
	\end{equation*}
	holds.
\end{prop}
The second isomorphism   \cite[Chap. 5, p. 105]{Gr52} can be proved using a version of the Schwartz kernel theorem for smoothing operators 
\cite[Theorem 2.4.5]{BaCr13} by means of the identification $\Hom^c(\mathcal{E}'(M),\mathcal{E}(N))\simeq \mathcal{E}(M\times N)$. The result then follows from 
\eqref{eq:E_otimes_F} applied to $\mathcal{E}(X)$ and $\mathcal{E}(Y)$ which are Fr\'echet nuclear spaces. 
\begin{defi} \label{defi:Hom_V_generalised}
	Let $V$ be a Fr\'echet nuclear space. For any $k,l\in  \N_0$, we set
	\[\Hom_V^c(k,l)=\Hom^c(V^{\hat \otimes k}, V^{\hat \otimes l})\simeq(V')^{\widehat\otimes  k}\widehat\otimes  V^{\widehat\otimes  l},\]
	where, as before $V'$ stands for the strong topological dual. Furthermore we set $\Hom^c_V:=(\Hom_V^c(k,l))_{k,l\geq0}$.
	 
	For any $\sigma \in \sym_n$, let  $\Theta_\sigma$ be the endomorphism of $V^{\otimes n}$ defined by
	\[\Theta_\sigma(v_1\otimes \ldots \otimes v_n)=v_{\sigma^{-1}(1)}\otimes \ldots \otimes v_{\sigma^{-1}(n)}.\] It extends to a continuous linear map 
	$\overline{\Theta_\sigma}$ on the closure 
	$V^{\widehat\otimes  n}$.   
	For any $f\in \Hom_V^c (k,l)$, $\sigma \in \sym_l$, $\tau\in \sym_k$, we set:
	\begin{align*}
	\sigma\cdot f&=\overline{\Theta_\sigma} \circ f ,&f\cdot \tau&=f\circ \overline{\Theta_\tau}.
	\end{align*} 
\end{defi}
In the above definition, the superscript ``c'' stands for continuous. The family $\Hom_V^c$ carries a TRAP structure.

\begin{theo} \label{thm:Hom_V_generalised}
	Let $V$ be a Fr\'echet nuclear space. $\Hom_V^c$, with the action of $\sym\times\sym^\mathrm{op}$ described above, is a TRAP. Its horizontal 
	concatenation is the usual (topological) tensor product of maps with $I_0:\K\longrightarrow\K$ given by the 
{identity map of $\K$}, its partial trace  maps  coincide with those of  the TRAP $\Hom_V$ on elements of $(V')^{\widehat\otimes  k}\widehat\otimes  V^{\widehat\otimes  l}$
\begin{equation*}
 t_{i,j}(f_1\cdots f_k\otimes v_1\cdots v_l) = f_i(v_j) f_1\cdots f_{i-1}f_{i+1}\cdots f_k\otimes v_1\cdots v_{j-1}v_{j+1}\cdots v_l
\end{equation*}
(with the same notations than the ones of Subsection \ref{subsection:HomTRAP}).
It is unitary if and only if $V$ is finite dimensional, in which case $I_1:V\longrightarrow V$ is the identity map of $V$.
\end{theo}

\begin{proof}
	The proof of the TRAP structure of $\Hom^c_V$ goes as in Proposition \ref{prop:HomV}. The proof of the unital case is the same than the proof of Proposition \ref{prop:Homfr}.
\end{proof}
\begin{example}
	For a  finite dimensional vector space $V$, the TRAP $\Hom_V^c$  coincides with the  TRAP  $\Hom_V$.
\end{example} 
\begin{example}
	Let $M$ be a smooth finite dimensional manifold.
	From Proposition \ref{prop:prod_function} and Equation \eqref{eq:echange_dual_prod}, it follows that
	the family  $({\mathcal K}_M(k, l))_{k,l\in \N_0}$,
	with ${\mathcal K}_M(k, l)= \left({\mathcal E}^\prime(M)\right)^{\widehat\otimes  k}\, \widehat\otimes  \,  {\mathcal E}(M)^{\widehat\otimes  l}$  
	defines a TRAP.
\end{example}

\subsection{The TRAP $\mathcal{K}_{M}^\infty$ of smoothing pseudo-differential operators}\label{sec:smoothingop}

We apply our results on TRAPs to  tensor products 
of  Fr\'echet spaces ${\mathcal E}(M)$  of smooth functions on a given smooth finite dimensional  {orientable} closed manifold $M$ and  { $ \mu(z)$ a volume form on $M$.}  Recall from Proposition 
\ref{prop:prod_function}   that such spaces are stable under tensor products and morphisms in  $\mathrm{Hom}^c( {\mathcal E}^\prime(M), {\mathcal E}(N))$ are determined by    smoothing kernels in ${\mathcal E}(M\times N)$.

  We consider smooth kernels
which stabilise ${\mathcal E}(M)$ and set, for $(k,l)\neq(0,0)$:
\begin{equation}\label{eq:Klm}{\mathcal K}_{M}^\infty(k, l):={\mathcal E}({M}^k\times {M}^l)\simeq{\mathcal E}({M})^{\widehat\otimes  k}\, \hat \otimes \,  {\mathcal E}({M})^{\widehat\otimes  l},\end{equation}
whose elements we refer to as smooth generalised kernels.
\begin{theo}\label{theo:Kinfty}
	The family of topological vector spaces $\left({\mathcal K}_{M}^\infty(k, l)\right)_{k,l\in \N_0}$ can be equipped with the partial trace  maps 
	\begin{align*}
	t_{i,j}&:\left\{\begin{array}{rcl}
	{\mathcal K}_{M}^\infty(k, l)&\longrightarrow&{\mathcal K}_{M}^\infty(k-1,l-1)\\
	K_1\otimes K_2 &\longmapsto&  {t_{i,j}(K_1\otimes K_2)}
	\end{array}\right.
	\end{align*}
	with
	$t_{i,j}(K_1\otimes K_2)$ defined by
	\begin{align}\label{eq:trconvK}
	t_{i,j}(K_1\otimes K_2) & (x_1, \cdots, x_{k-1}, y_1, \cdots,  y_{l-1}):= \\
	& \int_M K_1(x_1 ,  \cdots,x_{i-1},z,x_{i}\cdots, x_{k-1})\, K_2(y_1,\cdots,y_{j-1},z,y_{j}\cdots, y_k) \,d\mu(z).\nonumber
	\end{align}

Together with the  horizontal concatenation given by   the tensor product of maps $(K_1\otimes K_2)*(K'_1\otimes K'_2)=K_a\otimes K_b$ with $K_a:=K_1\otimes K'_1$ and $K_b:=K_2\otimes K'_2$
this	defines a TRAP, written $\mathcal{K}_{M}^\infty$,  which we call the TRAP of {\bf  generalised smooth kernels} on $M$. 
\end{theo}  
\begin{remark} Note that the partial trace    amounts to what one could call a partial convolution.
		\end{remark}
\begin{proof}
{ The unit $I_0\in\mathcal{K}_{M}^\infty(0,0)\simeq \C\otimes\C$ of the horizontal concatenation $*$ is the constant map {$f:\C\to \C$} defined by 
	$f(x)=1$. It is the unit of $*$ by bilinearity of the tensor product.{The} horizontal concatenation on the $\sym\times \sym^{op}$-module   $\left({\mathcal K}_{M}^\infty(k, l)\right)_{k,l\in \N_0}$    satisfies axioms 2. (a)-(d) of  Definition \ref{defi:Trap}. We want to check that the maps $t_{i,j}$ are well-defined and  satisfy axioms 3. (a)-(c).
}

	The existence of the integral follows from the smoothness of $K_1$ and $K_2$ and the 
	closedness of $M$. It is enough to show that the function 
	$t_{i,j}(K_1\otimes K_2):M^{k-1}\times M^{l-1}\longrightarrow \C$ is smooth. Since
	$K_1$ and $K_2$ are smooth, the map 
	\begin{equation*}
	(x_1 ,\cdots, x_{k-1},y_1,\cdots, y_k) \longmapsto K_1(x_1 ,  \cdots,x_{i-1},z,x_{i}\cdots, x_{k-1})\, K_2(y_1,\cdots,y_{j-1},z,y_{j}\cdots, y_k)
	\end{equation*}
	is infinitely differentiable for any $z\in M$. Since $M$ is compact, the partial derivatives
	\begin{equation*}
	\partial_{\vec x}^{\vec\alpha}\partial_{\vec y}^{\vec \beta}K_1(x_1 ,  \cdots,x_{i-1},z,x_{i}\cdots, x_{k-1})\, K_2(y_1,\cdots,y_{j-1},z,y_{j}\cdots, y_k)
	\end{equation*}
	are bounded uniformly in $z$ and hence
	\begin{align*}
	& \int_M \partial_{\vec x}^{\vec\alpha}\partial_{\vec y}^{\vec \beta}K_1(x_1 ,  \cdots,x_{i-1},z,x_{i}\cdots, x_{k-1})\, K_2(y_1,\cdots,y_{j-1},z,y_{j}\cdots, y_k) \,d\mu(z) \\
	= & \partial_{\vec x}^{\vec\alpha}\partial_{\vec y}^{\vec \beta}\int_MK_1(x_1 ,  \cdots,x_{i-1},z,x_{i}\cdots, x_{k-1})\, K_2(y_1,\cdots,y_{j-1},z,y_{j}\cdots, y_k) \,d\mu(z) \\
	= & \partial_{\vec x}^{\vec\alpha}\partial_{\vec y}^{\vec \beta} t_{i,j}(K_1\otimes K_2) (x_1, \cdots, x_{k-1}, y_1, \cdots,  y_{l-1}).
	\end{align*}
	Therefore the map $t_{i,j}(K_1\otimes K_2) (x_1, \cdots, x_{k-1}, y_1, \cdots,  y_{l-1})$ is smooth.

Finally, to check Axiom 3.(c), by the first item of Lemma \ref{lemmeaxiomessimples} it is enough to check the compatibility of the horizontal concatenation with the partial trace  to show that $t_{1,1}(p*p')=t_{1,1}(p)*p'$ for any pair $(p,p')\in\mathcal{K}_{M}^\infty(k,l)\times\mathcal{K}_{M}^\infty(k',l')$ with $k,k',l,l'\geq1$. Setting $p=K_1\otimes K_2$ and $p'=K'_1\otimes K'_2$ we have, by definition of the partial trace  maps and the horizontal concatenation
	\begin{align*}
	 & t_{1,1}(p*p')(x_1,\cdots,x_{k+k'-1},y_1,\cdots,y_{l+l'-1}) \\
	  = & \int K_1(z,x_1,\cdots,x_{k-1})K'_1(x_{k},\cdots,x_{k+k'-1})K_2(z,y_1,\cdots,y_{l-1})K'_2(x_{l},\cdots,x_{l+l'-1})dz \\
	  = & \left(\int K_1(z,x_1,\cdots,x_{k-1})K_2(z,y_1,\cdots,y_{l-1})dz\right)K'_1(x_{k},\cdots,x_{k+k'-1})K'_2(x_{l},\cdots,x_{l+l'-1}) \\
	  = & (t_{1,1}(p)*p')(x_1,\cdots,x_{k+k'-1},y_1,\cdots,y_{l+l'-1}).\qedhere
	\end{align*}
\end{proof}
{\begin{remark}
     Notice that $\mathcal{K}_{M}^\infty$ is non unitary, since the map $f:M\times M\longmapsto\C$ which could play the role of a vertical unity is a {$\delta$ distribution} supported on the diagonal. {The simple examples considered here}, namely $\mathcal{K}_{M}^\infty$ and the TRAP Hom$^c_V$ speak {for the fact that non-unitary TRAPs is an appropriate  framework to host infinite dimensional spaces. We  hope  non-unitary TRAPs to host more general distributions. }
    \end{remark}
}

\section{Free TRAPs}

\subsection{Various families of graphs}

 Here, we consider oriented multigraphs endowed with extra structures: indexed input and output edges, loops, ordering $\ldots$
These extra structures  {make it difficult to implement}  the usual definition of multigraphs \cite{Harary}, where the edges
{form} a multiset of pairs of vertices. In the literature on PROPs \cite{Markl,Merkulov2009}, 
graphs are defined as a set of half-edges (or flags), with an involution which  {tells us how} to glue them together
in order to obtain edges, and with a partition which defines the vertices. This definition does not take loops in account  
that is to say edges with no  {ends, yet loops enter our constructions in an}  essential way.  {Instead we borrow  a definition from the theory of} quivers \cite{Crawley,Derksen}, where two maps (called source and target,
or alternatively  tail and head) are given from the set of edges to the set of vertices.
 {In our approach, we  allow edges without source, or without target, or with neither source nor target}.  {Among these are the}  inputs and outputs 
in the graph we consider,  { which we then index}.

\begin{defi} \label{defi:graph}
A \textbf{graph} is a family $G=(V(G),E(G),I(G),O(G),IO(G),L(G),s,t,\alpha,\beta)$, where:
\begin{enumerate}
\item $V(G)$ (set of vertices), $E(G)$ (set of internal edges), $I(G)$ (set of input edges),
$O(G)$ (set of output edges), $IO(G)$ (set of input-output edges) and $L(G)$ (set of loops, that is to say
edges with no endpoints) are finite (possibly empty) sets.
\item $s:E(G)\sqcup O(G)\longrightarrow V(G)$ is a map (source map).
\item $t:E(G)\sqcup I(G)\longrightarrow V(G)$ is a map (target map).
\item $\alpha:I(G)\sqcup IO(G)\longrightarrow [i(G)]$ is a bijection, with $i(G)=|I(G)|+|IO(G)|$
(indexation of the input edges).  
\item $\beta:O(G)\sqcup IO(G)\longrightarrow [o(G)]$ is a bijection, with $o(G)=|O(G)|+|IO(G)|$
(indexation of the output edges).
\end{enumerate}
A \textbf{{corolla ordered} graph} is a graph $G$ such that for any vertex $v$, the set of incoming edges $I(v)$ of $v$
and the set of outgoing edges $O(v)$ of $v$ are totally ordered and we shall denote both order relations by $\leq_v$.

A graph $G$ is \textbf{solar}\footnote{The terminology solar refers to its radiating aspect with rays around a central body. {In \cite{JY15} such graphs are called ordinary.} } if $IO(G)=L(G)=\emptyset$. 
\end{defi}

\begin{example}\label{ex4}
Here is a graph $G$ : 
\begin{align*}
V(G)&=\{x,y\},&E(G)&=\{a,b\},&I(G)&=\{c,d\},&O(G)&=\{e,f\},&IO(G)&=\{g\}, &L(G)&=\{h,k\},
\end{align*}
and:
\begin{align*}
s&:\left\{\begin{array}{rcl}
a&\longmapsto&y\\
b&\longmapsto&x\\
e&\longmapsto&y\\
f&\longmapsto&y,
\end{array}\right.&
t&:\left\{\begin{array}{rcl}
a&\longmapsto&x\\
b&\longmapsto&y\\
c&\longmapsto&x\\
d&\longmapsto&x,
\end{array}\right.&
\alpha&:\left\{\begin{array}{rcl}
c&\longmapsto&1\\
d&\longmapsto&2\\
g&\longmapsto&3,
\end{array}\right.&
\beta&:\left\{\begin{array}{rcl}
e&\longmapsto&3\\
f&\longmapsto&1\\
g&\longmapsto&2.
\end{array}\right.&
\end{align*}
This is graphically represented as follows:
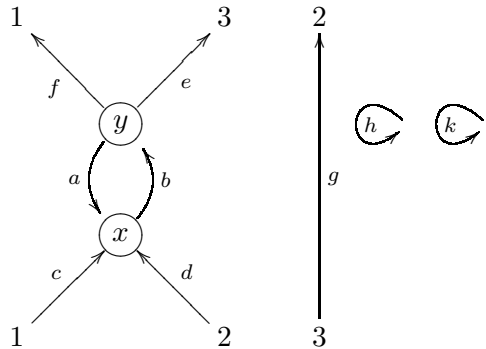
\[\xymatrix{1&&3&2&&\\
&\rond{y}\ar[ru]_e \ar[lu]^f \ar@/_1pc/[d]_a&&&\ar@(ul,dl)[]^h&\ar@(ul,dl)[]^k\\
&\rond{x}\ar@/_1pc/[u]_b&&&&\\
1\ar[ru]^c&&2\ar[lu]_d&3\ar[uuu]_g&&}\]
Note that this graph contains two loops, represented by $\xymatrix{&\ar@(ul,dl)[]^h}$
and $\xymatrix{&\ar@(ul,dl)[]^k}$.
\end{example}  

Graphically, if $G$ is a {corolla ordered} graph, we shall represent the orders on the incoming and outgoing edges of a vertex by drawing box-shaped vertices, with the incoming and outgoing edges of any vertex  ordered from left to right. For example, we distinguish the following two  situations:
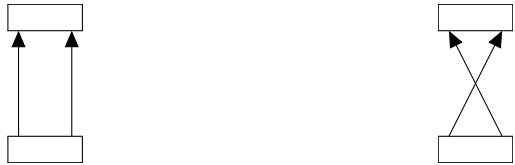
\begin{align*}
&\begin{tikzpicture}[line cap=round,line join=round,>=triangle 45,x=0.7cm,y=0.7cm]
\clip(0.8,-0.5) rectangle (2.2,2.5);
\draw [line width=.4pt] (0.8,0.)-- (2.2,0.);
\draw [line width=.4pt] (2.2,0.)-- (2.2,-0.5);
\draw [line width=.4pt] (2.2,-0.5)-- (0.8,-0.5);
\draw [line width=.4pt] (0.8,-0.5)-- (0.8,0.);
\draw [line width=.4pt] (0.8,2.)-- (2.2,2.);
\draw [line width=.4pt] (2.2,2.)-- (2.2,2.5);
\draw [line width=.4pt] (2.2,2.5)-- (0.8,2.5);
\draw [line width=.4pt] (0.8,2.5)-- (0.8,2.);
\draw [->,line width=.4pt] (1.,0.) -- (1.,2.);
\draw [->,line width=.4pt] (2.,0.) -- (2.,2.);
\end{tikzpicture}&
&\begin{tikzpicture}[line cap=round,line join=round,>=triangle 45,x=0.7cm,y=0.7cm]
\clip(0.8,-0.5) rectangle (2.2,2.5);
\draw [line width=.4pt] (0.8,0.)-- (2.2,0.);
\draw [line width=.4pt] (2.2,0.)-- (2.2,-0.5);
\draw [line width=.4pt] (2.2,-0.5)-- (0.8,-0.5);
\draw [line width=.4pt] (0.8,-0.5)-- (0.8,0.);
\draw [line width=.4pt] (0.8,2.)-- (2.2,2.);
\draw [line width=.4pt] (2.2,2.)-- (2.2,2.5);
\draw [line width=.4pt] (2.2,2.5)-- (0.8,2.5);
\draw [line width=.4pt] (0.8,2.5)-- (0.8,2.);
\draw [->,line width=.4pt] (1.,0.) -- (2.,2.);
\draw [->,line width=.4pt] (2.,0.) -- (1.,2.);
\end{tikzpicture} 
\end{align*}
We note that the graph of Example \ref{ex4} can be made  {corolla ordered}  in $3!\times 3!=36$ ways,
corresponding to the total orderings of the three incoming edges of $x$ and of the three outgoing edges of $y$.
Here are three of them:
\begin{align*}
\begin{tikzpicture}[line cap=round,line join=round,>=triangle 45,x=0.7cm,y=0.7cm]
\clip(0.9,-0.1) rectangle (9.5,4.7);
\draw [line width=0.4pt] (1.,2.)-- (3.,2.);
\draw [line width=0.4pt] (3.,2.)-- (3.,3.);
\draw [line width=0.4pt] (3.,3.)-- (1.,3.);
\draw [line width=0.4pt] (1.,3.)-- (1.,2.);
\draw [->,line width=0.4pt] (1.5,1.)-- (1.5,2.);
\draw [->,line width=0.4pt] (2.,1.)-- (2.,2.);
\draw [->,line width=0.4pt] (2.5,1.5)-- (2.5,2.);
\draw [shift={(3.75,1.5)},line width=0.4pt]  plot[domain=3.141592653589793:6.283185307179586,variable=\t]({1.*1.25*cos(\t r)+0.*1.25*sin(\t r)},{0.*1.25*cos(\t r)+1.*1.25*sin(\t r)});
\draw [line width=0.4pt] (5.,1.5)-- (5.,3.5);
\draw [line width=0.4pt] (2.,3.)-- (2.,3.5);
\draw [shift={(3.,3.5)},line width=0.4pt]  plot[domain=0.:3.141592653589793,variable=\t]({1.*1.*cos(\t r)+0.*1.*sin(\t r)},{0.*1.*cos(\t r)+1.*1.*sin(\t r)});
\draw [line width=0.4pt] (4.,3.5)-- (4.,1.5);
\draw [line width=0.4pt] (6.,2.)-- (8.,2.);
\draw [line width=0.4pt] (8.,2.)-- (8.,3.);
\draw [line width=0.4pt] (8.,3.)-- (6.,3.);
\draw [,line width=0.4pt] (6.,3.)-- (6.,2.);
\draw [line width=0.4pt] (4.,3.5)-- (4.,1.5);
\draw [->,line width=0.4pt] (7.,1.5)-- (7.,2.);
\draw [shift={(5.5,1.5)},line width=0.4pt]  plot[domain=3.141592653589793:6.283185307179586,variable=\t]({1.*1.5*cos(\t r)+0.*1.5*sin(\t r)},{0.*1.5*cos(\t r)+1.*1.5*sin(\t r)});
\draw [line width=0.4pt] (7.,3.)-- (7.,3.5);
\draw [shift={(6.,3.5)},line width=0.4pt]  plot[domain=0.:3.141592653589793,variable=\t]({1.*1.*cos(\t r)+0.*1.*sin(\t r)},{0.*1.*cos(\t r)+1.*1.*sin(\t r)});
\draw [->,line width=0.4pt] (6.5,3.)-- (6.5,3.5);
\draw [->,line width=0.4pt] (7.5,3.)-- (7.5,3.5);
\draw (1.8,2.8) node[anchor=north west] {$x$};
\draw (6.8,2.8) node[anchor=north west] {$y$};
\draw [->,line width=0.4pt] (9.,1.)-- (9.,4.);
\draw (1.1,1.1) node[anchor=north west] {$1$};
\draw (1.6,1.1) node[anchor=north west] {$2$};
\draw (8.6,1.1) node[anchor=north west] {$3$};
\draw (6.1,4.1) node[anchor=north west] {$1$};
\draw (7.1,4.1) node[anchor=north west] {$3$};
\draw (8.6,4.6) node[anchor=north west] {$2$};
\end{tikzpicture}
	\substack{\displaystyle \\ \vspace{3cm}	\xymatrix{&\ar@(ul,dl)[]}\xymatrix{&\ar@(ul,dl)[]}}
\end{align*}

\vspace{-1.5cm}

\begin{align*}
\begin{tikzpicture}[line cap=round,line join=round,>=triangle 45,x=0.7cm,y=0.7cm]
\clip(0.9,-0.1) rectangle (9.5,4.7);
\draw [line width=0.4pt] (1.,2.)-- (3.,2.);
\draw [line width=0.4pt] (3.,2.)-- (3.,3.);
\draw [line width=0.4pt] (3.,3.)-- (1.,3.);
\draw [line width=0.4pt] (1.,3.)-- (1.,2.);
\draw [->,line width=0.4pt] (1.5,1.)-- (2.,2.);
\draw [->,line width=0.4pt] (2.,1.)-- (1.5,2.);
\draw [->,line width=0.4pt] (2.5,1.5)-- (2.5,2.);
\draw [shift={(3.75,1.5)},line width=0.4pt]  plot[domain=3.141592653589793:6.283185307179586,variable=\t]({1.*1.25*cos(\t r)+0.*1.25*sin(\t r)},{0.*1.25*cos(\t r)+1.*1.25*sin(\t r)});
\draw [line width=0.4pt] (5.,1.5)-- (5.,3.5);
\draw [line width=0.4pt] (2.,3.)-- (2.,3.5);
\draw [shift={(3.,3.5)},line width=0.4pt]  plot[domain=0.:3.141592653589793,variable=\t]({1.*1.*cos(\t r)+0.*1.*sin(\t r)},{0.*1.*cos(\t r)+1.*1.*sin(\t r)});
\draw [line width=0.4pt] (4.,3.5)-- (4.,1.5);
\draw [line width=0.4pt] (6.,2.)-- (8.,2.);
\draw [line width=0.4pt] (8.,2.)-- (8.,3.);
\draw [line width=0.4pt] (8.,3.)-- (6.,3.);
\draw [,line width=0.4pt] (6.,3.)-- (6.,2.);
\draw [line width=0.4pt] (4.,3.5)-- (4.,1.5);
\draw [->,line width=0.4pt] (7.,1.5)-- (7.,2.);
\draw [shift={(5.5,1.5)},line width=0.4pt]  plot[domain=3.141592653589793:6.283185307179586,variable=\t]({1.*1.5*cos(\t r)+0.*1.5*sin(\t r)},{0.*1.5*cos(\t r)+1.*1.5*sin(\t r)});
\draw [->,line width=0.4pt] (7.,3.)-- (7.,3.5);
\draw [shift={(5.75,3.5)},line width=.4pt]  plot[domain=0.:3.141592653589793,variable=\t]({1.*0.75*cos(\t r)+0.*0.75*sin(\t r)},{0.*0.75*cos(\t r)+1.*0.75*sin(\t r)});
\draw [line width=0.4pt] (6.5,3.)-- (6.5,3.5);
\draw [->,line width=0.4pt] (7.5,3.)-- (7.5,3.5);
\draw (1.8,2.8) node[anchor=north west] {$x$};
\draw (6.8,2.8) node[anchor=north west] {$y$};
\draw [->,line width=0.4pt] (9.,1.)-- (9.,4.);
\draw (1.1,1.1) node[anchor=north west] {$1$};
\draw (1.6,1.1) node[anchor=north west] {$2$};
\draw (8.6,1.1) node[anchor=north west] {$3$};
\draw (6.6,4.1) node[anchor=north west] {$1$};
\draw (7.1,4.1) node[anchor=north west] {$3$};
\draw (8.6,4.6) node[anchor=north west] {$2$};
\end{tikzpicture}	
\substack{\displaystyle \\ \vspace{3cm}	\xymatrix{&\ar@(ul,dl)[]}\xymatrix{&\ar@(ul,dl)[]}}
\end{align*}

\vspace{-1.5cm}
\begin{align*}
\begin{tikzpicture}[line cap=round,line join=round,>=triangle 45,x=0.7cm,y=0.7cm]
\clip(0.9,-0.1) rectangle (9.5,4.7);
\draw [line width=0.4pt] (1.,2.)-- (3.,2.);
\draw [line width=0.4pt] (3.,2.)-- (3.,3.);
\draw [line width=0.4pt] (3.,3.)-- (1.,3.);
\draw [line width=0.4pt] (1.,3.)-- (1.,2.);
\draw [->,line width=0.4pt] (1.5,1.5)-- (1.5,2.);
\draw [->,line width=0.4pt] (2.,1.5)-- (2.,2.);
\draw [->,line width=0.4pt] (2.5,1.5)-- (2.5,2.);
\draw [line width=0.4pt] (5.,1.5)-- (5.,3.5);
\draw [line width=0.4pt] (2.,3.)-- (2.,3.5);
\draw [shift={(3.,3.5)},line width=0.4pt]  plot[domain=0.:3.141592653589793,variable=\t]({1.*1.*cos(\t r)+0.*1.*sin(\t r)},{0.*1.*cos(\t r)+1.*1.*sin(\t r)});
\draw [line width=0.4pt] (4.,3.5)-- (4.,1.5);
\draw [line width=0.4pt] (6.,2.)-- (8.,2.);
\draw [line width=0.4pt] (8.,2.)-- (8.,3.);
\draw [line width=0.4pt] (8.,3.)-- (6.,3.);
\draw [,line width=0.4pt] (6.,3.)-- (6.,2.);
\draw [line width=0.4pt] (4.,3.5)-- (4.,1.5);
\draw [->,line width=0.4pt] (7.,1.5)-- (7.,2.);
\draw [shift={(5.5,1.5)},line width=0.4pt]  plot[domain=3.141592653589793:6.283185307179586,variable=\t]({1.*1.5*cos(\t r)+0.*1.5*sin(\t r)},{0.*1.5*cos(\t r)+1.*1.5*sin(\t r)});
\draw [->,line width=0.4pt] (7.,3.)-- (7.5,4.);
\draw [shift={(5.75,3.5)},line width=.4pt]  plot[domain=0.:3.141592653589793,variable=\t]({1.*0.75*cos(\t r)+0.*0.75*sin(\t r)},{0.*0.75*cos(\t r)+1.*0.75*sin(\t r)});
\draw [shift={(2.5,1.5)},line width=.4pt]  plot[domain=3.141592653589793:4.71238898038469,variable=\t]({1.*1.*cos(\t r)+0.*1.*sin(\t r)},{0.*1.*cos(\t r)+1.*1.*sin(\t r)});
\draw [shift={(4.,1.5)},line width=.4pt]  plot[domain=-1.5707963267948966:0.,variable=\t]({1.*1.*cos(\t r)+0.*1.*sin(\t r)},{0.*1.*cos(\t r)+1.*1.*sin(\t r)});
\draw [line width=.4pt] (2.5,0.5)-- (4.,0.5);
\draw [line width=0.4pt] (6.5,3.)-- (6.5,3.5);
\draw [->,line width=0.4pt] (7.5,3.)-- (7.,4.);
\draw (1.8,2.8) node[anchor=north west] {$x$};
\draw (6.8,2.8) node[anchor=north west] {$y$};
\draw [->,line width=0.4pt] (9.,1.)-- (9.,4.);
\draw (1.6,1.6) node[anchor=north west] {$1$};
\draw (2.1,1.6) node[anchor=north west] {$2$};
\draw (8.6,1.1) node[anchor=north west] {$3$};
\draw (6.6,4.6) node[anchor=north west] {$1$};
\draw (7.1,4.6) node[anchor=north west] {$3$};
\draw (8.6,4.6) node[anchor=north west] {$2$};
\end{tikzpicture}
	\substack{\displaystyle \\ \vspace{3cm}	\xymatrix{&\ar@(ul,dl)[]}\xymatrix{&\ar@(ul,dl)[]}}
\end{align*}

\vspace{-1.5cm}


\begin{defi}
Let $G$ and $G'$ be two graphs. A \textbf{(resp. iso-)morphism} of graphs from $G$ to $G'$ is a family of  (resp. bijections) maps $f=(f_V,f_E,f_I,f_O,f_{IO},f_L)$ with:
\begin{align*}
f_V&:V(G)\longrightarrow V(G'),&f_E&:E(G)\longrightarrow E(G'),&f_I&:I(G)\longrightarrow I(G'),\\
f_O&:O(G)\longrightarrow O(G'),&f_{IO}&:IO(G)\longrightarrow IO(G'),&f_L&:L(G)\longrightarrow L(G'),
\end{align*}
such that:
\begin{align*}
s'\circ f_E&=f_V\circ s_{\mid E(G)},&s'\circ f_O&=f_V\circ s_{\mid O(G)},\\
t'\circ f_E&=f_V\circ t_{\mid E(G)},&t'\circ f_I&=f_V\circ t_{\mid I(G)},\\
\alpha'\circ f_I&=\alpha_{\mid I(G)},&\alpha'\circ f_{IO}&=\alpha_{\mid IO(G)},\\
\beta'\circ f_O&=\beta_{\mid O(G)},&\beta'\circ f_{IO}&=\beta_{\mid IO(G)}.
\end{align*}
Let $G$ and $G'$ be two {corolla ordered} graphs. A  {\textbf{(resp. iso-)morphism}} of {corolla ordered} graphs from $G$ to $G'$ is an {\textbf{(resp. iso-)morphism}}   of graphs
$f$ from $G$ to $G'$ {which preserves the order of incoming and outgoing edges i.e.,} for any vertex of $v$:
\begin{itemize}
\item For any incoming edges $e$, $e'$ of $v$, $e\leqslant_v e'$ in $G$ if, and only if, $f(e)\leqslant_{f(v)} f(e')$ in $G'$.
\item For any outgoing edges $e$, $e'$ of $v$, $e\leqslant_v e'$ in $G$ if, and only if, $f(e)\leqslant_{f(v)} f(e')$ in $G'$.
\end{itemize}
For any $k,l\in  \N_0$, we denote by $\Gr(k,l)$ the set of the isoclasses of graphs $G$ such that
$i(G)=k$ and $o(G)=l$, i.e. $\Gr(k,l)$ is the quotient space of graphs with $k$ input edges and $l$ output edges by the equivalence relation given by isomorphism. Similarly, we denote by $\PGr(k,l)$ the set of isoclasses of {corolla ordered} 
graphs $G$ such that $i(G)=k$ and $o(G)=l$.

The subset of $\Gr(k,l)$ formed by isoclasses of solar graphs is denoted by $\rGr(k,l)$
and the subset of $\PGr(k,l)$ formed by isoclasses of solar {corolla ordered} graphs is denoted by $\rPGr(k,l)$. 
\end{defi}

In what follows, we shall write \emph{graphs} for \emph{isoclasses of graphs}.
\begin{example}
The isoclass of the graph of Example \ref{ex4} is represented by:
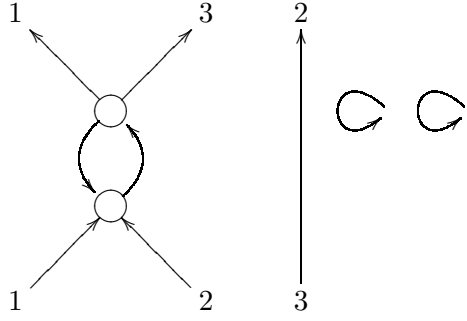
\[\xymatrix{1&&3&2&&\\
&\rond{}\ar[ru] \ar[lu] \ar@/_1pc/[d]&&&\ar@(ul,dl)[]&\ar@(ul,dl)[]\\
&\rond{}\ar@/_1pc/[u]&&&&\\
1\ar[ru]&&2\ar[lu]&3\ar[uuu]&&}\]
\end{example}

We will later define  a monad structure on graphs and {corolla ordered} graphs (Proposition \ref{prop:monades_graphes}).\\

 Throughout the  paper, $X=(X(k,l))_{k,l\geqslant 0}$ is  a family of sets.
 
 \begin{defi}
  A graph decorated by $X=(X(k,l))_{k,l\geqslant 0}$ (or $X$-decorated graph, or simply decorated graph) is a couple $(G,d_G)$ with $G$ a graph as in Definition \ref{defi:graph} and $\displaystyle d_G:V(G)\longrightarrow \bigsqcup_{k,l\in  \N_0}X(k,l)$ a map, such that for any vertex $v\in V(G)$, $d_G(v)\in X {(i(v),o(v))}$. We denote by $\Gr(X)$ the set of graphs decorated by $X$. Similarly, we define  $X$-decorated {corolla ordered} graphs which we denote by $\PGr(X)$.
  
We further write $\Gr(X)(k,l)$ (resp. $\PGr(X)(k,l)$, $\rGr(X)(k,l)$ and $\rPGr(X)(k,l)$) the set of graphs (resp. of {corolla ordered} graphs, solar graphs, solar {corolla ordered} graphs) 
  decorated by $X$ with $k$ inputs (i.e. $|I(G)|=k$) and $l$ ouputs (i.e. $|O(G)|=l$).
 \end{defi}

\subsection{TRAPs of graphs}

As before, $X=(X(k,l))_{k,l\geqslant 0}$ is  a family of sets. We equip the set of graphs (possibly decorated by $X$) with a structure of TRAP. 
Let us first define an action of $\sym\times \sym^{op}$ on graphs.
Let $G=(V(G),E(G),I(G),O(G),IO(G),L(G),s,t,\alpha,\beta) \in \Gr(k,l)$, $\sigma \in \sym_k$ and $\tau\in \sym_l$.
Then:
\[\tau\cdot G\cdot \sigma=
(V(G),E(G),I(G),O(G),IO(G),L(G),s,t,\sigma^{-1}\circ \alpha,\tau \circ \beta).\]
If $G$ is {corolla ordered}, then $\tau\cdot G\cdot \sigma$ is naturally {corolla ordered}; if $G$ is $X$-decorated,
then $\tau\cdot G\cdot \sigma$ is also $X$-decorated. Hence, this defines a structure of $\sym\times \sym^{op}$-module
on $\Gr$, $\PGr$, $\Gr(X)$ and $\PGr(X)$ for any $X$.\\

We now define the \textbf{horizontal concatenation}. If $G$ and $G^\prime$ are two graphs,
	we define a graph $G*G'$ in the following way:
	\begin{align*}
	V(G*G')&=V(G)\sqcup V(G'),&E(G*G')&=E(G)\sqcup E(G'),&L(G*G')&=L(G)\sqcup L(G'),\\
	I(G*G')&=I(G)\sqcup I(G'),&O(G*G')&=O(G)\sqcup O(G'),&IO(G*G')&=IO(G)\sqcup IO(G').
	\end{align*}
	The source and target maps are given by:
	\begin{align*}
	s''_{\mid E(G)\sqcup O(G)}&=s,&s''_{\mid E(G')\sqcup O(G')}&=s',\\
	t''_{\mid E(G)\sqcup I(G)}&=t,&t''_{\mid E(G')\sqcup I(G')}&=t'.
	\end{align*}
	The indexations of the input and output edges are given by:
	\begin{align*}
	\alpha''_{\mid I(G)\sqcup IO(G)}&=\alpha,&\alpha''_{\mid I(G')\sqcup IO(G')}&=i(G)+\alpha',\\
	\beta''_{\mid O(G)\sqcup IO(G)}&=\beta,&\beta''_{\mid O(G')\sqcup IO(G')}&=o(G)+\beta'
	\end{align*}
	with an obvious abuse of notation in the definition of the second column.
	Notice that this product is not commutative in the usual sense for $G*G'$ and $G'*G$ might differ by the indexation of their input and output 
	edges. However, it is commutative in the sense of Axiom 2.(d) of TRAPs.
	Roughly speaking, $G*G'$ is the disjoint union of $G$ and $G'$, the input and output edges of $G'$ being indexed
	after the input and output edges of $G$. 
	\begin{center}
		\begin{tikzpicture}[line cap=round,line join=round,>=triangle 45,x=0.5cm,y=0.5cm]
		\clip(-2.5,-4.) rectangle (1.,4.);
		\draw [line width=0.4pt] (-2.,1.)-- (0.5,1.);
		\draw [line width=0.4pt] (0.5,1.)-- (0.5,-1.);
		\draw [line width=0.4pt] (0.5,-1.)-- (-2.,-1.);
		\draw [line width=0.4pt] (-2.,-1.)-- (-2.,1.);
		\draw [->,line width=0.4pt] (-1.5,1.) -- (-1.5,3.);
		\draw [->,line width=0.4pt] (0.,1.) -- (0.,3.);
		\draw [->,line width=0.4pt] (-1.5,-3.) -- (-1.5,-1.);
		\draw [->,line width=0.4pt] (0.,-3.) -- (0.,-1.);
		\draw (-1.25,0.5) node[anchor=north west] {$G$};
		\draw (-1.8,-3) node[anchor=north west] {$1$};
		\draw (-0.3,-3) node[anchor=north west] {$k$};
		\draw (-1.4,-2.2) node[anchor=north west] {$\ldots$};
		\draw (-1.8,4.2) node[anchor=north west] {$1$};
		\draw (-0.3,4.2) node[anchor=north west] {$l$};
		\draw (-1.4,2.) node[anchor=north west] {$\ldots$};
		\end{tikzpicture}
		$\substack{\displaystyle *\\ \vspace{3cm}}$
		\begin{tikzpicture}[line cap=round,line join=round,>=triangle 45,x=0.5cm,y=0.5cm]
		\clip(-2.5,-4.) rectangle (0.7,4.);
		\draw [line width=0.4pt] (-2.,1.)-- (0.5,1.);
		\draw [line width=0.4pt] (0.5,1.)-- (0.5,-1.);
		\draw [line width=0.4pt] (0.5,-1.)-- (-2.,-1.);
		\draw [line width=0.4pt] (-2.,-1.)-- (-2.,1.);
		\draw [->,line width=0.4pt] (-1.5,1.) -- (-1.5,3.);
		\draw [->,line width=0.4pt] (0.,1.) -- (0.,3.);
		\draw [->,line width=0.4pt] (-1.5,-3.) -- (-1.5,-1.);
		\draw [->,line width=0.4pt] (0.,-3.) -- (0.,-1.);
		\draw (-1.25,0.5) node[anchor=north west] {$G'$};
		\draw (-1.8,-3) node[anchor=north west] {$1$};
		\draw (-0.3,-3) node[anchor=north west] {$k'$};
		\draw (-1.4,-2.2) node[anchor=north west] {$\ldots$};
		\draw (-1.8,4.2) node[anchor=north west] {$1$};
		\draw (-0.3,4.2) node[anchor=north west] {$l'$};
		\draw (-1.4,2.) node[anchor=north west] {$\ldots$};
		\end{tikzpicture}
		$\substack{\displaystyle =\\ \vspace{3cm}}$
		\begin{tikzpicture}[line cap=round,line join=round,>=triangle 45,x=0.5cm,y=0.5cm]
		\clip(-2.5,-4.) rectangle (0.5,4.);
		\draw [line width=0.4pt] (-2.,1.)-- (0.5,1.);
		\draw [line width=0.4pt] (0.5,1.)-- (0.5,-1.);
		\draw [line width=0.4pt] (0.5,-1.)-- (-2.,-1.);
		\draw [line width=0.4pt] (-2.,-1.)-- (-2.,1.);
		\draw [->,line width=0.4pt] (-1.5,1.) -- (-1.5,3.);
		\draw [->,line width=0.4pt] (0.,1.) -- (0.,3.);
		\draw [->,line width=0.4pt] (-1.5,-3.) -- (-1.5,-1.);
		\draw [->,line width=0.4pt] (0.,-3.) -- (0.,-1.);
		\draw (-1.25,0.5) node[anchor=north west] {$G$};
		\draw (-1.8,-3) node[anchor=north west] {$1$};
		\draw (-0.3,-3) node[anchor=north west] {$k$};
		\draw (-1.4,-2.2) node[anchor=north west] {$\ldots$};
		\draw (-1.8,4.2) node[anchor=north west] {$1$};
		\draw (-0.3,4.2) node[anchor=north west] {$l$};
		\draw (-1.4,2.) node[anchor=north west] {$\ldots$};
		\end{tikzpicture}
		\begin{tikzpicture}[line cap=round,line join=round,>=triangle 45,x=0.5cm,y=0.5cm]
		\clip(-2.5,-4.) rectangle (2.,4.);
		\draw [line width=0.4pt] (-2.,1.)-- (0.5,1.);
		\draw [line width=0.4pt] (0.5,1.)-- (0.5,-1.);
		\draw [line width=0.4pt] (0.5,-1.)-- (-2.,-1.);
		\draw [line width=0.4pt] (-2.,-1.)-- (-2.,1.);
		\draw [->,line width=0.4pt] (-1.5,1.) -- (-1.5,3.);
		\draw [->,line width=0.4pt] (0.,1.) -- (0.,3.);
		\draw [->,line width=0.4pt] (-1.5,-3.) -- (-1.5,-1.);
		\draw [->,line width=0.4pt] (0.,-3.) -- (0.,-1.);
		\draw (-1.25,0.5) node[anchor=north west] {$G'$};
		\draw (-2.3,-3) node[anchor=north west] {$k+1$};
		\draw (-0.3,-3) node[anchor=north west] {$k+k'$};
		\draw (-1.4,-2.2) node[anchor=north west] {$\ldots$};
		\draw (-2.3,4.2) node[anchor=north west] {$l+1$};
		\draw (-0.3,4.2) node[anchor=north west] {$l+l'$};
		\draw (-1.4,2.) node[anchor=north west] {$\ldots$};
		\end{tikzpicture}
		
		\vspace{-1.5cm}
	\end{center}
	
	\begin{example} Here is an example of horizontal concatenation :\\
	
	
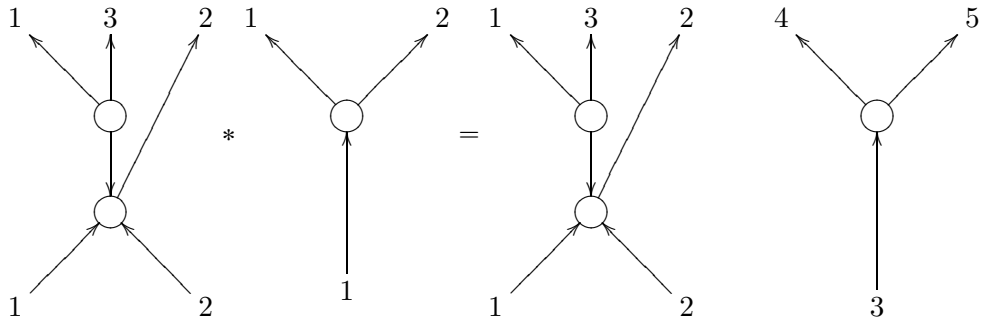
\vspace{-2cm}
	
			\[\xymatrix{1&3&2\\
			&\rond{}\ar[lu] \ar[u]\ar[d]&\\
			&\rond{}\ar[ruu]&\\
			1\ar[ru]&&2\ar[lu]} 
			\substack{\vspace{3cm}\\\displaystyle \mbox{$*$}}
		\xymatrix{1&&2\\
			&\rond{}\ar[ru] \ar[lu]&\\
			&&\\
			&1\ar[uu]&}
					\substack{\vspace{3cm}\\\displaystyle =}
					\xymatrix{1&3&2&4&&5\\
			&\rond{}\ar[lu] \ar[u]\ar[d]&&&\rond{}\ar[ru] \ar[lu]&\\
			&\rond{}\ar[ruu]&&&&\\
			1\ar[ru]&&2\ar[lu]&&3\ar[uu]&}\]
	\end{example}
	
Moreover, if $G$ and $H$ are {corolla ordered} graphs, then $G*H$ is naturally a {corolla ordered} graph.
If $G$ and $H$ are $X$-decorated graphs, then $G*H$ is also naturally an $X$-decorated graph.

Let us finally define the partial trace maps. We only define the outline of their definition, and refer the reader to Appendix \ref{Appendix:B} for a rigorous definition.
Let $G\in \Gr(k,l)$, $1\leqslant i\leqslant k$ and $1\leqslant j\leqslant l$. 
We set $e_i=\alpha_G^{-1}(i)$,  $f_j=\beta_G^{-1}(j)$ and define $t_{i,j}(G)$ 
as the graph obtained by identifying the input of $e_i$ with the output $j$ of $f_j$. 
If $e_i\in I(G)$ and $f_j\in O(G)$, this creates
an edge in $E(G)$. This case is illustrated in the figure below. Otherwise, we create an edge in $I(G)$, or $O(G)$ or $IO(G)$ or in $L(G)$. In all these cases, we then reindex {in non-decreasing order},  the inputs and the outputs of the obtained graph.

Graphically:
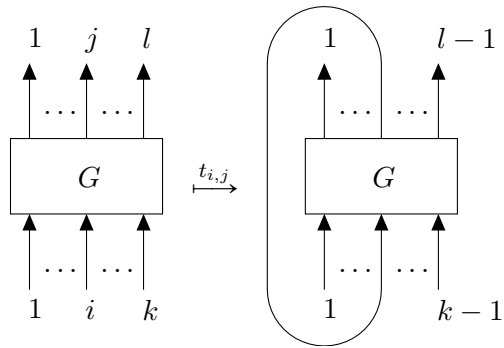
\begin{center}
\begin{tikzpicture}[line cap=round,line join=round,>=triangle 45,x=0.5cm,y=0.5cm]
\clip(-2.5,-4.5) rectangle (2.5,4.);
\draw [line width=0.4pt] (-2.,1.)-- (2.,1.);
\draw [line width=0.4pt] (2.,1.)-- (2.,-1.);
\draw [line width=0.4pt] (2.,-1.)-- (-2.,-1.);
\draw [line width=0.4pt] (-2.,-1.)-- (-2.,1.);
\draw [->,line width=0.4pt] (-1.5,1.) -- (-1.5,3.);
\draw [->,line width=0.4pt] (0.,1.) -- (0.,3.);
\draw [->,line width=0.4pt] (1.5,1.) -- (1.5,3.);
\draw [->,line width=0.4pt] (-1.5,-3.) -- (-1.5,-1.);
\draw [->,line width=0.4pt] (0.,-3.) -- (0.,-1.);
\draw [->,line width=0.4pt] (1.5,-3.) -- (1.5,-1.);
\draw (-0.5,0.5) node[anchor=north west] {$G$};
\draw (-1.8,-3) node[anchor=north west] {$1$};
\draw (-0.3,-3) node[anchor=north west] {$i$};
\draw (1.2,-3) node[anchor=north west] {$k$};
\draw (-1.4,-2.2) node[anchor=north west] {$\ldots$};
\draw (0.1,-2.2) node[anchor=north west] {$\ldots$};
\draw (-1.8,4.2) node[anchor=north west] {$1$};
\draw (-0.3,4.2) node[anchor=north west] {$j$};
\draw (1.2,4.2) node[anchor=north west] {$l$};
\draw (-1.4,2.) node[anchor=north west] {$\ldots$};
\draw (0.1,2.) node[anchor=north west] {$\ldots$};
\end{tikzpicture}
$\substack{\displaystyle \stackrel{t_{i,j}}{\longmapsto}\\ \vspace{4cm}}$
\begin{tikzpicture}[line cap=round,line join=round,>=triangle 45,x=0.5cm,y=0.5cm]
\clip(-3.5,-4.5) rectangle (3.5,5.);
\draw [line width=0.4pt] (-2.,1.)-- (2.,1.);
\draw [line width=0.4pt] (2.,1.)-- (2.,-1.);
\draw [line width=0.4pt] (2.,-1.)-- (-2.,-1.);
\draw [line width=0.4pt] (-2.,-1.)-- (-2.,1.);
\draw [->,line width=0.4pt] (-1.5,1.) -- (-1.5,3.);
\draw [line width=0.4pt] (0.,1.) -- (0.,3.);
\draw [->,line width=0.4pt] (1.5,1.) -- (1.5,3.);
\draw [->,line width=0.4pt] (-1.5,-3.) -- (-1.5,-1.);
\draw [->,line width=0.4pt] (0.,-3.) -- (0.,-1.);
\draw [->,line width=0.4pt] (1.5,-3.) -- (1.5,-1.);
\draw (-0.5,0.5) node[anchor=north west] {$G$};
\draw (-1.8,-3) node[anchor=north west] {$1$};
\draw (1.2,-3) node[anchor=north west] {$k-1$};
\draw (-1.4,-2.2) node[anchor=north west] {$\ldots$};
\draw (0.1,-2.2) node[anchor=north west] {$\ldots$};
\draw (-1.8,4.2) node[anchor=north west] {$1$};
\draw (1.2,4.2) node[anchor=north west] {$l-1$};
\draw (-1.4,2.) node[anchor=north west] {$\ldots$};
\draw (0.1,2.) node[anchor=north west] {$\ldots$};
\draw [shift={(-1.5,-3.)},line width=0.4pt]  plot[domain=3.141592653589793:6.283185307179586,variable=\t]({1.*1.5*cos(\t r)+0.*1.5*sin(\t r)},{0.*1.5*cos(\t r)+1.*1.5*sin(\t r)});
\draw [shift={(-1.5,3.)},line width=0.4pt]  plot[domain=0.:3.141592653589793,variable=\t]({1.*1.5*cos(\t r)+0.*1.5*sin(\t r)},{0.*1.5*cos(\t r)+1.*1.5*sin(\t r)});
\draw [line width=0.4pt] (-3.,3.) -- (-3.,-3.);
\end{tikzpicture}
\vspace{-2cm}
\end{center}
In particular, $t_{1,1}(I)$ is the graph $\grapheo$.
As before, if $G$ is {corolla ordered}, or if it is $X$-decorated, then $t_{i,j}(G)$ is {corolla ordered}, or $X$-decorated.

\begin{example} Let $G$ be the following graph:
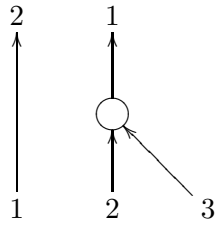
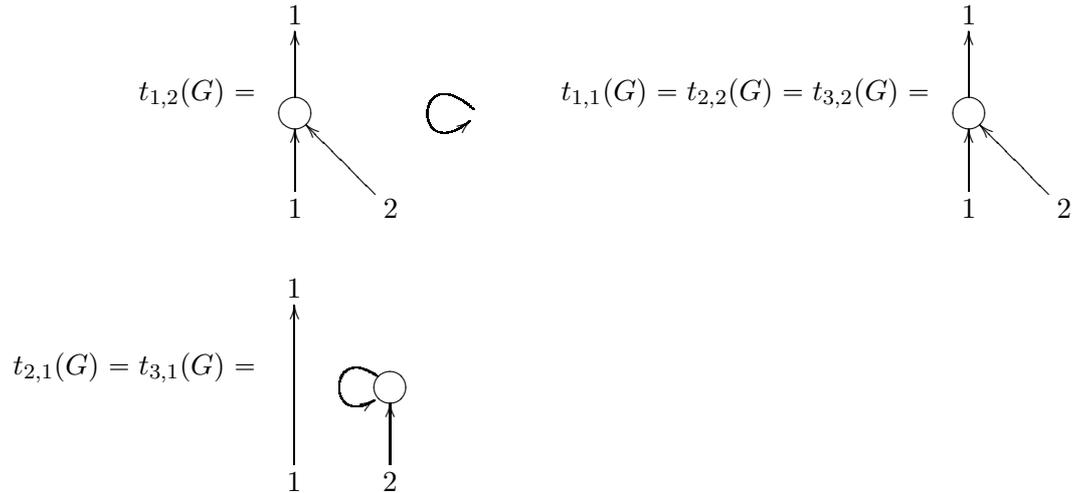
\[\xymatrix{2&1&\\
&\rond{}\ar[u]&\\
1\ar[uu]&2\ar[u]&3\ar[lu]}\]
Then:

\vspace{-2cm}

\begin{align*}
\substack{\vspace{2cm}\\ \displaystyle t_{1,2}(G)=\hspace{.3cm}}&\xymatrix{1&&\\
\rond{}\ar[u]&&\ar@(ul,dl)[]\\
1\ar[u]&2\ar[lu]&}&
\substack{\vspace{2cm}\\ \displaystyle t_{1,1}(G)=t_{2,2}(G)=t_{3,2}(G)=\hspace{.3cm}}
&\xymatrix{1&\\
\rond{}\ar[u]&\\
1\ar[u]&2\ar[lu]}\\[-5mm]
\substack{\vspace{2cm}\\ \displaystyle t_{2,1}(G)=t_{3,1}(G)=\hspace{.3cm}}
&\xymatrix{1&\\
&\rond{}\ar@(ul,dl)[]\\
1\ar[uu]&2\ar[u]}
\end{align*}
Note that $t_{1,2}$ creates a loop when applied on $G$. \end{example}

\begin{prop}\label{prop:PROPgraphs}
With these data, $\Gr$, $\PGr$, $\Gr(X)$ and $\PGr(X)$ are unitary TRAPs.
\end{prop}

\begin{proof} We carry out the proof for ${\PGr }$. The proof is similar for the three other cases.
Properties 1. and 2. are trivial. 
Let us give a graphical interpretation of the proof of Property 3.(a), when $i'<i$ and $j'<j$.

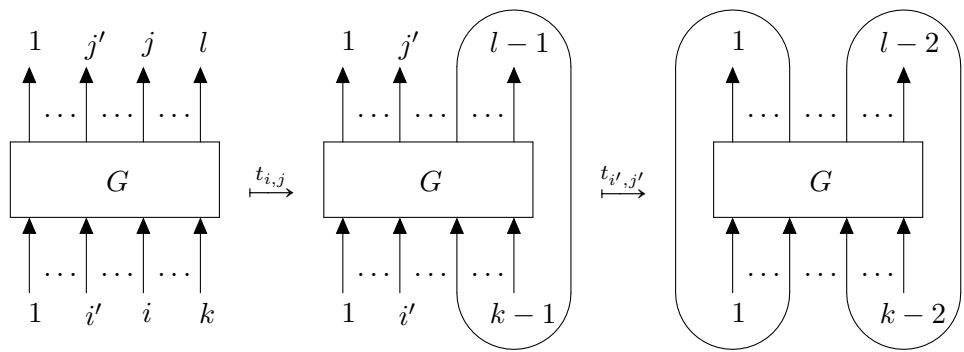
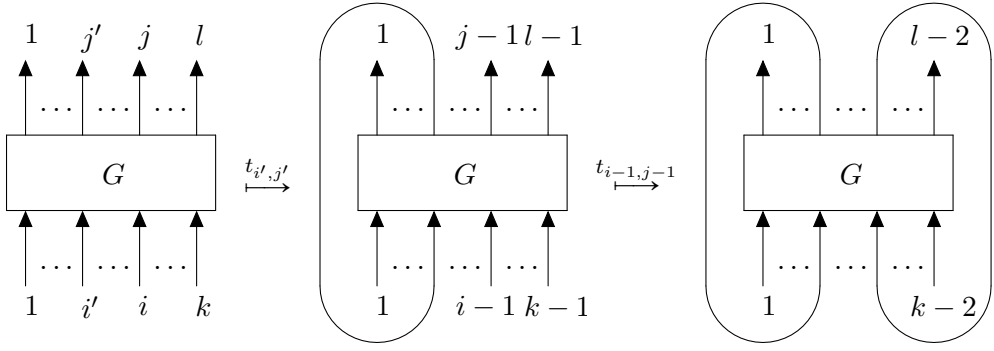
\begin{center}
\begin{tikzpicture}[line cap=round,line join=round,>=triangle 45,x=0.5cm,y=0.5cm]
\clip(-2.5,-4.5) rectangle (4,4.);
\draw [line width=0.4pt] (-2.,1.)-- (3.5,1.);
\draw [line width=0.4pt] (3.5,1.)-- (3.5,-1.);
\draw [line width=0.4pt] (3.5,-1.)-- (-2.,-1.);
\draw [line width=0.4pt] (-2.,-1.)-- (-2.,1.);
\draw [->,line width=0.4pt] (-1.5,1.) -- (-1.5,3.);
\draw [->,line width=0.4pt] (0.,1.) -- (0.,3.);
\draw [->,line width=0.4pt] (1.5,1.) -- (1.5,3.);
\draw [->,line width=0.4pt] (3.,1.) -- (3.,3.);
\draw [->,line width=0.4pt] (-1.5,-3.) -- (-1.5,-1.);
\draw [->,line width=0.4pt] (0.,-3.) -- (0.,-1.);
\draw [->,line width=0.4pt] (1.5,-3.) -- (1.5,-1.);
\draw [->,line width=0.4pt] (3.,-3.) -- (3.,-1.);
\draw (0.25,0.5) node[anchor=north west] {$G$};
\draw (-1.8,-3) node[anchor=north west] {$1$};
\draw (-0.3,-3) node[anchor=north west] {$i'$};
\draw (1.2,-3) node[anchor=north west] {$i$};
\draw (2.7,-3) node[anchor=north west] {$k$};
\draw (-1.4,-2.2) node[anchor=north west] {$\ldots$};
\draw (0.1,-2.2) node[anchor=north west] {$\ldots$};
\draw (1.6,-2.2) node[anchor=north west] {$\ldots$};
\draw (-1.8,4.2) node[anchor=north west] {$1$};
\draw (-0.3,4.2) node[anchor=north west] {$j'$};
\draw (1.2,4.2) node[anchor=north west] {$j$};
\draw (2.7,4.2) node[anchor=north west] {$l$};
\draw (-1.4,2.) node[anchor=north west] {$\ldots$};
\draw (0.1,2.) node[anchor=north west] {$\ldots$};
\draw (1.6,2.) node[anchor=north west] {$\ldots$};
\end{tikzpicture}
$\substack{\displaystyle \stackrel{t_{i,j}}{\longmapsto}\\ \vspace{4cm}}$
\begin{tikzpicture}[line cap=round,line join=round,>=triangle 45,x=0.5cm,y=0.5cm]
\clip(-2.5,-4.5) rectangle (5.,5.);
\draw [line width=0.4pt] (-2.,1.)-- (3.5,1.);
\draw [line width=0.4pt] (3.5,1.)-- (3.5,-1.);
\draw [line width=0.4pt] (3.5,-1.)-- (-2.,-1.);
\draw [line width=0.4pt] (-2.,-1.)-- (-2.,1.);
\draw [->,line width=0.4pt] (-1.5,1.) -- (-1.5,3.);
\draw [->,line width=0.4pt] (0.,1.) -- (0.,3.);
\draw [line width=0.4pt] (1.5,1.) -- (1.5,3.);
\draw [->,line width=0.4pt] (3.,1.) -- (3.,3.);
\draw [->,line width=0.4pt] (-1.5,-3.) -- (-1.5,-1.);
\draw [->,line width=0.4pt] (0.,-3.) -- (0.,-1.);
\draw [->,line width=0.4pt] (1.5,-3.) -- (1.5,-1.);
\draw [->,line width=0.4pt] (3.,-3.) -- (3.,-1.);
\draw (0.25,0.5) node[anchor=north west] {$G$};
\draw (-1.8,-3) node[anchor=north west] {$1$};
\draw (-0.3,-3) node[anchor=north west] {$i'$};
\draw (2.1,-3) node[anchor=north west] {$k-1$};
\draw (-1.4,-2.2) node[anchor=north west] {$\ldots$};
\draw (0.1,-2.2) node[anchor=north west] {$\ldots$};
\draw (1.6,-2.2) node[anchor=north west] {$\ldots$};
\draw (-1.8,4.2) node[anchor=north west] {$1$};
\draw (-0.3,4.2) node[anchor=north west] {$j'$};
\draw (2.1,4.2) node[anchor=north west] {$l-1$};
\draw (-1.4,2.) node[anchor=north west] {$\ldots$};
\draw (0.1,2.) node[anchor=north west] {$\ldots$};
\draw (1.6,2.) node[anchor=north west] {$\ldots$};
\draw [shift={(3.,-3.)},line width=0.4pt]  plot[domain=3.141592653589793:6.283185307179586,variable=\t]({1.*1.5*cos(\t r)+0.*1.5*sin(\t r)},{0.*1.5*cos(\t r)+1.*1.5*sin(\t r)});
\draw [shift={(3.,3.)},line width=0.4pt]  plot[domain=0.:3.141592653589793,variable=\t]({1.*1.5*cos(\t r)+0.*1.5*sin(\t r)},{0.*1.5*cos(\t r)+1.*1.5*sin(\t r)});
\draw [line width=0.4pt] (4.5,3.) -- (4.5,-3.);
\end{tikzpicture}
$\substack{\displaystyle \stackrel{t_{i',j'}}{\longmapsto}\\ \vspace{4cm}}$
\begin{tikzpicture}[line cap=round,line join=round,>=triangle 45,x=0.5cm,y=0.5cm]
\clip(-3.5,-4.5) rectangle (5.,5.);
\draw [line width=0.4pt] (-2.,1.)-- (3.5,1.);
\draw [line width=0.4pt] (3.5,1.)-- (3.5,-1.);
\draw [line width=0.4pt] (3.5,-1.)-- (-2.,-1.);
\draw [line width=0.4pt] (-2.,-1.)-- (-2.,1.);
\draw [->,line width=0.4pt] (-1.5,1.) -- (-1.5,3.);
\draw [line width=0.4pt] (0.,1.) -- (0.,3.);
\draw [line width=0.4pt] (1.5,1.) -- (1.5,3.);
\draw [->,line width=0.4pt] (3.,1.) -- (3.,3.);
\draw [->,line width=0.4pt] (-1.5,-3.) -- (-1.5,-1.);
\draw [->,line width=0.4pt] (0.,-3.) -- (0.,-1.);
\draw [->,line width=0.4pt] (1.5,-3.) -- (1.5,-1.);
\draw [->,line width=0.4pt] (3.,-3.) -- (3.,-1.);
\draw (0.25,0.5) node[anchor=north west] {$G$};
\draw (-1.8,-3) node[anchor=north west] {$1$};
\draw (2.1,-3) node[anchor=north west] {$k-2$};
\draw (-1.4,-2.2) node[anchor=north west] {$\ldots$};
\draw (0.1,-2.2) node[anchor=north west] {$\ldots$};
\draw (1.6,-2.2) node[anchor=north west] {$\ldots$};
\draw (-1.8,4.2) node[anchor=north west] {$1$};
\draw (2.1,4.2) node[anchor=north west] {$l-2$};
\draw (-1.4,2.) node[anchor=north west] {$\ldots$};
\draw (0.1,2.) node[anchor=north west] {$\ldots$};
\draw (1.6,2.) node[anchor=north west] {$\ldots$};
\draw [shift={(-1.5,-3.)},line width=0.4pt]  plot[domain=3.141592653589793:6.283185307179586,variable=\t]({1.*1.5*cos(\t r)+0.*1.5*sin(\t r)},{0.*1.5*cos(\t r)+1.*1.5*sin(\t r)});
\draw [shift={(-1.5,3.)},line width=0.4pt]  plot[domain=0.:3.141592653589793,variable=\t]({1.*1.5*cos(\t r)+0.*1.5*sin(\t r)},{0.*1.5*cos(\t r)+1.*1.5*sin(\t r)});
\draw [line width=0.4pt] (-3.,3.) -- (-3.,-3.);
\draw [shift={(3.,-3.)},line width=0.4pt]  plot[domain=3.141592653589793:6.283185307179586,variable=\t]({1.*1.5*cos(\t r)+0.*1.5*sin(\t r)},{0.*1.5*cos(\t r)+1.*1.5*sin(\t r)});
\draw [shift={(3.,3.)},line width=0.4pt]  plot[domain=0.:3.141592653589793,variable=\t]({1.*1.5*cos(\t r)+0.*1.5*sin(\t r)},{0.*1.5*cos(\t r)+1.*1.5*sin(\t r)});
\draw [line width=0.4pt] (4.5,3.) -- (4.5,-3.);
\end{tikzpicture}

\vspace{-2cm}
\end{center}

\begin{center}
\begin{tikzpicture}[line cap=round,line join=round,>=triangle 45,x=0.5cm,y=0.5cm]
\clip(-2.5,-4.5) rectangle (4,4.);
\draw [line width=0.4pt] (-2.,1.)-- (3.5,1.);
\draw [line width=0.4pt] (3.5,1.)-- (3.5,-1.);
\draw [line width=0.4pt] (3.5,-1.)-- (-2.,-1.);
\draw [line width=0.4pt] (-2.,-1.)-- (-2.,1.);
\draw [->,line width=0.4pt] (-1.5,1.) -- (-1.5,3.);
\draw [->,line width=0.4pt] (0.,1.) -- (0.,3.);
\draw [->,line width=0.4pt] (1.5,1.) -- (1.5,3.);
\draw [->,line width=0.4pt] (3.,1.) -- (3.,3.);
\draw [->,line width=0.4pt] (-1.5,-3.) -- (-1.5,-1.);
\draw [->,line width=0.4pt] (0.,-3.) -- (0.,-1.);
\draw [->,line width=0.4pt] (1.5,-3.) -- (1.5,-1.);
\draw [->,line width=0.4pt] (3.,-3.) -- (3.,-1.);
\draw (0.25,0.5) node[anchor=north west] {$G$};
\draw (-1.8,-3) node[anchor=north west] {$1$};
\draw (-0.3,-3) node[anchor=north west] {$i'$};
\draw (1.2,-3) node[anchor=north west] {$i$};
\draw (2.7,-3) node[anchor=north west] {$k$};
\draw (-1.4,-2.2) node[anchor=north west] {$\ldots$};
\draw (0.1,-2.2) node[anchor=north west] {$\ldots$};
\draw (1.6,-2.2) node[anchor=north west] {$\ldots$};
\draw (-1.8,4.2) node[anchor=north west] {$1$};
\draw (-0.3,4.2) node[anchor=north west] {$j'$};
\draw (1.2,4.2) node[anchor=north west] {$j$};
\draw (2.7,4.2) node[anchor=north west] {$l$};
\draw (-1.4,2.) node[anchor=north west] {$\ldots$};
\draw (0.1,2.) node[anchor=north west] {$\ldots$};
\draw (1.6,2.) node[anchor=north west] {$\ldots$};
\end{tikzpicture}
$\substack{\displaystyle \stackrel{t_{i',j'}}{\longmapsto}\\ \vspace{4cm}}$
\begin{tikzpicture}[line cap=round,line join=round,>=triangle 45,x=0.5cm,y=0.5cm]
\clip(-3.5,-4.5) rectangle (4.,5.);
\draw [line width=0.4pt] (-2.,1.)-- (3.5,1.);
\draw [line width=0.4pt] (3.5,1.)-- (3.5,-1.);
\draw [line width=0.4pt] (3.5,-1.)-- (-2.,-1.);
\draw [line width=0.4pt] (-2.,-1.)-- (-2.,1.);
\draw [->,line width=0.4pt] (-1.5,1.) -- (-1.5,3.);
\draw [line width=0.4pt] (0.,1.) -- (0.,3.);
\draw [->,line width=0.4pt] (1.5,1.) -- (1.5,3.);
\draw [->,line width=0.4pt] (3.,1.) -- (3.,3.);
\draw [->,line width=0.4pt] (-1.5,-3.) -- (-1.5,-1.);
\draw [->,line width=0.4pt] (0.,-3.) -- (0.,-1.);
\draw [->,line width=0.4pt] (1.5,-3.) -- (1.5,-1.);
\draw [->,line width=0.4pt] (3.,-3.) -- (3.,-1.);
\draw (0.25,0.5) node[anchor=north west] {$G$};
\draw (-1.8,-3) node[anchor=north west] {$1$};
\draw (0.3,-3) node[anchor=north west] {$i-1$};
\draw (2.1,-3) node[anchor=north west] {$k-1$};
\draw (-1.4,-2.2) node[anchor=north west] {$\ldots$};
\draw (0.1,-2.2) node[anchor=north west] {$\ldots$};
\draw (1.6,-2.2) node[anchor=north west] {$\ldots$};
\draw (-1.8,4.2) node[anchor=north west] {$1$};
\draw (0.3,4.2) node[anchor=north west] {$j-1$};
\draw (2.1,4.2) node[anchor=north west] {$l-1$};
\draw (-1.4,2.) node[anchor=north west] {$\ldots$};
\draw (0.1,2.) node[anchor=north west] {$\ldots$};
\draw (1.6,2.) node[anchor=north west] {$\ldots$};
\draw [shift={(-1.5,-3.)},line width=0.4pt]  plot[domain=3.141592653589793:6.283185307179586,variable=\t]({1.*1.5*cos(\t r)+0.*1.5*sin(\t r)},{0.*1.5*cos(\t r)+1.*1.5*sin(\t r)});
\draw [shift={(-1.5,3.)},line width=0.4pt]  plot[domain=0.:3.141592653589793,variable=\t]({1.*1.5*cos(\t r)+0.*1.5*sin(\t r)},{0.*1.5*cos(\t r)+1.*1.5*sin(\t r)});
\draw [line width=0.4pt] (-3.,3.) -- (-3.,-3.);
\end{tikzpicture}
$\substack{\displaystyle \stackrel{t_{i-1,j-1}}{\longmapsto}\\ \vspace{4cm}}$
\begin{tikzpicture}[line cap=round,line join=round,>=triangle 45,x=0.5cm,y=0.5cm]
\clip(-3.5,-4.5) rectangle (5.,5.);
\draw [line width=0.4pt] (-2.,1.)-- (3.5,1.);
\draw [line width=0.4pt] (3.5,1.)-- (3.5,-1.);
\draw [line width=0.4pt] (3.5,-1.)-- (-2.,-1.);
\draw [line width=0.4pt] (-2.,-1.)-- (-2.,1.);
\draw [->,line width=0.4pt] (-1.5,1.) -- (-1.5,3.);
\draw [line width=0.4pt] (0.,1.) -- (0.,3.);
\draw [line width=0.4pt] (1.5,1.) -- (1.5,3.);
\draw [->,line width=0.4pt] (3.,1.) -- (3.,3.);
\draw [->,line width=0.4pt] (-1.5,-3.) -- (-1.5,-1.);
\draw [->,line width=0.4pt] (0.,-3.) -- (0.,-1.);
\draw [->,line width=0.4pt] (1.5,-3.) -- (1.5,-1.);
\draw [->,line width=0.4pt] (3.,-3.) -- (3.,-1.);
\draw (0.25,0.5) node[anchor=north west] {$G$};
\draw (-1.8,-3) node[anchor=north west] {$1$};
\draw (2.1,-3) node[anchor=north west] {$k-2$};
\draw (-1.4,-2.2) node[anchor=north west] {$\ldots$};
\draw (0.1,-2.2) node[anchor=north west] {$\ldots$};
\draw (1.6,-2.2) node[anchor=north west] {$\ldots$};
\draw (-1.8,4.2) node[anchor=north west] {$1$};
\draw (2.1,4.2) node[anchor=north west] {$l-2$};
\draw (-1.4,2.) node[anchor=north west] {$\ldots$};
\draw (0.1,2.) node[anchor=north west] {$\ldots$};
\draw (1.6,2.) node[anchor=north west] {$\ldots$};
\draw [shift={(-1.5,-3.)},line width=0.4pt]  plot[domain=3.141592653589793:6.283185307179586,variable=\t]({1.*1.5*cos(\t r)+0.*1.5*sin(\t r)},{0.*1.5*cos(\t r)+1.*1.5*sin(\t r)});
\draw [shift={(-1.5,3.)},line width=0.4pt]  plot[domain=0.:3.141592653589793,variable=\t]({1.*1.5*cos(\t r)+0.*1.5*sin(\t r)},{0.*1.5*cos(\t r)+1.*1.5*sin(\t r)});
\draw [line width=0.4pt] (-3.,3.) -- (-3.,-3.);
\draw [shift={(3.,-3.)},line width=0.4pt]  plot[domain=3.141592653589793:6.283185307179586,variable=\t]({1.*1.5*cos(\t r)+0.*1.5*sin(\t r)},{0.*1.5*cos(\t r)+1.*1.5*sin(\t r)});
\draw [shift={(3.,3.)},line width=0.4pt]  plot[domain=0.:3.141592653589793,variable=\t]({1.*1.5*cos(\t r)+0.*1.5*sin(\t r)},{0.*1.5*cos(\t r)+1.*1.5*sin(\t r)});
\draw [line width=0.4pt] (4.5,3.) -- (4.5,-3.);
\end{tikzpicture}

\vspace{-2cm}
\end{center}
For Property 3.(b), let us consider a graph $p=G$ .
As the input edge indexed by $i$ in $\sigma \cdot G\cdot \tau$ is the input edge of $G$ indexed by $\tau(i)$
and the output edge indexed by $j$ in $\sigma \cdot G\cdot \tau$ is the output edge of $G$ indexed by $\sigma^{-1}(j)$,
$G_1=t_{i,j}(\sigma\cdot G\cdot \tau)$ is the graph obtained by gluing together
the input indexed by $\tau(j)$ and the output indexed by $\sigma^{-1}(j)$, reindexing the input
according to $\sigma_i$ and the output edges by $\tau_j$, so $G_1=\sigma_i\cdot t_{\tau(i),\sigma^{-1}(j)}(G)\cdot \tau_j$.

Let us prove Property 3.(c). By Lemma  \ref{lemmeaxiomessimples}, it is enough to prove it for
$(p,p')=(G,G')$ a pair of graphs and $(i,j)=(1,1)$. In this case,
$e_i$ and $f_j$ are both edges of $G$, so $t_{1,1}(G*G')=t_{1,1}(G)*G'$. 

Let us consider the graph $I$ such that
\[V(I)=E(I)=O(I)=I(I)=L(I)=\emptyset,\]
with $IO(I)$  reduced to a single element. Then for any graph $G$ with $|O(G)|\geq1$,
\[t_{1,2}(I*G)=G.\]
By Lemma \ref{lemmeaxiomessimples}, $I$ is a unit of $\Gr$. 
\end{proof}

If $G$ and $H$ are solar, then   $G*H$ is clearly also solar. If $G\in {\PGr}(k,l)$ is solar, then for any $i\in [k]$ and $j\in [l]$,
$t_{i,j}(G)$ is solar. Indeed, as $IO(G)=\emptyset$, $IO(t_{i,j}(G))=\emptyset$; as $IO(G)=L(G)=\emptyset$,
$L(t_{i,j}(G))=\emptyset$. Hence:

\begin{cor}
$\rGr$, $\rPGr$, $\rGr(X)$ and $\rPGr(X)$ are subTRAPs of
$\Gr$, $\PGr$, $\Gr(X)$ and $\PGr(X)$ {in the sense of Definition \ref{defn:subtrap}}. They are non unitary. 
\end{cor}
{\begin{remark}
     $\Gr$, $\PGr$, $\Gr(X)$ and $\PGr(X)$ admit other sub-TRAPs, for example with vertices with only a prescribed number of possible vertices. These sub-TRAPs might be of importance in question of renormalisability of QFTs, but this question is far from the scope of this work and we therefore do not define rigorously these other objects.  
    \end{remark}
}

\subsection{Morphisms of TRAPs and free TRAPs} \label{sec:morphism_TRAPs_free}

   As before, $X=(X(k,l))_{k,l\geqslant 0}$ is  a family of sets.   It turns out that $\rPGr(X)$ is the free TRAP generated by $X$.
For any $x \in X(k,l)$, we identify $x$ with the graph in $\rPGr(k,l)(X)$ with one vertex decorated by $x$,
$k$ incoming edges, totally ordered according to their indices, and $l$ outgoing edges, totally ordered according to their indices.
For example,   $x\in X(3,2)$  is identified with the {corolla ordered} graph
\begin{align} \label{eq:simple_decorated_graph}
\substack{\hspace{5mm} \:\begin{tikzpicture}[line cap=round,line join=round,>=triangle 45,x=0.7cm,y=0.7cm]
\clip(0.6,-2.1) rectangle (2.2,2.);
\draw [line width=.4pt] (0.8,0.)-- (2.2,0.);
\draw [line width=.4pt] (2.2,0.)-- (2.2,-0.5);
\draw [line width=.4pt] (2.2,-0.5)-- (0.8,-0.5);
\draw [line width=.4pt] (0.8,-0.5)-- (0.8,0.);
\draw (1.2,0.05) node[anchor=north west] {\scriptsize $x$};
\draw [->,line width=.4pt] (1.,0.) -- (1.,1.);
\draw [->,line width=.4pt] (2.,0.) -- (2.,1.);
\draw [->,line width=.4pt] (1.,-1.5) -- (1.,-0.5);
\draw [->,line width=.4pt] (1.5,-1.5) -- (1.5,-0.5);
\draw [->,line width=.4pt] (2.,-1.5) -- (2.,-0.5);
\draw (0.7,-1.4) node[anchor=north west] {\scriptsize $1$};
\draw (1.2,-1.4) node[anchor=north west] {\scriptsize $2$};
\draw (1.7,-1.4) node[anchor=north west] {\scriptsize $3$};
\draw (0.7,1.6) node[anchor=north west] {\scriptsize $1$};
\draw (1.7,1.6) node[anchor=north west] {\scriptsize $2$};
\end{tikzpicture}}
\end{align}

\begin{theo} \label{freetraps}
Let $P$ be a  TRAP and $\phi=(\phi(k,l))_{k,l{\in \N_0}}$ be a map from $X$ to $P$ i.e.,
for any $k,l\in \N_0$, $\phi(k,l):X(k,l)\longrightarrow P(k,l)$ is a map.
Then there exists a unique TRAP morphism $\Phi:\rPGr(X)\longrightarrow P$, sending $x$ to $\phi(x)$ for any $x\in X$. If moreover $P$ is unitary, this morphism $\Phi$ is uniquely extended as a unitary TRAP morphism from $\PGr(X)$ to $P$. 

In other words, $\rPGr(X)$ (resp. $\PGr(X)$) is the free TRAP (resp. the free unitary TRAP, i.e. the free wheeled PROP) generated by $X$.
\end{theo}

\begin{remark}\label{rk:PhiPGrX}In practice we often have $P=X$ and $\phi=\mathrm{Id}$  which yields a    map \begin{equation}\label{eq:PhiPGrX}\Phi:\rPGr(X)\longrightarrow X\end{equation} from decorated graphs   to the space $X$ of decorations.\end{remark}

\begin{example} Here is a trivial yet enlightening example of how $\Phi$ acts on graphs: for $G=\grapheo$, we have $G=t_{1,1}(I)$ and hence $\Phi(G)=t_{1,1}(I_P)$.
	\end{example}
	
\begin{proof} 
We provide here a sketch of the proof,  and refer the reader to the appendix for a full proof. Since $\rPGr(X)\subseteq\PGr(X)$, we take $G\in \PGr(k,l)(X)$ and treat simultaneously the case of solar graphs and the other. We define $\Phi(G)$ for any graph 
$G\in \PGr(k,l)(X)$ by induction on the number $N$  of internal edges of $G$. 

If $N=0$, then $G$ can be written as
\[G=\grapheo^{*p}*\sigma\cdot(I^{*q}*x_1*\ldots *x_r)\cdot \tau,\]
(recall that $\grapheo$ is the graph with no vertex, with only one edge in $L(G)$) where $p,q,r\in  \N_0$, $(k_i,{l_i}) \in  \N_0^2$ for any $i$, $x_i\in X{({k_i,l_i})}$ and $\sigma \in \sym_{q+k_1+\ldots+k_r}$, $\tau\in \sym_{q+l_1+\ldots+l_r}$. 
If $G$ is solar, then $p=q=0$ and this reduces to 
\begin{equation*} 
 G=\sigma\cdot (x_1*\ldots *x_r)\cdot \tau.
\end{equation*}
We then set
\begin{equation} \label{eq:G_simple_solar}
 \Phi(G)=\sigma\cdot(\phi(x_1)*\ldots*\phi(x_k))\cdot \tau.
\end{equation}
If $G$ is not solar and if $P$ is unitary, we denote by $I_P$ the identity of $P$, and we put
\[\Phi(G)={t_{1,1}(I_P)^{*p}}*\sigma\cdot(I_P^{*q}*\phi(x_1)*\ldots*\phi(x_r))\cdot \tau.\]
We can prove that this does not depend on the choice of the decomposition of $G$, with the help of the TRAP axioms applied to $P$.
Let us  now assume that $\Phi(G')$ is defined for any graph with $N-1$ internal edges, for a given $N \geqslant 1$. Let $G$ be a graph with $N$ internal edges and let $e$ be one of these edges. Let $G_e$ be a graph obtained by cutting this edge in two, such that $G=t_{1,1}(G_e)$.  We then set:
\[\Phi(G)=t_{1,1}\circ \Phi(G_e).\]
One can prove that this does not depend on the choice of $e$. It can then be shown that $\Phi$ defined as above is compatible with the partial trace maps. 
\end{proof}
 Since the ingoing and outgoing edges of each vertex  of a {corolla ordered} graph are totally ordered, each {corolla ordered} graph $\PGr$ is naturally acted upon by $\sym\times \sym^{op}$.
 \begin{defi}\label{defiactionsommets}
  For any {corolla ordered} graph $G\in\PGr$ and any vertex $v\in V(G)$, there is a natural action of $\sym_{o(v)}\times \sym_{i(v)}^{op}$ {induced by the action on the  totally ordered edges in} $O(v)$ and $I(v)$. The {corolla ordered} graph obtained from $G$ by the action of $(\sigma,\tau)$ on the vertex $v$ is denoted by
  \begin{equation*}
   \sigma \cdot_v G\cdot_v \tau.
  \end{equation*}
 A similar action can be built on a {corolla ordered} graph $G$   decorated by a family of sets $X$:
  \begin{equation*}
   \sigma \cdot_v (G,d_G)\cdot_v \tau:=(\sigma \cdot_v G\cdot_v \tau,d_G).
  \end{equation*}
 \end{defi}

\begin{example}
\[\substack{\displaystyle (12)\cdot_v\\ \vspace{1.5cm}}
	\begin{tikzpicture}[line cap=round,line join=round,>=triangle 45,x=0.7cm,y=0.7cm]
	\clip(0.8,-0.5) rectangle (2.2,2.5);
	\draw [line width=.4pt] (0.8,0.)-- (2.2,0.);
	\draw [line width=.4pt] (2.2,0.)-- (2.2,-0.5); 
	\draw [line width=.4pt] (2.2,-0.5)-- (0.8,-0.5);
	\draw [line width=.4pt] (0.8,-0.5)-- (0.8,0.); 
	\draw [line width=.4pt] (0.8,2.)-- (2.2,2.); 
	\draw [line width=.4pt] (2.2,2.)-- (2.2,2.5);
	\draw [line width=.4pt] (2.2,2.5)-- (0.8,2.5); 
	\draw [line width=.4pt] (0.8,2.5)-- (0.8,2.);
	\draw [->,line width=.4pt] (1.,0.) -- (1.,2.);
	\draw [->,line width=.4pt] (2.,0.) -- (2.,2.);
	\draw (1.2,0.) node[anchor=north west] {\scriptsize $v$};
	\draw (1.2,2.5) node[anchor=north west] {\scriptsize $w$};
	\end{tikzpicture}\,\substack{\displaystyle =\\ \vspace{1.5cm}}\,
	\begin{tikzpicture}[line cap=round,line join=round,>=triangle 45,x=0.7cm,y=0.7cm]
	\clip(0.8,-0.5) rectangle (2.2,2.5);
	\draw [line width=.4pt] (0.8,0.)-- (2.2,0.); 
	\draw [line width=.4pt] (2.2,0.)-- (2.2,-0.5);
	\draw [line width=.4pt] (2.2,-0.5)-- (0.8,-0.5);
	\draw [line width=.4pt] (0.8,-0.5)-- (0.8,0.);
	\draw [line width=.4pt] (0.8,2.)-- (2.2,2.);
	\draw [line width=.4pt] (2.2,2.)-- (2.2,2.5);
	\draw [line width=.4pt] (2.2,2.5)-- (0.8,2.5);
	\draw [line width=.4pt] (0.8,2.5)-- (0.8,2.);
	\draw [->,line width=.4pt] (1.,0.) -- (1.,2.);
	\draw [->,line width=.4pt] (2.,0.) -- (2.,2.);
	\draw (1.2,0.) node[anchor=north west] {\scriptsize $v$};
	\draw (1.2,2.5) node[anchor=north west] {\scriptsize $w$};
	\end{tikzpicture}
	\substack{\displaystyle \cdot_w (12)\\ \vspace{1.5cm}}
	\substack{\displaystyle =\\ \vspace{1.5cm}}
	\begin{tikzpicture}[line cap=round,line join=round,>=triangle 45,x=0.7cm,y=0.7cm]
	\clip(0.8,-0.5) rectangle (2.2,2.5);
	\draw [line width=.4pt] (0.8,0.)-- (2.2,0.);
	\draw [line width=.4pt] (2.2,0.)-- (2.2,-0.5);
	\draw [line width=.4pt] (2.2,-0.5)-- (0.8,-0.5);
	\draw [line width=.4pt] (0.8,-0.5)-- (0.8,0.);
	\draw [line width=.4pt] (0.8,2.)-- (2.2,2.);
	\draw [line width=.4pt] (2.2,2.)-- (2.2,2.5);
	\draw [line width=.4pt] (2.2,2.5)-- (0.8,2.5);
	\draw [line width=.4pt] (0.8,2.5)-- (0.8,2.);
	\draw [->,line width=.4pt] (1.,0.) -- (2.,2.);
	\draw [->,line width=.4pt] (2.,0.) -- (1.,2.);
	\draw (1.2,0.) node[anchor=north west] {\scriptsize $v$};
	\draw (1.2,2.5) node[anchor=north west] {\scriptsize $w$};
	\end{tikzpicture}\substack{\displaystyle .\\ \vspace{1.5cm}}\]
\vspace{-.5cm}
	 
	 	 \noindent In these pictures, the labelling of the edges outgoing (resp. ingoing  {to})  {from} the vertex $v$ (resp. $w$) are labelled from left to right.
 \end{example}
 
Note that $\Gr(X)$ is obtained from $\PGr(X)$ by forgetting the total orders on the edges, which in fact is equivalent to the trivialisation of this action of symmetric groups on incoming and outgoing edges of any vertex. Hence: 

\begin{cor} \label{cor:extension_identity}
Let $P$ be a TRAP and $\phi=(\phi(k,l))_{k,l\in \N_0}$ be a map from $X$ to $P$.
We assume that for any $x\in X(k,l)$, for any $(\sigma,\tau)\in \sym_k\otimes \sym_l$,
\[\tau\cdot \phi(x)\cdot \sigma=\phi(x).\]
There exists a unique TRAP morphism $\Phi:\rGr(X)\longrightarrow P$, sending $x$ to $\phi(x)$ for any $x\in X$.
If moreover $P$ is unitary, this morphism $\Phi$ is uniquely extended as a unitary TRAP morphism
from $\Gr(X)$ to $P$. 
\end{cor}

We end this paragraph with the non {corolla ordered} counterpart of Remark   \ref{rk:PhiPGrX}:
 \begin{remark}\label{rk:PhiGrX} In practice we often have $P=X$ and $\phi=\mathrm{Id}_P$  which yields a  map 
 \begin{equation}\label{eq:PhiGrX}\Phi:\rPGr(X)\longrightarrow X\end{equation} 
 from decorated {corolla ordered} graphs   to the space $X$ of decorations.\end{remark}

\subsection{ {Extending non-unitary TRAPs}}

In this section, we embed any TRAP $P$ in a unitary TRAP denoted by $\uPGr(P)$. We proceed in the following way:
\begin{itemize}
\item  {We start with the} canonical TRAP morphism from the free TRAP $\rPGr(P)$ generated by $P$ { to $P$}.
\item By Proposition \ref{prop:quotient_from_TRAP_morph}, this defines an equivalence $\sim$ on $\rPGr(P)$,
compatible with the TRAP structure of $\rPGr(P)$. 
\item We then extend this equivalence to $\PGr(P)$, in such a way that it is compatible with the unitary TRAP structure
of $\PGr(P)$, as required in Lemma \ref{lem:relation_respection_TRAP}.
\item Consequently, the quotient $\PGr(P)/\sim$ is a unitary TRAP which contains $P$.
\end{itemize}
For this, we shall need the solar part of any {corolla ordered} graph $G$, which we now define:

\begin{notation}
Let $G\in \PGr(P)(k,l)$. Then there exist a unique $(p,k',l')\in \N_0^2$, {$k'\leq k$, $l'\leq l$}, a unique solar graph $G'\in \rPGr(P)(k',l')$,
a unique pair of permutations $(\sigma,\tau)\in \sym_k\times \sym_l$ such that:
\begin{itemize}
\item $\sigma(1)<\ldots<\sigma(k')$ and $\sigma(k'+1)<\ldots<\sigma(k)$ (that is to say $\sigma$ is a $(k',k-k')$-shuffle);
\item $\tau(1)<\ldots<\tau(l')$ and $\tau(l'+1)<\ldots<\tau(l)$ (that is to say $\tau$ is a $(l',l-l')$-shuffle);
\item $G=\grapheo^{p}*\tau^{-1}\cdot (G'*I^{k-k'})\cdot \sigma$.
\end{itemize}
The graph $G'$ is the \textbf{solar part} of $G$ and is denoted by $\reg(G)$. We also set
\begin{align*}
\sigma&:=\sh_{{in}}(G),& \tau&:=\sh_{out}(G),&p&:=\val_{\grapheo}(G).
\end{align*}
Here $\sh$ stands for shuffle and $\val$ for valuation.
\end{notation}

\begin{remark} As the subsequent example will show, a reindexation of ingoing and outgoing edges
is useful to write the graph as a horizontal product of a solar graph, loops and  $I$.\end{remark}

\begin{example}
Let $G$ be the  graph:
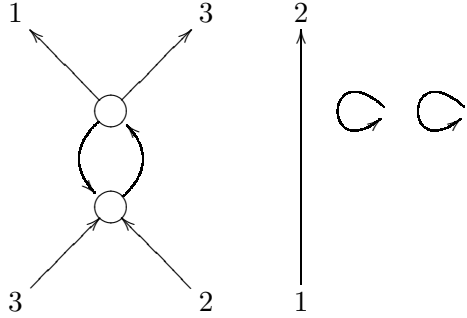
\[\xymatrix{1&&3&2&&\\
&\rond{}\ar[ru] \ar[lu] \ar@/_1pc/[d]&&&\ar@(ul,dl)[]&\ar@(ul,dl)[]\\
&\rond{}\ar@/_1pc/[u]&&&&\\
3\ar[ru]&&2\ar[lu]&1\ar[uuu]&&}\]
Then:
\begin{align*}
G&=\grapheo^2* (132)^{-1}\cdot
	\left(\:\substack{\displaystyle \\ 
\xymatrix{1&&2&3\\
&\rond{}\ar[ru] \ar[lu] \ar@/_1pc/[d]&&\\
&\rond{}\ar@/_1pc/[u]&&\\
2\ar[ru]&&1\ar[lu]&3\ar[uuu]}}\:	\right)\cdot (231)\\[2mm]
&=\grapheo^2* (132)^{-1}\cdot*(G'*I)\cdot(231),
\end{align*}
where $G'$ is the following graph, which as a result is the solar part of $G$:
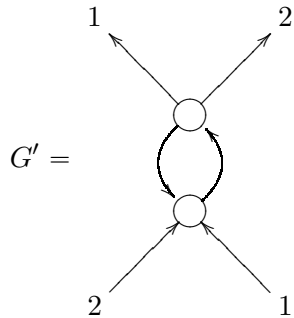
\[G'=\substack{\hspace{5mm}\\ \xymatrix{1&&2&\\
&\rond{}\ar[ru] \ar[lu] \ar@/_1pc/[d]&\\
&\rond{}\ar@/_1pc/[u]&&\\
2\ar[ru]&&1\ar[lu]}}\]
Moreover:
\begin{align*}
\sh_{in}(G)&=(231),& \sh_{out}(G)&=(132),&\val_{\grapheo}(G)&=2.
\end{align*}
\end{example}

Let $P$ be a TRAP and let us consider the unique TRAP morphism $\Phi:{\rPGr}(P)\longrightarrow P$,
extending the identity of $P$. We define a relation $\sim$ on ${\PGr}(P)$ as follows: for $G,G'\in {\PGr}(P)$,
\[G\sim G'\Longleftrightarrow \begin{cases}
\Phi(\reg(G))=\Phi(\reg(G')),\\
\sh_{in}(G)=\sh_{in}(G'),\\
\sh_{out}(G)=\sh_{out}(G'),\\
\val_{\grapheo}(G)=\val_{\grapheo}(G').
\end{cases}\]
This is clearly an equivalence.
Roughly speaking, this equivalence identifies graphs with the same input-output edges and loops and   which coincide after contraction of  their components obtained from deleting 	input-output edges and loops.

\begin{theo}\label{theo:uPGr}
The quotient $\uPGr(P):=\PGr(P)/\sim$ is a unitary TRAP, containing a sub-TRAP isomorphic to $P$. 
\end{theo}

\begin{proof}

	\begin{itemize}
		\item
We first show  the compatibility of the equivalence relation with the left and right actions of the symmetric group.

Let $G,G'\in {\PGr}(P)$, such that $G\sim G'$.  
If $\sigma\in \sym_{k+l}$, there exists a  unique triple $(\sigma_1,\sigma_2,\sigma')\in \sym_{k'}\otimes \sym_{k-k'}\times \sym_{k}$
such that:
\begin{itemize}
\item $\sh_{in}(G)\circ \sigma=\sigma'\circ (\sigma_1\otimes \sigma_2)$.
\item $\sigma'(1)<\ldots<\sigma'(k')$ and $\sigma'(k'+1)<\ldots<\sigma'(k)$.
\end{itemize}
Then $\reg(G\cdot \sigma)=\reg(G)\cdot \sigma_1$, $\sh_{in}(G\cdot \sigma)=\sigma'$ and $\sh_{out}(G\cdot \sigma)=sh_{{out}}(G)$.
Obviously, $\val_{\grapheo}(G\cdot \sigma)=\val_{\grapheo}(G')$. A similar result holds for $G'$. We immediately obtain that:
\begin{align*}
\sh_{in}(G\cdot \sigma)&=\sh_{in}(G'\cdot \sigma),&
sh_o(G\cdot \sigma)&=sh_o(G'\cdot \sigma),&
\val_{\grapheo}(G\cdot \sigma)&=\val_{\grapheo}(G'\cdot \sigma).
\end{align*}
Moreover, as $\Phi$ is a TRAP morphism:
\[\Phi(\reg(G\cdot \sigma))=\Phi(\reg(G)\cdot\sigma_1)=\Phi(\reg(G))\cdot\sigma_1
=\Phi(\reg(G'))\cdot\sigma_1=\Phi(\reg(G'\cdot \sigma)),\]
So $G\cdot \sigma\sim G'\cdot \sigma$. Similarly, if $\tau \in \sym_l$,
$\tau\cdot G\sim \tau \cdot G'$.

	\item 
Let us show the compatibility of the equivalence relation  with the  horizontal concatenation $*$ on the left and on the right.

Let $H\in {\PGr}(P)$. Then:
\begin{align*}
\reg(G*H)&=\reg(G)*\reg(H),\\
\sh_{in}(G*H)&=\sh_{in}(G)*\sh_{in}(H),\\
\sh_{out}(G*H)&=\sh_{out}(G)*\sh_{out}(H),\\
\val_{\grapheo}(G*H)&=\val_{\grapheo}(G)+\val_{\grapheo}(H).
\end{align*}
A similar result holds for $G'*H$. As $G\sim G'$,
\begin{align*}
\sh_{in}(G*H)&=\sh_{in}(G'*H),&
sh_{{out}}(G*H)&=sh_{{out}}(G'*H),&
\val_{\grapheo}(G*H)&=\val_{\grapheo}(G'*H).
\end{align*}
Moreover, as $\Phi$ is a TRAP morphism:
\begin{align*}
\Phi(\reg(G*H))&=\Phi(\reg(G)*\reg(H))\\
&=\Phi(\reg(G))*\Phi(\reg(H))\\
&=\Phi(\reg(G'))*\Phi(\reg(H))=\Phi(\reg(G'*H)).
\end{align*}
Hence, $G*H\sim G'*H$. Similarly, $H*G\sim H*G'$. \\
\item
We now check the compatibility of the equivalence relation with the partial trace maps.

Let $i\in [k]$ and $j\in [l]$. We denote by $e_i$ (respectively $e'_i$) the $i$-th input of $G$ (respectively of $G'$)
and  by $f_j$ (respectively $f'_j$) the $j$-th output of $G$ (respectively of $G'$). There are five possible cases:
\begin{enumerate}
\item If $e_i=f_j\in IO(G)$, then $e'_i=f'_j\in IO(G)$.  Moreover:
\begin{align*}
\sh_{in}(t_{i,j}(G))&=\sh_{in}(t_{i,j}(G')),\\
sh_o(t_{i,j}(G))&=sh_o(t_{i,j}(G')),\\
\val_{\grapheo}(t_{i,j}(G))&=\val_{\grapheo}(t_{i,j}(G'))=\val_{\grapheo}(G)+1,\\
\reg(t_{i,j}(G))&=\reg(G),&\reg(t_{i,j}(G'))&=\reg(G').
\end{align*}
As $G\sim G'$, $\Phi(\reg(G))=\Phi(\reg(G'))$, so $t_{i,j}(G)\sim t'_{i,j}(G)$.
\item If $e_i,f_j \in IO(G)$, with $e_i\neq f_j$, then $e'_i,f'_j \in IO(G')$, with $e'_i\neq f'_j$. Moreover:
\begin{align*}
\sh_{in}(t_{i,j}(G))&=\sh_{in}(t_{i,j}(G')),\\
\sh_o(t_{i,j}(G))&=\sh_o(t_{i,j}(G')),\\
\val_{\grapheo}(t_{i,j}(G))&=\val_{\grapheo}(t_{i,j}(G'))=\val_{\grapheo}(G),\\
\reg(t_{i,j}(G))&=\reg(G),&\reg(t_{i,j}(G'))&=\reg(G').
\end{align*}
So $t_{i,j}(G)\sim t'_{i,j}(G)$.
\item If $e_i\in IO(G)$ and $f_j\notin IO(G)$, then $e'_i\in IO(G')$ and $f_j\notin IO(G')$. Moreover:
\begin{align*}
\sh_{in}(t_{i,j}(G))&=\sh_{in}(t_{i,j}(G')),\\
\sh_o(t_{i,j}(G))&=\sh_o(t_{i,j}(G')),\\
\val_{\grapheo}(t_{i,j}(G))&=\val_{\grapheo}(t_{i,j}(G'))=\val_{\grapheo}(G),
\end{align*}
and there exists a permutation $\alpha \in \sym_{k'}$ such that
\begin{align*}
\reg(t_{i,j}(G))&=\reg(G)\cdot \alpha,&\reg(t_{i,j}(G'))&=\reg(G')\cdot \alpha.
\end{align*}
As $\Phi$ is a TRAP morphism:
\[\Phi(\reg(t_{i,j}(G))=\Phi(\reg(G))\cdot \alpha
=\Phi(\reg(G'))\cdot \alpha=\Phi(\reg(G')\cdot \alpha)=\Phi(\reg(t_{i,j}(G')),\]
so $t_{i,j}(G)\sim t_{i,j}(G')$.
\item The case where  $e_i\notin IO(G)$ and $f_j\in IO(G)$ is treated similarly.
\item If $e_i,f_j\notin IO(G)$, then $e_i,f_j\notin IO(G')$. Moreover:
 \begin{align*}
\sh_{in}(t_{i,j}(G))&=\sh_{in}(t_{i,j}(G')),\\
\sh_{{out}}(t_{i,j}(G))&=\sh_{{out}}(t_{i,j}(G')),\\
\val_{\grapheo}(t_{i,j}(G))&=\val_{\grapheo}(t_{i,j}(G'))=\val_{\grapheo}(G),
\end{align*}
and there exist $i'\in [k']$, $j'\in [l']$, such that
\begin{align*}
\reg(t_{i,j}(G))&=t_{i',j'}(\reg(G)),&\reg(t_{i,j}(G'))&=t_{i',j'}(\reg(G')).
\end{align*}
As $\Phi$ is a TRAP morphism:
\begin{align*}
\Phi(\reg(t_{i,j}(G)))&=\Phi\circ t_{i',j'}(\reg(G))\\
&=t_{i',j'}\circ \Phi(\reg(G))\\
&=t_{i',j'}\circ \Phi(\reg(G'))=\Phi(\reg(t_{i,j}(G'))),
\end{align*}
so $t_{i,j}(G)\sim t_{i,j}(G')$. 
\end{enumerate}
\end{itemize}
By Lemma \ref{lem:relation_respection_TRAP}, $\uPGr(P)$ is a unitary TRAP. 

The canonical injection $\iota:P\longrightarrow \PGr(P)$ induces a map $\iota':P\longrightarrow \uPGr(P)$.
If $p,q$ lie in $ P$, then in $\PGr(P)$, $\iota(p)*\iota(q)$ and $\iota(p*q)$ are solar graphs, and:
\[\Phi(\iota(p)*\iota(q))=\Phi\circ\iota(p)*\Phi\circ\iota(q)=p*q=\Phi\circ \iota(p*q),\]
so $\iota(p)*\iota(q)\sim \iota(p*q)$. Hence, $\iota'(p)*\iota'(q)=\iota(p*q)$. 
If $p\in P(k,l)$, $i\in [k]$ and $j\in [l]$, then $t_{i,j}\circ \iota(p)$ and $\iota\circ t_{i,j}(p)$ are solar graphs in $\PGr(P)$, and:
\[\Phi\circ t_{i,j}\circ \iota(p)=t_{i,j}\circ \Phi\circ \iota(p)=t_{i,j}(p)=\Phi\circ \iota \circ t_{i,j}(p),\]
so $t_{i,j}\circ \iota(p)\sim\iota\circ t_{i,j}(p)$, which implies that $t_{i,j}\circ \iota'(p)=\iota'\circ t_{i,j}(p)$:
the map $\iota'$ is a TRAP morphism. 

Let $p,q\in P$, such that $\iota'(p)=\iota'(q)$. Then $\iota(p)\sim \iota(q)$. As $\iota(p)$ and $\iota(q)$ are solar graphs,
\[p=\Phi\circ \iota(p)=\Phi\circ \iota(q)=q,\]
so $\iota'$ is injective. We have  proved that the unitary TRAP $\uPGr(P)$ contains a (non unitary) sub-TRAP isomorphic to $P$.
\end{proof}

\subsection{ {A functor} from TRAPs to unitary TRAPs}

We now identify the sub-TRAP $\iota'(P)$ of $\uPGr(P)$ and $P$. 
We obtain a description the $\sym_l\times \sym_k^{op}$-module $\uPGr(P)(k,l)$:
\[\uPGr(P)(k,l)=\left(\bigsqcup_{{i=0}}^{\min(k,l)} \mathrm{ind}_{(\sym_i\otimes \sym_{l-i})\times 
(\sym_i\otimes \sym_{k-i})^{op}}^{\sym_l\times \sym_k^{op}} \sym_i\times P(k-i,l-i)\right)\times \{\grapheo^i,i\in \N_0\},\]
where $\mathrm{ind}$ is the  induction of modules and $\sym_i$ is {{equipped with} its canonical $\sym_i\times \sym_i^{op}$-action.
The elements of $\sym_i$ correspond to input-output edges and $\grapheo$ is   the trace of the identity.
The  partial trace maps can be computed with the help of the unitary TRAP axioms. For example, if $p\in P(k-1,l-1)$,
if $j>1$, then:
\[t_{1,j}(\mathrm{Id}_{[1]},p,\grapheo^i)=(\mathrm{Id}_{[0]},(1\ldots j-1)\cdot p, \grapheo^i),\]
which is graphically represented by
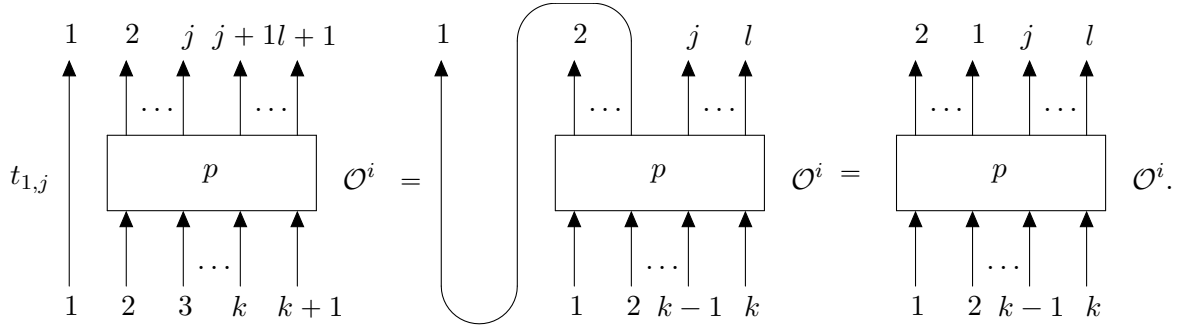
\begin{align*}
\substack{\displaystyle  t_{1,j}\\ \vspace{4cm}}
\begin{tikzpicture}[line cap=round,line join=round,>=triangle 45,x=0.5cm,y=0.5cm]
\clip(-3.5,-4.5) rectangle (4.2,4.);
\draw [line width=0.4pt] (-2.,1.)-- (3.5,1.);
\draw [line width=0.4pt] (3.5,1.)-- (3.5,-1.);
\draw [line width=0.4pt] (3.5,-1.)-- (-2.,-1.);
\draw [line width=0.4pt] (-2.,-1.)-- (-2.,1.);
\draw [->,line width=0.4pt] (-3,-3.) -- (-3.,3.);
\draw [->,line width=0.4pt] (-1.5,1.) -- (-1.5,3.);
\draw [->,line width=0.4pt] (0.,1.) -- (0.,3.);
\draw [->,line width=0.4pt] (1.5,1.) -- (1.5,3.);
\draw [->,line width=0.4pt] (3.,1.) -- (3.,3.);
\draw [->,line width=0.4pt] (-1.5,-3.) -- (-1.5,-1.);
\draw [->,line width=0.4pt] (0.,-3.) -- (0.,-1.);
\draw [->,line width=0.4pt] (1.5,-3.) -- (1.5,-1.);
\draw [->,line width=0.4pt] (3.,-3.) -- (3.,-1.);
\draw (0.25,0.5) node[anchor=north west] {$p$};
\draw (-3.4,-3) node[anchor=north west] {$1$};
\draw (-1.9,-3) node[anchor=north west] {$2$};
\draw (-0.4,-3) node[anchor=north west] {$3$};
\draw (1.,-3) node[anchor=north west] {$k$};
\draw (2.2,-3) node[anchor=north west] {$k+1$};
\draw (0.1,-2.2) node[anchor=north west] {$\ldots$};
\draw (-3.4,4.2) node[anchor=north west] {$1$};
\draw (-1.8,4.2) node[anchor=north west] {$2$};
\draw (-0.3,4.2) node[anchor=north west] {$j$};
\draw (0.5,4.2) node[anchor=north west] {$j+1$};
\draw (2.2,4.2) node[anchor=north west] {$l+1$};
\draw (-1.4,2.) node[anchor=north west] {$\ldots$};
\draw (1.6,2.) node[anchor=north west] {$\ldots$};
\end{tikzpicture}
\substack{\displaystyle \grapheo^i \hspace{.2cm} =\\ \vspace{4cm}}
\begin{tikzpicture}[line cap=round,line join=round,>=triangle 45,x=0.5cm,y=0.5cm]
\clip(-5.5,-4.5) rectangle (4.2,4.5);
\draw [line width=0.4pt] (-2.,1.)-- (3.5,1.);
\draw [line width=0.4pt] (3.5,1.)-- (3.5,-1.);
\draw [line width=0.4pt] (3.5,-1.)-- (-2.,-1.);
\draw [line width=0.4pt] (-2.,-1.)-- (-2.,1.);
\draw [->,line width=0.4pt] (-5.,-3.) -- (-5.,3.);
\draw [line width=0.4pt] (-3,-3.) -- (-3.,3.5);
\draw [->,line width=0.4pt] (-1.5,1.) -- (-1.5,3.);
\draw [line width=0.4pt] (0.,1.) -- (0.,3.5);
\draw [->,line width=0.4pt] (1.5,1.) -- (1.5,3.);
\draw [->,line width=0.4pt] (3.,1.) -- (3.,3.);
\draw [->,line width=0.4pt] (-1.5,-3.) -- (-1.5,-1.);
\draw [->,line width=0.4pt] (0.,-3.) -- (0.,-1.);
\draw [->,line width=0.4pt] (1.5,-3.) -- (1.5,-1.);
\draw [->,line width=0.4pt] (3.,-3.) -- (3.,-1.);
\draw (0.25,0.5) node[anchor=north west] {$p$};
\draw (-1.9,-3) node[anchor=north west] {$1$};
\draw (-0.4,-3) node[anchor=north west] {$2$};
\draw (0.4,-3) node[anchor=north west] {$k-1$};
\draw (2.7,-3) node[anchor=north west] {$k$};
\draw (0.1,-2.2) node[anchor=north west] {$\ldots$};
\draw (-5.4,4.2) node[anchor=north west] {$1$};
\draw (-1.8,4.2) node[anchor=north west] {$2$};
\draw (1.2,4.2) node[anchor=north west] {$j$};
\draw (2.7,4.2) node[anchor=north west] {$l$};
\draw (-1.4,2.) node[anchor=north west] {$\ldots$};
\draw (1.6,2.) node[anchor=north west] {$\ldots$};
\draw [shift={(-1.,3.5)},line width=.4pt]  plot[domain=0.:1.5707963267948966,variable=\t]({1.*1.*cos(\t r)+0.*1.*sin(\t r)},{0.*1.*cos(\t r)+1.*1.*sin(\t r)});	
\draw [line width=0.4pt] (-1.,4.5)-- (-2.,4.5);
\draw [shift={(-2.,3.5)},line width=.4pt]  plot[domain=1.5707963267948966:3.141592654,variable=\t]({1.*1.*cos(\t r)+0.*1.*sin(\t r)},{0.*1.*cos(\t r)+1.*1.*sin(\t r)});	
\draw [shift={(-4.,-3.)},line width=.4pt]  plot[domain=-3.141592654:0,variable=\t]({1.*1.*cos(\t r)+0.*1.*sin(\t r)},{0.*1.*cos(\t r)+1.*1.*sin(\t r)});	
\end{tikzpicture}	
\substack{\displaystyle \grapheo^i \hspace{.2cm} \\ \vspace{4cm}}
&\substack{\displaystyle = \hspace{.2cm} \\ \vspace{4cm}}
\begin{tikzpicture}[line cap=round,line join=round,>=triangle 45,x=0.5cm,y=0.5cm]
\clip(-2.5,-4.5) rectangle (4.2,4.);
\draw [line width=0.4pt] (-2.,1.)-- (3.5,1.);
\draw [line width=0.4pt] (3.5,1.)-- (3.5,-1.);
\draw [line width=0.4pt] (3.5,-1.)-- (-2.,-1.);
\draw [line width=0.4pt] (-2.,-1.)-- (-2.,1.);
\draw [->,line width=0.4pt] (-1.5,1.) -- (-1.5,3.);
\draw [->,line width=0.4pt] (0.,1.) -- (0.,3.);
\draw [->,line width=0.4pt] (1.5,1.) -- (1.5,3.);
\draw [->,line width=0.4pt] (3.,1.) -- (3.,3.);
\draw [->,line width=0.4pt] (-1.5,-3.) -- (-1.5,-1.);
\draw [->,line width=0.4pt] (0.,-3.) -- (0.,-1.);
\draw [->,line width=0.4pt] (1.5,-3.) -- (1.5,-1.);
\draw [->,line width=0.4pt] (3.,-3.) -- (3.,-1.);
\draw (0.25,0.5) node[anchor=north west] {$p$};
\draw (-1.9,-3) node[anchor=north west] {$1$};
\draw (-0.4,-3) node[anchor=north west] {$2$};
\draw (0.4,-3) node[anchor=north west] {$k-1$};
\draw (2.7,-3) node[anchor=north west] {$k$};
\draw (0.1,-2.2) node[anchor=north west] {$\ldots$};
\draw (-1.8,4.2) node[anchor=north west] {$2$};
\draw (-0.3,4.2) node[anchor=north west] {$1$};
\draw (1.,4.2) node[anchor=north west] {$j$};
\draw (2.7,4.2) node[anchor=north west] {$l$};
\draw (-1.4,2.) node[anchor=north west] {$\ldots$};
\draw (1.6,2.) node[anchor=north west] {$\ldots$};
\end{tikzpicture}
\substack{\displaystyle \grapheo^i. \\ \vspace{4cm}}
\end{align*}
 \vspace{-2.5cm}

\noindent and
\[t_{1,1}(\mathrm{Id}_{[1]},p,\grapheo^i)=(\mathrm{Id}_{[0]},p,\grapheo^{i+1}),\]
which is graphically represented by
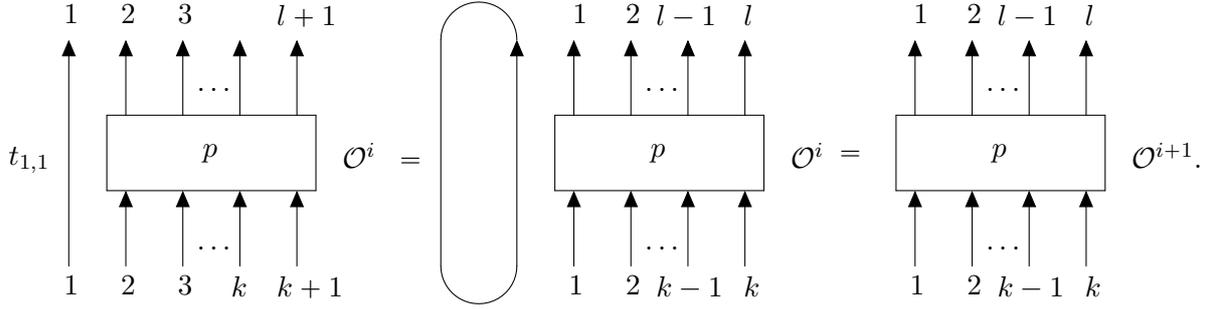
\begin{align*}
\substack{\displaystyle  t_{1,1}\\ \vspace{4cm}}
\begin{tikzpicture}[line cap=round,line join=round,>=triangle 45,x=0.5cm,y=0.5cm]
\clip(-3.5,-4.5) rectangle (4.2,4.);
\draw [line width=0.4pt] (-2.,1.)-- (3.5,1.);
\draw [line width=0.4pt] (3.5,1.)-- (3.5,-1.);
\draw [line width=0.4pt] (3.5,-1.)-- (-2.,-1.);
\draw [line width=0.4pt] (-2.,-1.)-- (-2.,1.);
\draw [->,line width=0.4pt] (-3,-3.) -- (-3.,3.);
\draw [->,line width=0.4pt] (-1.5,1.) -- (-1.5,3.);
\draw [->,line width=0.4pt] (0.,1.) -- (0.,3.);
\draw [->,line width=0.4pt] (1.5,1.) -- (1.5,3.);
\draw [->,line width=0.4pt] (3.,1.) -- (3.,3.);
\draw [->,line width=0.4pt] (-1.5,-3.) -- (-1.5,-1.);
\draw [->,line width=0.4pt] (0.,-3.) -- (0.,-1.);
\draw [->,line width=0.4pt] (1.5,-3.) -- (1.5,-1.);
\draw [->,line width=0.4pt] (3.,-3.) -- (3.,-1.);
\draw (0.25,0.5) node[anchor=north west] {$p$};
\draw (-3.4,-3) node[anchor=north west] {$1$};
\draw (-1.9,-3) node[anchor=north west] {$2$};
\draw (-0.4,-3) node[anchor=north west] {$3$};
\draw (1.,-3) node[anchor=north west] {$k$};
\draw (2.2,-3) node[anchor=north west] {$k+1$};
\draw (0.1,-2.2) node[anchor=north west] {$\ldots$};
\draw (-3.4,4.2) node[anchor=north west] {$1$};
\draw (-1.9,4.2) node[anchor=north west] {$2$};
\draw (-0.4,4.2) node[anchor=north west] {$3$};
\draw (2.2,4.2) node[anchor=north west] {$l+1$};
\draw (0.1,2.) node[anchor=north west] {$\ldots$};
\end{tikzpicture}
&\substack{\displaystyle \grapheo^i \hspace{.2cm} =\\ \vspace{4cm}}
\begin{tikzpicture}[line cap=round,line join=round,>=triangle 45,x=0.5cm,y=0.5cm]
\clip(-5.5,-4.5) rectangle (4.2,4.5);
\draw [line width=0.4pt] (-2.,1.)-- (3.5,1.);
\draw [line width=0.4pt] (3.5,1.)-- (3.5,-1.);
\draw [line width=0.4pt] (3.5,-1.)-- (-2.,-1.);
\draw [line width=0.4pt] (-2.,-1.)-- (-2.,1.);
\draw [line width=0.4pt] (-5.,-3.) -- (-5.,3.);
\draw [->,line width=0.4pt] (-3,-3.) -- (-3.,3.);
\draw [->,line width=0.4pt] (-1.5,1.) -- (-1.5,3.);
\draw [->,line width=0.4pt] (0.,1.) -- (0.,3.);
\draw [->,line width=0.4pt] (1.5,1.) -- (1.5,3.);
\draw [->,line width=0.4pt] (3.,1.) -- (3.,3.);
\draw [->,line width=0.4pt] (-1.5,-3.) -- (-1.5,-1.);
\draw [->,line width=0.4pt] (0.,-3.) -- (0.,-1.);
\draw [->,line width=0.4pt] (1.5,-3.) -- (1.5,-1.);
\draw [->,line width=0.4pt] (3.,-3.) -- (3.,-1.);
\draw (0.25,0.5) node[anchor=north west] {$p$};
\draw (-1.9,-3) node[anchor=north west] {$1$};
\draw (-0.4,-3) node[anchor=north west] {$2$};
\draw (0.4,-3) node[anchor=north west] {$k-1$};
\draw (2.7,-3) node[anchor=north west] {$k$};
\draw (0.1,-2.2) node[anchor=north west] {$\ldots$};
\draw (-1.8,4.2) node[anchor=north west] {$1$};
\draw (-0.4,4.2) node[anchor=north west] {$2$};
\draw (0.4,4.2) node[anchor=north west] {$l-1$};
\draw (2.7,4.2) node[anchor=north west] {$l$};
\draw (0.1,2.) node[anchor=north west] {$\ldots$};
\draw [shift={(-4.,3.)},line width=.4pt]  plot[domain=0.:3.141592654,variable=\t]({1.*1.*cos(\t r)+0.*1.*sin(\t r)},{0.*1.*cos(\t r)+1.*1.*sin(\t r)});	
\draw [shift={(-4.,-3.)},line width=.4pt]  plot[domain=-3.141592654:0,variable=\t]({1.*1.*cos(\t r)+0.*1.*sin(\t r)},{0.*1.*cos(\t r)+1.*1.*sin(\t r)});	
\end{tikzpicture}	
\substack{\displaystyle \grapheo^i \hspace{.2cm} \\ \vspace{4cm}}
\substack{\displaystyle = \hspace{.2cm} \\ \vspace{4cm}}
\begin{tikzpicture}[line cap=round,line join=round,>=triangle 45,x=0.5cm,y=0.5cm]
\clip(-2.5,-4.5) rectangle (4.2,4.);
\draw [line width=0.4pt] (-2.,1.)-- (3.5,1.);
\draw [line width=0.4pt] (3.5,1.)-- (3.5,-1.);
\draw [line width=0.4pt] (3.5,-1.)-- (-2.,-1.);
\draw [line width=0.4pt] (-2.,-1.)-- (-2.,1.);
\draw [->,line width=0.4pt] (-1.5,1.) -- (-1.5,3.);
\draw [->,line width=0.4pt] (0.,1.) -- (0.,3.);
\draw [->,line width=0.4pt] (1.5,1.) -- (1.5,3.);
\draw [->,line width=0.4pt] (3.,1.) -- (3.,3.);
\draw [->,line width=0.4pt] (-1.5,-3.) -- (-1.5,-1.);
\draw [->,line width=0.4pt] (0.,-3.) -- (0.,-1.);
\draw [->,line width=0.4pt] (1.5,-3.) -- (1.5,-1.);
\draw [->,line width=0.4pt] (3.,-3.) -- (3.,-1.);
\draw (0.25,0.5) node[anchor=north west] {$p$};
\draw (-1.9,-3) node[anchor=north west] {$1$};
\draw (-0.4,-3) node[anchor=north west] {$2$};
\draw (0.4,-3) node[anchor=north west] {$k-1$};
\draw (2.7,-3) node[anchor=north west] {$k$};
\draw (0.1,-2.2) node[anchor=north west] {$\ldots$};
\draw (-1.8,4.2) node[anchor=north west] {$1$};
\draw (-0.4,4.2) node[anchor=north west] {$2$};
\draw (0.4,4.2) node[anchor=north west] {$l-1$};
\draw (2.7,4.2) node[anchor=north west] {$l$};
\draw (0.1,2.) node[anchor=north west] {$\ldots$};
\end{tikzpicture}
\substack{\displaystyle \grapheo^{i+1}. \\ \vspace{4cm}}
\end{align*}
\vspace{-2.5cm}

The unitary TRAP $\uPGr(P)$ satisfies the following universal property:

\begin{prop}\label{prop:univpropPGr}
Let $P$ be a TRAP, $Q$  a unitary TRAP and $\theta:P\longrightarrow Q$ be a TRAP morphism. There exists a unique unitary TRAP
morphism $\Theta:\uPGr(P)\longrightarrow Q$ extending $\theta$.
\end{prop}

\begin{proof}
\textit{{Uniqueness}}. Let $\Theta$ be such a morphism. For any $G\in \PGr(P)$, we denote by $[G]$ its class in $\uPGr(P)$. Then:
\[G=\sh_{out}(G)\cdot (\reg(G)*I^{*k''})\cdot \sh_{in}(G),\]
so 
\[\Theta([G])=\sh_{out}(G)\cdot (\theta(\reg(G))*I_Q^{*k''})\cdot \sh_{out}(G),\]
which entirely determines $\Theta$. \\

\textit{Existence}. Let $\overline{\Theta}:\PGr(P)\longrightarrow Q$ be the unique unitary TRAP morphism such that 
$\Theta(p)=\theta(p)$ for any $p\in P$. Let $G,G'\in \PGr(P)$, such that $G\sim G'$. Then:
\begin{align*}
G&=\sh_{out}(G)\cdot (\reg(G)*I^{*k''})\cdot \sh_{in}(G),\\
G'&=\sh_{out}(G)\cdot (\reg(G')*I^{*k''})\cdot \sh_{in}(G),
\end{align*}
and $\Phi(\reg(G))=\Phi(\reg(G'))$ in $P$, so:
\begin{align*}
\overline{\Theta}(G)&=\sh_{out}(G)\cdot (\overline{\Theta}(\reg(G))*I_Q^{*k''})\cdot \sh_{in}(G)\\
&=\sh_{out}(G)\cdot (\theta\circ \Phi(\reg(G))*I_Q^{*k''})\cdot \sh_{in}(G)\\
&=\sh_{out}(G)\cdot (\theta\circ \Phi(\reg(G'))*I_Q^{*k''})\cdot \sh_{in}(G)\\
&=\sh_{out}(G)\cdot (\overline{\Theta}(\reg(G'))*I_Q^{*k''})\cdot \sh_{in}(G)\\
&=\overline{\Theta}(G').
\end{align*}
Hence, $\overline{\Theta}$ induces a unitary TRAP morphism $\Theta:\uPGr(P)\longrightarrow Q$,
extending $\theta$. 
\end{proof}

In other words, $\uPGr$ is a functor from the category of TRAPs to the category of unitary TRAPs,
left adjoint of the forgetful functor from the category of unitary TRAPs to the category of TRAPs. {This functor is the functor $L$ of \cite[Theorem 12.1]{JY15} (with the difference that in \cite{JY15}, one works in the coloured setup). Notice that we have a more explicit and straightforward construction of this tensor than the one of \cite{JY15}.}

\section{ Compositions,  generalised trace and  convolution}

\subsection{Vertical concatenation in a TRAP}

In the same way as wheeled PROPs are PROPs and are equipped with a second associative product \cite{JY15},
TRAPs can be equipped with a natural operation, called the \emph{vertical concatenation}.
We start with the various  TRAPs of graphs we introduced.

Let $G$ and $G'$ be two graphs such that $o(G)=i(G')$. We define a graph 
	$G''=G'\circ G$
	in the following way:
	\begin{align*}
	V(G'')&=V(G)\sqcup V(G'),\\
	E(G'')&=E(G)\sqcup E(G')\sqcup \{(f,e)\in O(G)\times I(G'):\beta(f)=\alpha'(e)\},\\
	I(G'')&=I(G)\sqcup \{(f,e)\in IO(G)\times I(G'):\beta(f)=\alpha'(e)\},\\
	O(G'')&=O(G)\sqcup \{(f,e)\in O(G)\times IO(G') : \beta(f)=\alpha'(e)\},\\
	IO(G'')&=\{(f,e)\in IO(G)\times IO(G'): \beta(f)=\alpha'(e)\},\\
	L(G'')&=L(G)\sqcup L(G').
	\end{align*}	
	Its \textbf{source} and \textbf{target} maps  are given by:
	\begin{align*}
	s''_{\mid E(G)}&=s_{\mid E(G)},&s''_{\mid E(G')}&=s'_{\mid E(G')},&
	s''_{\mid O(G')}&=s'_{\mid O(G')},&s''((f,e))&=s(f),\\
	t''_{\mid E(G)}&=t_{\mid E(G)},&t''_{\mid E(G')}&=s'_{\mid E(G')},&
	t''_{\mid I(G)}&=s_{\mid I(G)},&s''((f,e))&=t'(e).
	\end{align*}
	The indexations of its input and output edges are given by:
	\begin{align*}
	\alpha''_{\mid I(G)}&=\alpha_{\mid I(G)},& \alpha''((f,e))&=\alpha(f),\\
	\beta''_{\mid O(G')}&=\beta'_{\mid O(G')},&\beta''((f,e))&=\beta'(e).
	\end{align*}
	Roughly speaking, $G'\circ G$ is obtained by gluing together the outgoing edges of $G$ and the incoming
	edges of $G'$ according to their indexation. 
		
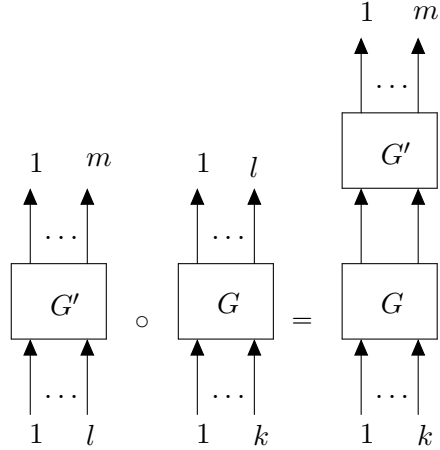
\begin{center}
		\begin{tikzpicture}[line cap=round,line join=round,>=triangle 45,x=0.5cm,y=0.5cm]
		\clip(-2.5,-4.) rectangle (1.,4.);
		\draw [line width=0.4pt] (-2.,1.)-- (0.5,1.);
		\draw [line width=0.4pt] (0.5,1.)-- (0.5,-1.);
		\draw [line width=0.4pt] (0.5,-1.)-- (-2.,-1.);
		\draw [line width=0.4pt] (-2.,-1.)-- (-2.,1.);
		\draw [->,line width=0.4pt] (-1.5,1.) -- (-1.5,3.);
		\draw [->,line width=0.4pt] (0.,1.) -- (0.,3.);
		\draw [->,line width=0.4pt] (-1.5,-3.) -- (-1.5,-1.);
		\draw [->,line width=0.4pt] (0.,-3.) -- (0.,-1.);
		\draw (-1.25,0.5) node[anchor=north west] {$G'$};
		\draw (-1.8,-3) node[anchor=north west] {$1$};
		\draw (-0.3,-3) node[anchor=north west] {$l$};
		\draw (-1.4,-2.2) node[anchor=north west] {$\ldots$};
		\draw (-1.8,4.2) node[anchor=north west] {$1$};
		\draw (-0.3,4.2) node[anchor=north west] {$m$};
		\draw (-1.4,2.) node[anchor=north west] {$\ldots$};
		\end{tikzpicture}
		$\substack{\displaystyle \circ\\ \vspace{3cm}}$		
		\begin{tikzpicture}[line cap=round,line join=round,>=triangle 45,x=0.5cm,y=0.5cm]
		\clip(-2.5,-4.) rectangle (0.7,4.);
		\draw [line width=0.4pt] (-2.,1.)-- (0.5,1.);
		\draw [line width=0.4pt] (0.5,1.)-- (0.5,-1.);
		\draw [line width=0.4pt] (0.5,-1.)-- (-2.,-1.);
		\draw [line width=0.4pt] (-2.,-1.)-- (-2.,1.);
		\draw [->,line width=0.4pt] (-1.5,1.) -- (-1.5,3.);
		\draw [->,line width=0.4pt] (0.,1.) -- (0.,3.);
		\draw [->,line width=0.4pt] (-1.5,-3.) -- (-1.5,-1.);
		\draw [->,line width=0.4pt] (0.,-3.) -- (0.,-1.);
		\draw (-1.25,0.5) node[anchor=north west] {$G$};
		\draw (-1.8,-3) node[anchor=north west] {$1$};
		\draw (-0.3,-3) node[anchor=north west] {$k$};
		\draw (-1.4,-2.2) node[anchor=north west] {$\ldots$};
		\draw (-1.8,4.2) node[anchor=north west] {$1$};
		\draw (-0.4,4.1) node[anchor=north west] {$l$};
		\draw (-1.4,2.) node[anchor=north west] {$\ldots$};
		\end{tikzpicture}
		$\substack{\displaystyle =\\ \vspace{3cm}}$
		\begin{tikzpicture}[line cap=round,line join=round,>=triangle 45,x=0.5cm,y=0.5cm]
		\clip(-2.5,-4.) rectangle (0.5,8.);
		\draw [line width=0.4pt] (-2.,1.)-- (0.5,1.);
		\draw [line width=0.4pt] (0.5,1.)-- (0.5,-1.);
		\draw [line width=0.4pt] (0.5,-1.)-- (-2.,-1.);
		\draw [line width=0.4pt] (-2.,-1.)-- (-2.,1.);
		\draw [->,line width=0.4pt] (-1.5,1.) -- (-1.5,3.);
		\draw [->,line width=0.4pt] (0.,1.) -- (0.,3.);
		\draw [->,line width=0.4pt] (-1.5,-3.) -- (-1.5,-1.);
		\draw [->,line width=0.4pt] (0.,-3.) -- (0.,-1.);
		\draw (-1.25,0.5) node[anchor=north west] {$G$};
		\draw (-1.8,-3) node[anchor=north west] {$1$};
		\draw (-0.3,-3) node[anchor=north west] {$k$};
		\draw (-1.4,-2.2) node[anchor=north west] {$\ldots$};
		\draw [line width=0.4pt] (-2.,5.)-- (0.5,5.);
		\draw [line width=0.4pt] (0.5,5.)-- (0.5,3.);
		\draw [line width=0.4pt] (0.5,3.)-- (-2.,3.);
		\draw [line width=0.4pt] (-2.,3.)-- (-2.,5.);
		\draw [->,line width=0.4pt] (-1.5,5.) -- (-1.5,7.);
		\draw [->,line width=0.4pt] (0.,5.) -- (0.,7.);
		\draw (-1.25,4.5) node[anchor=north west] {$G'$};
		\draw (-1.8,8.2) node[anchor=north west] {$1$};
		\draw (-0.4,8.1) node[anchor=north west] {$m$};
		\draw (-1.4,6.) node[anchor=north west] {$\ldots$};
		\end{tikzpicture}
		
		\vspace{-1.5cm}
	\end{center}
	\begin{example} Here is an example of vertical concatenation :\\
	
	
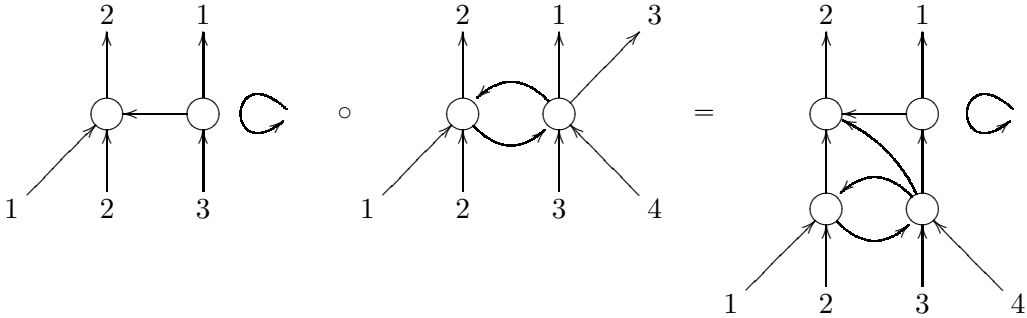
\vspace{-1.8cm}
		\[
		\xymatrix{&2&1& \\ 
			&\rond{}\ar[u]&\rond{}\ar[l] \ar[u] &\ar@(ul,dl)[]\\
			1\ar[ru]&2\ar[u]&3\ar[u]&}\hspace{5mm}
		\substack{\vspace{2.5cm}\\ \displaystyle \circ}
		\xymatrix{&2&1&3\\
			&\rond{}\ar[u]\ar@/_1pc/[r]&\rond{}\ar[u]\ar[ru]\ar@/_1pc/[l]&\\
			1\ar[ru]&2\ar[u]&3\ar[u]&4\ar[lu]} 
						\substack{\vspace{2.5cm}\\ \hspace{.3cm}\displaystyle=}
		\xymatrix{&2&1&\\
			&\rond{}\ar[u]&\rond{}\ar[u]\ar[l]&\ar@(ul,dl)[]\\
			&\rond{}\ar[u]\ar@/_1pc/[r]&\rond{}\ar@/_.5pc/[lu]\ar[u]\ar@/_1pc/[l]&\\
			1\ar[ru]&2\ar[u]&3\ar[u]&4\ar[lu]}\]
	\end{example}
If $G$ and $G'$ are {corolla ordered} (respectively $X$-decorated) graphs, then $G$ and $G'$ are naturally {corolla ordered} (respectively $X$-decorated).
This operation $\circ$ is clearly associative. Moreover, denoting by $I$ the identity graph, for any $k,l\in \N$,
for any graph $G$ with $k$ inputs and $l$ outputs,
\[I^{*l}\circ G=G\circ I^{*k}=G.\]
The vertical concatenation  can be described in terms of the horizontal concatenation and of the partial trace maps:
If $G$ is a graph with $k$ inputs and $l$ outputs, and $G'$ a graph with $l$ inputs and $m$ outputs, then:
\begin{align*}
t_{k+1,1}\circ \ldots \circ t_{k+l-1,l-1}\circ t_{k+l,l}(G*G')&=G\circ G',
\end{align*}
or, graphically:
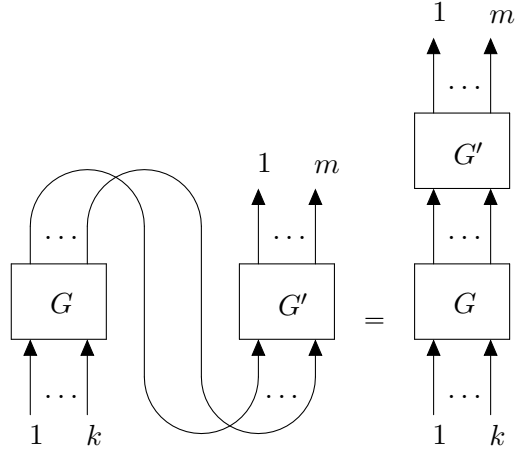
\begin{center}
\begin{tikzpicture}[line cap=round,line join=round,>=triangle 45,x=0.5cm,y=0.5cm]
\clip(-2.5,-4.) rectangle (7.,5.);
\draw [line width=0.4pt] (-2.,1.)-- (0.5,1.);
\draw [line width=0.4pt] (0.5,1.)-- (0.5,-1.);
\draw [line width=0.4pt] (0.5,-1.)-- (-2.,-1.);
\draw [line width=0.4pt] (-2.,-1.)-- (-2.,1.);
\draw [line width=0.4pt] (-1.5,1.) -- (-1.5,2.);
\draw [line width=0.4pt] (0.,1.) -- (0.,2.);
\draw [->,line width=0.4pt] (-1.5,-3.) -- (-1.5,-1.);
\draw [->,line width=0.4pt] (0.,-3.) -- (0.,-1.);
\draw (-1.25,0.5) node[anchor=north west] {$G$};
\draw (-1.8,-3) node[anchor=north west] {$1$};
\draw (-0.3,-3) node[anchor=north west] {$k$};
\draw (-1.4,-2.2) node[anchor=north west] {$\ldots$};
\draw (-1.4,2.) node[anchor=north west] {$\ldots$};
\draw [shift={(0.,2.)},line width=0.4pt]  plot[domain=0.:3.141592653589793,variable=\t]({1.*1.5*cos(\t r)+0.*1.5*sin(\t r)},{0.*1.5*cos(\t r)+1.*1.5*sin(\t r)});
\draw [shift={(3.,-2.)},line width=0.4pt]  plot[domain=3.141592653589793:6.283185307179586,variable=\t]({1.*1.5*cos(\t r)+0.*1.5*sin(\t r)},{0.*1.5*cos(\t r)+1.*1.5*sin(\t r)});
\draw [line width=0.4pt] (1.5,2.) -- (1.5,-2.);
\draw [shift={(1.5,2.)},line width=0.4pt]  plot[domain=0.:3.141592653589793,variable=\t]({1.*1.5*cos(\t r)+0.*1.5*sin(\t r)},{0.*1.5*cos(\t r)+1.*1.5*sin(\t r)});
\draw [shift={(4.5,-2.)},line width=0.4pt]  plot[domain=3.141592653589793:6.283185307179586,variable=\t]({1.*1.5*cos(\t r)+0.*1.5*sin(\t r)},{0.*1.5*cos(\t r)+1.*1.5*sin(\t r)});
\draw [line width=0.4pt] (3.,2.) -- (3.,-2.);
\draw [line width=0.4pt] (4.,1.)-- (6.5,1.);
\draw [line width=0.4pt] (6.5,1.)-- (6.5,-1.);
\draw [line width=0.4pt] (6.5,-1.)-- (4.,-1.);
\draw [line width=0.4pt] (4.,-1.)-- (4.,1.);
\draw [->,line width=0.4pt] (4.5,1.) -- (4.5,3.);
\draw [->,line width=0.4pt] (6.,1.) -- (6.,3.);
\draw [->,line width=0.4pt] (4.5,-2.) -- (4.5,-1.);
\draw [->,line width=0.4pt] (6.,-2.) -- (6.,-1.);
\draw (4.75,0.5) node[anchor=north west] {$G'$};
\draw (4.2,4.2) node[anchor=north west] {$1$};
\draw (5.7,4) node[anchor=north west] {$m$};
\draw (4.4,-2.2) node[anchor=north west] {$\ldots$};
\draw (4.6,2.) node[anchor=north west] {$\ldots$};
\end{tikzpicture}
$\substack{\displaystyle =\\ \vspace{3cm}}$
\begin{tikzpicture}[line cap=round,line join=round,>=triangle 45,x=0.5cm,y=0.5cm]
\clip(-2.5,-4.) rectangle (1.,8.);
\draw [line width=0.4pt] (-2.,1.)-- (0.5,1.);
\draw [line width=0.4pt] (0.5,1.)-- (0.5,-1.);
\draw [line width=0.4pt] (0.5,-1.)-- (-2.,-1.);
\draw [line width=0.4pt] (-2.,-1.)-- (-2.,1.);
\draw [->,line width=0.4pt] (-1.5,1.) -- (-1.5,3.);
\draw [->,line width=0.4pt] (0.,1.) -- (0.,3.);
\draw [->,line width=0.4pt] (-1.5,-3.) -- (-1.5,-1.);
\draw [->,line width=0.4pt] (0.,-3.) -- (0.,-1.);
\draw (-1.25,0.5) node[anchor=north west] {$G$};
\draw (-1.8,-3) node[anchor=north west] {$1$};
\draw (-0.3,-3) node[anchor=north west] {$k$};
\draw (-1.4,-2.2) node[anchor=north west] {$\ldots$};
\draw (-1.4,2.) node[anchor=north west] {$\ldots$};
\draw [line width=0.4pt] (-2.,5.)-- (0.5,5.);
\draw [line width=0.4pt] (0.5,5.)-- (0.5,3.);
\draw [line width=0.4pt] (0.5,3.)-- (-2.,3.);
\draw [line width=0.4pt] (-2.,3.)-- (-2.,5.);
\draw [->,line width=0.4pt] (-1.5,5.) -- (-1.5,7.);
\draw [->,line width=0.4pt] (0.,5.) -- (0.,7.);
\draw (-1.25,4.5) node[anchor=north west] {$G'$};
\draw (-1.8,8.2) node[anchor=north west] {$1$};
\draw (-0.3,8.) node[anchor=north west] {$m$};
\draw (-1.4,6.) node[anchor=north west] {$\ldots$};
\end{tikzpicture}
\end{center}

\vspace{-1.5cm}

This construction can be generalised from TRAPs of graphs to an arbitrary TRAP:

\begin{prop}\label{propverticalconcatenation}
Let $P$ be a TRAP. We define a vertical concatenation\footnote{
	When there is a risk of confusion, we will write $\circ_P$ for the vertical concatenation of a given TRAP $P$.}  $\circ$ in the following way:
\begin{align*}
&\forall (k,l,m)\in \N_0^3,\:\forall p\in P(k,l),\:\forall q\in P(l,m),&
q\circ p&=t_{k+1,1}\circ \ldots \circ t_{k+l-1,l-1}\circ t_{k+l,l}(p*q).
\end{align*}
This operation is associative: for any $(k,l,m,n)\in \N^4$, for any $(p,q,r)\in P(k,l)\times P(l,m)\times P(l,n)$,
\begin{equation}\label{eq:assovertconcTRAP}
r\circ (q\circ p)=(r\circ q)\circ p.\end{equation}
If the TRAP is unitary, then for any $(k,l)\in \N_0^2$, for any $p\in P(k,l)$, if $I_P$ is the unit of $P$, then
\[I_P^{*l}\circ p=p\circ I_P^{*k}=p.\]
\end{prop}

\begin{proof} 
Recall that in Subsection \ref{sec:morphism_TRAPs_free} we identified any element $p$ of the decorating set and the solar graph with one vertex decorated $p$ (see Equation (\ref{eq:simple_decorated_graph}). 
Let $\alpha:\rPGr(P)\longrightarrow P$ be the unique TRAP morphism such that $\alpha(p)=p$ for any $p\in P$.
This a surjective TRAP morphism. As $\alpha$ respects the horizontal concatenation and the partial trace maps, 
for any graphs $G,G' \in \rPGr(P)$,
\[\alpha(G)\circ \alpha(G')=\alpha(G\circ G').\]
Since the vertical concatenation is clearly associative in $\rPGr(P)$, the vertical concatenation is associative in $P$.
If $P$ is unitary, this morphism is extended as a unitary TRAP morphism from $\PGr(P)$ to $P$, which we also denote by $\alpha$.
For any $p\in P(k,l)$, in $\PGr(P)$:
\[I^{*l}\circ p=p\circ I^{*k}=p.\]
As $\alpha(I)=I_P$, in $P$:
\[\alpha(I^{*l}\circ p)=I_P^{*l}\circ p=p=\alpha(p\circ I^{*k})=p\circ I_P^{*k}.\qedhere\]
\end{proof}
{
\begin{remark}
 One could also define {\bf partial vertical concatenations}, where only a 
 subset of the outputs are glued to the inputs with the partial trace maps, in the spirit of \cite[Paragraph 3.3.3]{JY15}. We do not pursue this course here since such partial vertical concatenations will play no role in the rest of the paper.
\end{remark}
}

\begin{example}
Let $V$ be a vector space and let $f=\theta(v_1\ldots v_l\otimes f_1\ldots f_k)\in \Hom_V^{fr}(k,l)$,
$g=\theta(w_1\ldots w_m\otimes g_1\ldots g_l) \in \Hom_V^{fr}(l,m)$. Then, denoting by $\bullet$ the vertical
concatenation of $\Hom_V^{fr}$:
\begin{align*}
g\bullet f&=g_1(v_1)\ldots g_l(v_l)\theta(w_1\ldots w_m\otimes f_1\ldots f_k)\\
&=\theta(w_1\ldots w_m\otimes g_1\ldots g_l) \circ \theta(v_1\ldots v_l\otimes f_1\ldots f_k)\\
&=f\circ g.
\end{align*}
Hence, the vertical concatenation induced by the TRAP structure is the usual composition of linear maps. 
If $V$ is not finite-dimensional, this composition does not have a unit, as $\mathrm{Id}_V$ is not of finite rank.
\end{example}

We end  this subsection with a simple yet important Proposition.
\begin{prop} \label{prop:morphismPROPfromTRAP}
For any two   TRAPs $P=(P(k,l))_{k,l\in\N^2}$ and $Q$, any   TRAP morphism $\phi=(\phi(k,l))_{k,l\in\N^2}:P\longrightarrow Q$ yields a morphism for the vertical concatenations of   $P$ and $Q$.
\end{prop}

\begin{proof}
 We  need to show that for any TRAPs $P$ and $Q$ and any TRAP morphism $\phi:P\longrightarrow Q$ as in the statement of the proposition, for any $(k,l,m)\in\N^3$, $p_1\in P(k,l)$ and $p_2\in P(l,m)$ we have
 \begin{equation*}
  \phi(k,m)(p_2\circ p_1) = \phi(k,l)(p_1) \circ\phi(l,m)(p_2).
 \end{equation*}
 Using the definition of the vertical concatenation in the TRAP $P$ and the third property of the Definition \ref{defn:trap_morphism} of morphisms of TRAPs we have
 \begin{equation*}
  \phi(k,m)(p_2\circ p_1) = t_{k+1,1}^Q\circ\cdots\circ t_{k+l,l}^Q[\phi(k+l,m+l)(p_1*p_2)]
 \end{equation*}
 with $t_{i,j}^Q$ the  partial trace maps  of the TRAP $Q$. Then using the second property of Definition \ref{defn:trap_morphism} 
 we obtain:
 \begin{equation*}
  \phi(k,m)(p_2\circ p_1) = t_{k+1,1}^Q\circ\cdots\circ t_{k+l,l}^Q[\phi(k,l)(p_1)*\phi(l,m)(p_2)] = \phi(k,l)(p_1) \circ\phi(l,m)(p_2).
  \qedhere  \end{equation*}
\end{proof}

\subsection{The generalised trace  on  a TRAP}

If $G$ is a graph with the same number of inputs and outputs, we define its generalised trace by, roughly speaking, grafting
any of its input to the output with the same index:
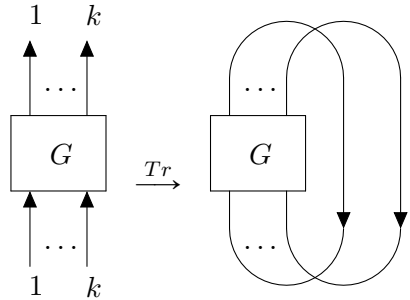
\begin{center}
\begin{tikzpicture}[line cap=round,line join=round,>=triangle 45,x=0.5cm,y=0.5cm]
\clip(-2.5,-4.) rectangle (1.,4.);
\draw [line width=0.4pt] (-2.,1.)-- (0.5,1.);
\draw [line width=0.4pt] (0.5,1.)-- (0.5,-1.);
\draw [line width=0.4pt] (0.5,-1.)-- (-2.,-1.);
\draw [line width=0.4pt] (-2.,-1.)-- (-2.,1.);
\draw [->,line width=0.4pt] (-1.5,1.) -- (-1.5,3.);
\draw [->,line width=0.4pt] (0.,1.) -- (0.,3.);
\draw [->,line width=0.4pt] (-1.5,-3.) -- (-1.5,-1.);
\draw [->,line width=0.4pt] (0.,-3.) -- (0.,-1.);
\draw (-1.25,0.5) node[anchor=north west] {$G$};
\draw (-1.8,-3) node[anchor=north west] {$1$};
\draw (-0.3,-3) node[anchor=north west] {$k$};
\draw (-1.4,-2.2) node[anchor=north west] {$\ldots$};
\draw (-1.8,4.2) node[anchor=north west] {$1$};
\draw (-0.3,4.2) node[anchor=north west] {$k$};
\draw (-1.4,2.) node[anchor=north west] {$\ldots$};
\end{tikzpicture}
$\substack{\displaystyle \stackrel{Tr}{\longrightarrow}\\ \vspace{3cm}}$
\begin{tikzpicture}[line cap=round,line join=round,>=triangle 45,x=0.5cm,y=0.5cm]
\clip(-2.5,-4.) rectangle (3.5,4.);
\draw [line width=0.4pt] (-2.,1.)-- (0.5,1.);
\draw [line width=0.4pt] (0.5,1.)-- (0.5,-1.);
\draw [line width=0.4pt] (0.5,-1.)-- (-2.,-1.);
\draw [line width=0.4pt] (-2.,-1.)-- (-2.,1.);
\draw [line width=0.4pt] (-1.5,1.) -- (-1.5,2.);
\draw [line width=0.4pt] (0.,1.) -- (0.,2.);
\draw [line width=0.4pt] (-1.5,-2.) -- (-1.5,-1.);
\draw [line width=0.4pt] (0.,-2.) -- (0.,-1.);
\draw (-1.25,0.5) node[anchor=north west] {$G$};
\draw (-1.4,-2.2) node[anchor=north west] {$\ldots$};
\draw (-1.4,2.) node[anchor=north west] {$\ldots$};
\draw [shift={(1.5,-2.)},line width=0.4pt]  plot[domain=3.141592653589793:6.283185307179586,variable=\t]({1.*1.5*cos(\t r)+0.*1.5*sin(\t r)},{0.*1.5*cos(\t r)+1.*1.5*sin(\t r)});
\draw [shift={(0.,-2.)},line width=0.4pt]  plot[domain=3.141592653589793:6.283185307179586,variable=\t]({1.*1.5*cos(\t r)+0.*1.5*sin(\t r)},{0.*1.5*cos(\t r)+1.*1.5*sin(\t r)});
\draw [shift={(1.5,2.)},line width=0.4pt]  plot[domain=0.:3.141592653589793,variable=\t]({1.*1.5*cos(\t r)+0.*1.5*sin(\t r)},{0.*1.5*cos(\t r)+1.*1.5*sin(\t r)});
\draw [shift={(0.,2.)},line width=0.4pt]  plot[domain=0.:3.141592653589793,variable=\t]({1.*1.5*cos(\t r)+0.*1.5*sin(\t r)},{0.*1.5*cos(\t r)+1.*1.5*sin(\t r)});
\draw [->,line width=0.4pt] (1.5,2.) -- (1.5,-2.);
\draw [->,line width=0.4pt] (3.,2.) -- (3.,-2.);
\end{tikzpicture}

\vspace{-1.5cm}
\end{center}
In particular, {in\ $\PGr(X)$,} $\grapheo=\mathrm{Tr}(I)$. Note that if $G$ is an solar graph, then $\mathrm{Tr}(G)$ is solar:
$\mathrm{Tr}({\rPGr}(k,k))\subseteq {\rPGr}(0,0)$ for any $k$.
If $G$ is a {corolla ordered} graph (respectively an $X$-decorated graph), then $\mathrm{Tr}(G)$ is {corolla ordered} (respectively $X$-decorated).
This operation can be described in terms of the partial trace maps: if $G\in {\rPGr}(k,k)$, then
\[\mathrm{Tr}(G)=t_{1,1}\circ \ldots \circ t_{k,k}(G)
=t_{1,1}\circ \ldots \circ t_{1,1}(G).\]
Similarly, we can define this generalised trace on any TRAP:

\begin{prop} \label{prop:generalised_traces}
Let $P$ be a TRAP. For any $p\in P(k,k)$, {$k\in \N_0$},  the generalised trace on $P$ is defined as:
\[\mathrm{Tr}_P(p){:=}t_{1,1}\circ \ldots \circ t_{k,k}(p)\in P(0,0).\]
\begin{enumerate}
\item For any $(k,l)\in  \N_0^2$, for any $(p,q)\in P(k,l)\times P(l,k)$,
\begin{align*}
\mathrm{Tr}_P(p\circ q)&=\mathrm{Tr}_P(q\circ p),
\end{align*}
{which justifies the terminology "trace".}
\item For any $(k,l)\in  \N_0^2$, for any $(p,q)\in P(k,k)\times P(l,l)$,
\begin{align*}
\mathrm{Tr}_P(p*q)&=\mathrm{Tr}_P(p)*\mathrm{Tr}_P(q).
\end{align*}\end{enumerate}\end{prop}

\begin{proof}
Let $\alpha:{\rPGr}(P)\longrightarrow P$ be, as before, the unique TRAP morphism which extends the identity map on $P$.
Since $\alpha$ respects the partial trace maps, for any graph $G\in {\rPGr}(P)(k,k)$,
\[\alpha \circ \mathrm{Tr}(G)=\mathrm{Tr}_P\circ \alpha(G).\]
Let $p,q\in P(k,k)$. In $\rGr(P)$, $\mathrm{Tr}(q\circ p)$ and  $\mathrm{Tr}(p\circ q)$ are represented by the graphs
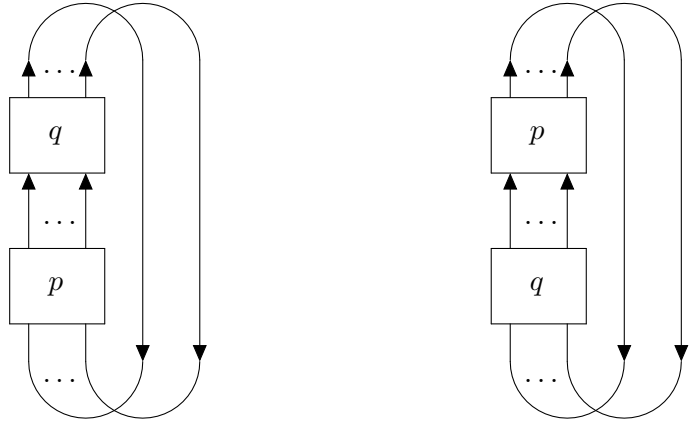
\begin{align*}
&\begin{tikzpicture}[line cap=round,line join=round,>=triangle 45,x=0.5cm,y=0.5cm]
\clip(-2.5,-4.) rectangle (3.5,8.);
\draw [line width=0.4pt] (-2.,1.)-- (0.5,1.);
\draw [line width=0.4pt] (0.5,1.)-- (0.5,-1.);
\draw [line width=0.4pt] (0.5,-1.)-- (-2.,-1.);
\draw [line width=0.4pt] (-2.,-1.)-- (-2.,1.);
\draw [->,line width=0.4pt] (-1.5,1.) -- (-1.5,3.);
\draw [->,line width=0.4pt] (0.,1.) -- (0.,3.);
\draw [line width=0.4pt] (-1.5,-2.) -- (-1.5,-1.);
\draw [line width=0.4pt] (0.,-2.) -- (0.,-1.);
\draw (-1.25,0.5) node[anchor=north west] {$p$};
\draw (-1.4,-2.2) node[anchor=north west] {$\ldots$};
\draw (-1.4,2.) node[anchor=north west] {$\ldots$};
\draw [line width=0.4pt] (-2.,5.)-- (0.5,5.);
\draw [line width=0.4pt] (0.5,5.)-- (0.5,3.);
\draw [line width=0.4pt] (0.5,3.)-- (-2.,3.);
\draw [line width=0.4pt] (-2.,3.)-- (-2.,5.);
\draw [->,line width=0.4pt] (-1.5,5.) -- (-1.5,6.);
\draw [->,line width=0.4pt] (0.,5.) -- (0.,6.);
\draw (-1.25,4.5) node[anchor=north west] {$q$};
\draw (-1.4,6.) node[anchor=north west] {$\ldots$};
\draw [shift={(1.5,-2.)},line width=0.4pt]  plot[domain=3.141592653589793:6.283185307179586,variable=\t]({1.*1.5*cos(\t r)+0.*1.5*sin(\t r)},{0.*1.5*cos(\t r)+1.*1.5*sin(\t r)});
\draw [shift={(0.,-2.)},line width=0.4pt]  plot[domain=3.141592653589793:6.283185307179586,variable=\t]({1.*1.5*cos(\t r)+0.*1.5*sin(\t r)},{0.*1.5*cos(\t r)+1.*1.5*sin(\t r)});
\draw [shift={(1.5,6.)},line width=0.4pt]  plot[domain=0.:3.141592653589793,variable=\t]({1.*1.5*cos(\t r)+0.*1.5*sin(\t r)},{0.*1.5*cos(\t r)+1.*1.5*sin(\t r)});
\draw [shift={(0.,6.)},line width=0.4pt]  plot[domain=0.:3.141592653589793,variable=\t]({1.*1.5*cos(\t r)+0.*1.5*sin(\t r)},{0.*1.5*cos(\t r)+1.*1.5*sin(\t r)});
\draw [->,line width=0.4pt] (1.5,6.) -- (1.5,-2.);
\draw [->,line width=0.4pt] (3.,6.) -- (3.,-2.);
\end{tikzpicture}& 
\begin{tikzpicture}[line cap=round,line join=round,>=triangle 45,x=0.5cm,y=0.5cm]
\clip(-2.5,-4.) rectangle (3.5,8.);
\draw [line width=0.4pt] (-2.,1.)-- (0.5,1.);
\draw [line width=0.4pt] (0.5,1.)-- (0.5,-1.);
\draw [line width=0.4pt] (0.5,-1.)-- (-2.,-1.);
\draw [line width=0.4pt] (-2.,-1.)-- (-2.,1.);
\draw [->,line width=0.4pt] (-1.5,1.) -- (-1.5,3.);
\draw [->,line width=0.4pt] (0.,1.) -- (0.,3.);
\draw [line width=0.4pt] (-1.5,-2.) -- (-1.5,-1.);
\draw [line width=0.4pt] (0.,-2.) -- (0.,-1.);
\draw (-1.25,0.5) node[anchor=north west] {$q$};
\draw (-1.4,-2.2) node[anchor=north west] {$\ldots$};
\draw (-1.4,2.) node[anchor=north west] {$\ldots$};
\draw [line width=0.4pt] (-2.,5.)-- (0.5,5.);
\draw [line width=0.4pt] (0.5,5.)-- (0.5,3.);
\draw [line width=0.4pt] (0.5,3.)-- (-2.,3.);
\draw [line width=0.4pt] (-2.,3.)-- (-2.,5.);
\draw [->,line width=0.4pt] (-1.5,5.) -- (-1.5,6.);
\draw [->,line width=0.4pt] (0.,5.) -- (0.,6.);
\draw (-1.25,4.5) node[anchor=north west] {$p$};
\draw (-1.4,6.) node[anchor=north west] {$\ldots$};
\draw [shift={(1.5,-2.)},line width=0.4pt]  plot[domain=3.141592653589793:6.283185307179586,variable=\t]({1.*1.5*cos(\t r)+0.*1.5*sin(\t r)},{0.*1.5*cos(\t r)+1.*1.5*sin(\t r)});
\draw [shift={(0.,-2.)},line width=0.4pt]  plot[domain=3.141592653589793:6.283185307179586,variable=\t]({1.*1.5*cos(\t r)+0.*1.5*sin(\t r)},{0.*1.5*cos(\t r)+1.*1.5*sin(\t r)});
\draw [shift={(1.5,6.)},line width=0.4pt]  plot[domain=0.:3.141592653589793,variable=\t]({1.*1.5*cos(\t r)+0.*1.5*sin(\t r)},{0.*1.5*cos(\t r)+1.*1.5*sin(\t r)});
\draw [shift={(0.,6.)},line width=0.4pt]  plot[domain=0.:3.141592653589793,variable=\t]({1.*1.5*cos(\t r)+0.*1.5*sin(\t r)},{0.*1.5*cos(\t r)+1.*1.5*sin(\t r)});
\draw [->,line width=0.4pt] (1.5,6.) -- (1.5,-2.);
\draw [->,line width=0.4pt] (3.,6.) -- (3.,-2.);
\end{tikzpicture}
\end{align*}
which are the same. Applying $\alpha$, we obtain $\mathrm{Tr}_P(p\circ q)=\mathrm{Tr}(q\circ p)$. 
Moreover, the graph $\mathrm{Tr}(p*q)$ is represented by 

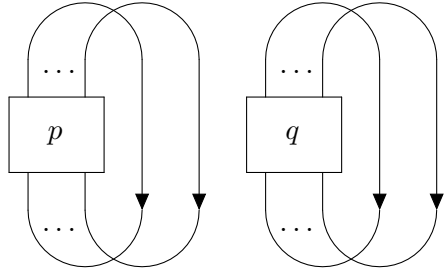
\begin{center}
\begin{tikzpicture}[line cap=round,line join=round,>=triangle 45,x=0.5cm,y=0.5cm]
\clip(-2.5,-4.) rectangle (3.5,4.);
\draw [line width=0.4pt] (-2.,1.)-- (0.5,1.);
\draw [line width=0.4pt] (0.5,1.)-- (0.5,-1.);
\draw [line width=0.4pt] (0.5,-1.)-- (-2.,-1.);
\draw [line width=0.4pt] (-2.,-1.)-- (-2.,1.);
\draw [line width=0.4pt] (-1.5,1.) -- (-1.5,2.);
\draw [line width=0.4pt] (0.,1.) -- (0.,2.);
\draw [line width=0.4pt] (-1.5,-2.) -- (-1.5,-1.);
\draw [line width=0.4pt] (0.,-2.) -- (0.,-1.);
\draw (-1.25,0.5) node[anchor=north west] {$p$};
\draw (-1.4,-2.2) node[anchor=north west] {$\ldots$};
\draw (-1.4,2.) node[anchor=north west] {$\ldots$};
\draw [shift={(1.5,-2.)},line width=0.4pt]  plot[domain=3.141592653589793:6.283185307179586,variable=\t]({1.*1.5*cos(\t r)+0.*1.5*sin(\t r)},{0.*1.5*cos(\t r)+1.*1.5*sin(\t r)});
\draw [shift={(0.,-2.)},line width=0.4pt]  plot[domain=3.141592653589793:6.283185307179586,variable=\t]({1.*1.5*cos(\t r)+0.*1.5*sin(\t r)},{0.*1.5*cos(\t r)+1.*1.5*sin(\t r)});
\draw [shift={(1.5,2.)},line width=0.4pt]  plot[domain=0.:3.141592653589793,variable=\t]({1.*1.5*cos(\t r)+0.*1.5*sin(\t r)},{0.*1.5*cos(\t r)+1.*1.5*sin(\t r)});
\draw [shift={(0.,2.)},line width=0.4pt]  plot[domain=0.:3.141592653589793,variable=\t]({1.*1.5*cos(\t r)+0.*1.5*sin(\t r)},{0.*1.5*cos(\t r)+1.*1.5*sin(\t r)});
\draw [->,line width=0.4pt] (1.5,2.) -- (1.5,-2.);
\draw [->,line width=0.4pt] (3.,2.) -- (3.,-2.);
\end{tikzpicture}
\begin{tikzpicture}[line cap=round,line join=round,>=triangle 45,x=0.5cm,y=0.5cm]
\clip(-2.5,-4.) rectangle (6.,4.);
\draw [line width=0.4pt] (-2.,1.)-- (0.5,1.);
\draw [line width=0.4pt] (0.5,1.)-- (0.5,-1.);
\draw [line width=0.4pt] (0.5,-1.)-- (-2.,-1.);
\draw [line width=0.4pt] (-2.,-1.)-- (-2.,1.);
\draw [line width=0.4pt] (-1.5,1.) -- (-1.5,2.);
\draw [line width=0.4pt] (0.,1.) -- (0.,2.);
\draw [line width=0.4pt] (-1.5,-2.) -- (-1.5,-1.);
\draw [line width=0.4pt] (0.,-2.) -- (0.,-1.);
\draw (-1.25,0.5) node[anchor=north west] {$q$};
\draw (-1.4,-2.2) node[anchor=north west] {$\ldots$};
\draw (-1.4,2.) node[anchor=north west] {$\ldots$};
\draw [shift={(1.5,-2.)},line width=0.4pt]  plot[domain=3.141592653589793:6.283185307179586,variable=\t]({1.*1.5*cos(\t r)+0.*1.5*sin(\t r)},{0.*1.5*cos(\t r)+1.*1.5*sin(\t r)});
\draw [shift={(0.,-2.)},line width=0.4pt]  plot[domain=3.141592653589793:6.283185307179586,variable=\t]({1.*1.5*cos(\t r)+0.*1.5*sin(\t r)},{0.*1.5*cos(\t r)+1.*1.5*sin(\t r)});
\draw [shift={(1.5,2.)},line width=0.4pt]  plot[domain=0.:3.141592653589793,variable=\t]({1.*1.5*cos(\t r)+0.*1.5*sin(\t r)},{0.*1.5*cos(\t r)+1.*1.5*sin(\t r)});
\draw [shift={(0.,2.)},line width=0.4pt]  plot[domain=0.:3.141592653589793,variable=\t]({1.*1.5*cos(\t r)+0.*1.5*sin(\t r)},{0.*1.5*cos(\t r)+1.*1.5*sin(\t r)});
\draw [->,line width=0.4pt] (1.5,2.) -- (1.5,-2.);
\draw [->,line width=0.4pt] (3.,2.) -- (3.,-2.);
\end{tikzpicture}
\end{center}
which is also a graphical representation of $\mathrm{Tr}(p)*\mathrm{Tr}(q)$. Applying $\alpha$,
we obtain $\mathrm{Tr}_P(p*q)=\mathrm{Tr}_P(p)*Tr_P(q)$. 
\end{proof}

\begin{example}
Let $V$ be a finite dimensional vector space and $f=\Theta(v_1\ldots v_k\otimes f_1\ldots f_k) \in \Hom_V^{fr}(k,k)$. 
Identifying $\Hom_V(0,0)$ with $\mathbb{R}$, we obtain that
\[\mathrm{Tr}_{\Hom_V}(f)=f_1(v_1)\ldots f_k(v_k),\]
which is the usual trace of linear endomorphisms of a finite-dimensional vector space.
If $V$ is not finite-dimensional, $\mathrm{Tr}_{\Hom_V^{fr}}$ is a direct generalisation of this trace for linear endomorphisms
of finite rank.  
\end{example}

\subsection{Amplitudes and generalised convolutions}

By Theorem \ref{freetraps} applied to $\phi=\mathrm{Id}_P$,
 we know that for any TRAP $P$ there exists a canonical TRAP map $\Phi_P:\rPGr(P)\longrightarrow P$ (see Remark \ref{rk:PhiPGrX}).
\begin{defi} \label{defi:generalised-P-convolution}
 Let $G$ be a graph decorated by a TRAP $P$. 
 The \textbf{ $P$-amplitude} associated to $G$ is the  image $\Phi_P(G)$  of $G$ under $\Phi_P$.
 
 When $P=\mathcal{K}_M^\infty$ is the TRAP of smooth generalised kernel over a smooth finite dimensional closed Riemannian manifold $M$ of \S \ref{sec:smoothingop} (i.e. $P(k,l):=\mathcal{K}_M^\infty(k,l)$ with the r.h.s defined in \eqref{eq:Klm}), 
 we simply write $\Phi$ for $\Phi_P$ and call $\Phi(G)$   the \textbf{smooth amplitudes} 
 associated to $G\in\rPGr(\mathcal{K}_M^\infty)$.
\end{defi}

\begin{remark}The terminology $P$-amplitude is justified in so far as it associates to a graph an expression in $P$ depending on the ingoing and outgoing edges of the graph in a similar way  as an amplitude  associated to a Feynman depend on the external ingoing and outgoing momenta.\end{remark}
\begin{remark} If we specialise to    spaces $  {\mathcal E}({M})^{\widehat\odot k}\, \hat \otimes \,  {\mathcal E}({M})^{\widehat\odot  l} $ which are symmetric in both sets of input and output variables, then  $\phi$ can be extended to $\rGr(X)$ 
 (see Corollary \ref{cor:extension_identity}).
	\end{remark}

The case of a ladder graph relates amplitudes to convolutions:
\begin{remark} \label{rk:generalised_conv}
	Let $G$ be a ladder graph decorated by $X=(\mathcal{K}_X^\infty(k,l))_{kl,\in  \N_0}$  i.e., a graph such that $I(G)=O(G)=[1]$, 
	$IO(G)=L(G)=\emptyset$,  $V(G)=\{v_1,\cdots,v_n\}$, $E(G)=\{e_1,\cdots,e_{n-1}\}$ and the source and target maps defined by
	\begin{align*}
	&& s_G(1)&=v_n,&t_G(1)&=v_1,\\
	&\forall i\in [n-1],&s_G(e_i)&=v_i,&t_G(e_i)&=v_{i+1}.
	\end{align*}
	Here is a graphical representation of this graph:
	
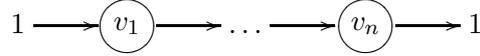
\[\xymatrix{1\ar[r]& \rond{v_1}\ar[r]&\ldots \ar[r]&\rond{v_n}\ar[r]&1}\]
	Let $P_i, i=1, \cdots, n$ be   smoothing pseudo-differential operator each of which is defined by the kernel $K_i$ that decorates  the vertex $v_i$, namely 
	$K_i:=\mathrm{dec}(v_i)$ for any 
	$v_i, i\in[n]$. Then the generalised convolution associated to the graph $G$ is the convolution {$K_1\star\cdots\star K_n$} of the kernels $ {K_1},\cdots,K_n$, which is the 
	kernel of the smoothing pseudo-differential 
	operator $P_1\circ\cdots\circ P_n$. In this sense, $P$-amplitudes can be seen as a  generalisation of the convolution of multiple smooth kernels.
	
	 \end{remark}

    \begin{prop} \label{prop:gen_conv_vert_conc}
     For any TRAP $P$, the $P$-amplitude  associated to a horizontal concatenation of graphs is the horizontal concatenation of their $P$-amplitudes: for any $G_1,G_2\in\rPGr(P)$,
\begin{equation*}
\Phi_P(G_1*G_2)=\Phi_P(G_1)*\Phi_P(G_2),
\end{equation*}     
     and the same holds for the vertical concatenation: if $G_1\circ G_2$ exists, then 
     \begin{equation*}
      \Phi_P(G_1\circ G_2) = \Phi_P(G_1)\circ_P\Phi_P(G_2)
     \end{equation*}
     with $\circ_P$ the vertical concatenation of $P$.
    \end{prop}
    \begin{proof}
     First, notice that since $G_2\circ G_2$ exists and since $\Phi_P$ is a TRAP morphism, $\Phi_P(G_1)\circ\Phi_P(G_2)$ exists. 
     The results then follows directly from the fact that $\Phi_P$ is a TRAP morphism and from Proposition \ref{prop:morphismPROPfromTRAP}.
    \end{proof}
For any TRAP $P$, let $\iota_P:P\longrightarrow\rPGr(P)$ be the canonical embedding of $P$ into the  TRAP of $P$-decorated graphs i.e.,  $\iota_P(p)$ is the solar graph with only one vertex decorated by $p$. We  have the following simple yet useful Lemma.
\begin{lemma}
 For any TRAP $P$ the following diagram commutes:
 \begin{align*}
  & \xymatrix{P\times P~\ar@{^{(}->}[r]
  \ar[d]_{\circ_P} & \rPGr(P)\times\rPGr(P) \ar[d]^{\circ} \\
  P & \ar[l]^{\Phi_P} \rPGr(P)}
 \end{align*}
 with $\circ_p$ the vertical concatenation of the TRAP $P$ and the obvious abuse of notation that vertical concatenations, if seen as maps, are not defined on the whole of $P\times P$.
\end{lemma}
 
  In words,  the horizontal concatenation of two elements $p_1$ $p_2$ of $P$ is the $P$-amplitude associated with the graph given by the vertical concatenation of two graphs with exactly one vertex, each decorated by one $p_i$.
Graphically, if $p\in P(k,l)$ and $q\in P(l,m)$:
\[\Phi_P\left(\substack{\hspace{5mm}\\ 
\begin{tikzpicture}[line cap=round,line join=round,>=triangle 45,x=0.5cm,y=0.5cm]
		\clip(-2.5,-4.) rectangle (0.5,8.);
		\draw [line width=0.4pt] (-2.,1.)-- (0.5,1.);
		\draw [line width=0.4pt] (0.5,1.)-- (0.5,-1.);
		\draw [line width=0.4pt] (0.5,-1.)-- (-2.,-1.);
		\draw [line width=0.4pt] (-2.,-1.)-- (-2.,1.);
		\draw [->,line width=0.4pt] (-1.5,1.) -- (-1.5,3.);
		\draw [->,line width=0.4pt] (0.,1.) -- (0.,3.);
		\draw [->,line width=0.4pt] (-1.5,-3.) -- (-1.5,-1.);
		\draw [->,line width=0.4pt] (0.,-3.) -- (0.,-1.);
		\draw (-1.25,0.5) node[anchor=north west] {$p$};
		\draw (-1.8,-3) node[anchor=north west] {$1$};
		\draw (-0.3,-3) node[anchor=north west] {$k$};
		\draw (-1.4,-2.2) node[anchor=north west] {$\ldots$};
		\draw [line width=0.4pt] (-2.,5.)-- (0.5,5.);
		\draw [line width=0.4pt] (0.5,5.)-- (0.5,3.);
		\draw [line width=0.4pt] (0.5,3.)-- (-2.,3.);
		\draw [line width=0.4pt] (-2.,3.)-- (-2.,5.);
		\draw [->,line width=0.4pt] (-1.5,5.) -- (-1.5,7.);
		\draw [->,line width=0.4pt] (0.,5.) -- (0.,7.);
		\draw (-1.25,4.5) node[anchor=north west] {$q$};
		\draw (-1.8,8.2) node[anchor=north west] {$1$};
		\draw (-0.4,8.1) node[anchor=north west] {$m$};
		\draw (-1.4,6.) node[anchor=north west] {$\ldots$};
		\end{tikzpicture}}\hspace{3mm}\right)
=\Phi_P(p)\circ_P \Phi_P(q).\]
\begin{proof}
 Let $P$ be a TRAP. Then for any $p_1$, $p_2$ in $P$ such that $p_1\circ_P p_2$ is well defined, $\iota_P(p_1)\circ\iota_P(p_2)$ is well-defined since $\iota_P$ respects the gradings and we have
 \begin{align*}
  \Phi_P(\iota_P(p_1)\circ\iota_P(p_2)) & = \Phi_P(\iota_P(p_1))\circ_P\Phi(\iota_P(p_2)) \quad \text{by Proposition \ref{prop:gen_conv_vert_conc}} \\
  & = p_1\circ_P p_2
 \end{align*}
 since for any TRAP $P$, $\Phi_P\circ\iota_P = \mathrm{Id}_P$ by definition of $\Phi_P$ (Equation \eqref{eq:G_simple_solar} with $k=1$ and $\phi=\mathrm{Id}_P$).
\end{proof}

 \begin{remark} Note that the vertical concatenation is \emph{not} the same as the $P$-amplitude: 
the latter has a much larger domain. 
\end{remark}

Applying the above constructions to the TRAP of smooth kernels  described in Theorem \ref{theo:Kinfty} whose  partial traces  \ref{eq:trconvK} are given by integrations on the underlying manifold,  easily yields the following statement (we use the notations of 
Paragraph \ref{sec:smoothingop}{: $M$ is a smooth, finite dimensional orientable closed manifold and $\mu(z)$ is a volume form on $M$}).
	
\begin{theo} \label{thm:generalised_convolution}
  The family $\left(X_{k,l}:=\mathcal{K}_M^\infty(k,l)\right)_{k, l\in \N_0}$ is a TRAP and:
  \begin{enumerate}
   \item The vertical concatenation  of two kernels corresponds to their  \textbf{generalised  convolution}: 
   \begin{align*}
&\forall (k,l,m)\in \N_0^3,\:\forall K_1\in {\mathcal K}_{M}^\infty(k,l),\:\forall K_2\in {\mathcal K}_{M}^\infty(l,m), \forall (x_1, \cdots, x_k, z_1, \cdots, z_m)\in M^{k+m} ,\\
&K_2\circ K_1(x_1,  \cdots, x_k, z_1, \cdots, z_m)\\
&=t_{k+1,1}\circ \ldots \circ t_{k+l-1,l-1}\circ t_{k+l,l}(K_1\otimes K_2)(x_1,  \cdots, x_k, z_1, \cdots, z_m)\\
& = \int_{M^l} K_1(x_1, \cdots, x_k, y_1, \cdots, y_l)\, K_2(y_1,  \cdots, y_l,  z_1, \cdots, z_m)\, {d\mu(y_1)\, \cdots d\mu(y_l)},
\end{align*} 
obtained  by integrating along the diagonal $\Delta_M^l:=\{(y_1,\cdots, y_l, y_1, \cdots, y_l), y_i\in M\}\subset M^{2l}$. 

   \item The associativity property $K_3\circ (K_2\circ K_1)= (K_3\circ K_2)\circ K_1$ (cfr. (\ref{eq:assovertconcTRAP})) for any   $K_3\in {\mathcal K}_{M}^\infty(m,n)$,
     amounts to  the Fubini property for the corresponding multiple integrals{:
     \begin{equation*}
      \int_{M^m} \left(\int_{M^l} K_1(\vec{x},\vec{y}_1)K_2(\vec{y}_1,\vec{y}_2)d\vec{\mu}(\vec{y}_1)\right)K_3(\vec{y}_2,\vec{z})d\vec{\mu}(\vec{y}_2) = \int_{M^l} K_1(\vec{x},\vec{y}_1)\left(\int_{M^m} K_2(\vec{y}_1,\vec{y}_2)K_3(\vec{y}_2,\vec{z})d\vec{\mu}(\vec{y}_2)\right) d\vec{\mu}(\vec{y}_1)
     \end{equation*}
     for any $\vec{x}\in M^k$ and $\vec{z}\in M^n$, where we use the short-hand notations $d\vec{\mu}(\vec{y}_i):=d\mu(y_1)\cdots d\mu(y_{l_i})$.}

     \item The generalised trace  of a generalised kernel $K\in \mathcal{K}_M^\infty(k,k)$ is given by the integral along the small diagonal of $M^k$:
   \[\mathrm{Tr}_{{K}^\infty}(K)=\int_{M^k}K(x_1, \cdots, x_k, x_1, \cdots, x_k)\, dx_1\cdots dx_k   \]
and obeys the following cyclicity property:
   \[\mathrm{Tr}_{{K}^\infty}(K_2\circ K_1)= \mathrm{Tr}_{{K}^\infty}(K_1\circ K_2)\]
   for $K_1\in\mathcal{K}_M^\infty(k,l)$ and $K_2\in\mathcal{K}_M^\infty(l,k)$.
   
   \item The $\mathcal{K}_M^\infty$-amplitude is compatible with the horizontal and vertical concatenations
   in $\mathcal{K}_M^\infty$.
    \end{enumerate}
   
\end{theo}

\section{Categorical interpretation} \label{section:monad}

We describe the categories of (resp. unitary) TRAPs as (resp. $\rGacirc$-) $\Gacirc$-algebras, where  $\Gacirc$ is an endofunctor on the category of  {$\sym\times \sym^{op}-$} modules,  which sends a (resp. unitary) TRAP  $X$ to the free (resp. unitary) TRAP of graphs  (resp. $\rGacirc(X)$) $\Gacirc(X)$ generated by $X$, thus extending known results of \cite{JY15} on the categorial aspects of wheeled PROPs.

\subsection{Two endofunctors in the category of \texorpdfstring{$\sym\times \sym^{op}$}-modules} 

 We  consider graphs decorated by a $\sym\times \sym^{op}$-module $X=(X(k,l))_{k,l\geqslant 0}$ 
 and use the action of the symmetric groups on the vertices of Definition \ref{defiactionsommets} to define
an endofunctor $\Gacirc$.
 \begin{defi}
We define a relation on $\PGr(X)(k,l)$ by  $(G,d_G)\mathcal{R}_{k,l} (G',d_{G'})$ for $(G,d_G)$ and $(G',d_{G'})$  if there exists  a vertex $v$ of $G$ and permutations
 $\sigma \in \sym_{o(v)}$, $\tau\in \sym_{i(v)}$ such that
  \begin{equation*}
   \sigma\cdot_v(G,d_G)\cdot_v\tau {=} (G,d_G^{\sigma,v,\tau})
  \end{equation*}
with
  \begin{align*}
   d_G^{\sigma,v,\tau}(v') = \begin{cases}
                              & d_G(v')\quad\text{for }v'\neq v \\
                              & \sigma\cdot d_G(v)\cdot\tau \mbox{ otherwise}.
                             \end{cases}
  \end{align*}
  We denote by $\sim_{k,l}$ the transitive closure of $\mathcal{R}_{k,l}$ which defines an equivalence. We further define
  \begin{align*}
     \Gacirc(X)(k,l)&:=\dfrac{\PGr(X)(k,l)}{\sim_{k,l}},&
          \rGacirc(X)(k,l)&:=\dfrac{\rPGr(X)(k,l)}{\sim_{k,l}}.
  \end{align*}
  \end{defi}

Here is the type of relations we obtain graphically:
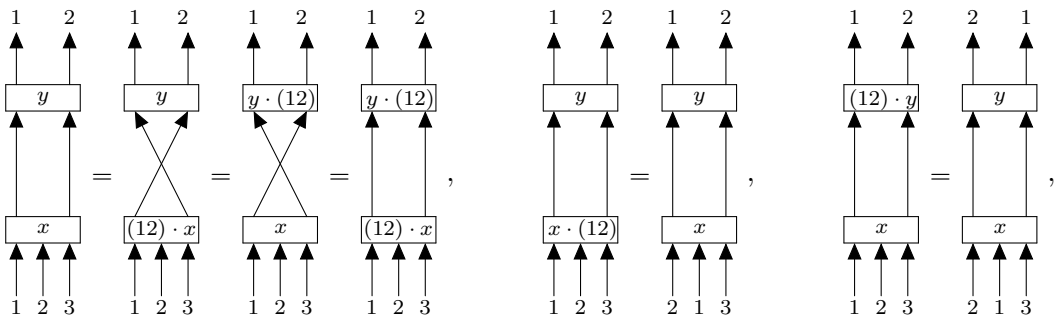
\begin{align*}
\begin{tikzpicture}[line cap=round,line join=round,>=triangle 45,x=0.7cm,y=0.7cm]
\clip(0.6,-2.1) rectangle (2.4,4.5);
\draw [line width=.4pt] (0.8,0.)-- (2.2,0.);
\draw [line width=.4pt] (2.2,0.)-- (2.2,-0.5);
\draw [line width=.4pt] (2.2,-0.5)-- (0.8,-0.5);
\draw [line width=.4pt] (0.8,-0.5)-- (0.8,0.);
\draw [line width=.4pt] (0.8,2.)-- (2.2,2.);
\draw [line width=.4pt] (2.2,2.)-- (2.2,2.5);
\draw [line width=.4pt] (2.2,2.5)-- (0.8,2.5);
\draw [line width=.4pt] (0.8,2.5)-- (0.8,2.);
\draw (1.2,0.05) node[anchor=north west] {\scriptsize $x$};
\draw (1.2,2.55) node[anchor=north west] {\scriptsize $y$};
\draw [->,line width=.4pt] (1.,0.) -- (1.,2.);
\draw [->,line width=.4pt] (2.,0.) -- (2.,2.);
\draw [->,line width=.4pt] (1.,-1.5) -- (1.,-0.5);
\draw [->,line width=.4pt] (1.5,-1.5) -- (1.5,-0.5);
\draw [->,line width=.4pt] (2.,-1.5) -- (2.,-0.5);
\draw (0.7,-1.4) node[anchor=north west] {\scriptsize $1$};
\draw (1.2,-1.4) node[anchor=north west] {\scriptsize $2$};
\draw (1.7,-1.4) node[anchor=north west] {\scriptsize $3$};
\draw [->,line width=.4pt] (1.,2.5) -- (1,3.5);
\draw [->,line width=.4pt] (2.,2.5) -- (2.,3.5);
\draw (0.7,4.1) node[anchor=north west] {\scriptsize $1$};
\draw (1.7,4.1) node[anchor=north west] {\scriptsize $2$};
\end{tikzpicture}&
\substack{\displaystyle =\\ \vspace{3.5cm}}\begin{tikzpicture}[line cap=round,line join=round,>=triangle 45,x=0.7cm,y=0.7cm]
\clip(0.6,-2.1) rectangle (2.4,4.5);
\draw [line width=.4pt] (0.8,0.)-- (2.2,0.);
\draw [line width=.4pt] (2.2,0.)-- (2.2,-0.5);
\draw [line width=.4pt] (2.2,-0.5)-- (0.8,-0.5);
\draw [line width=.4pt] (0.8,-0.5)-- (0.8,0.);
\draw [line width=.4pt] (0.8,2.)-- (2.2,2.);
\draw [line width=.4pt] (2.2,2.)-- (2.2,2.5);
\draw [line width=.4pt] (2.2,2.5)-- (0.8,2.5);
\draw [line width=.4pt] (0.8,2.5)-- (0.8,2.);
\draw (0.65,0.15) node[anchor=north west] {\scriptsize $(12)\cdot x$};
\draw (1.2,2.55) node[anchor=north west] {\scriptsize $y$};
\draw [->,line width=.4pt] (1.,0.) -- (2.,2.);
\draw [->,line width=.4pt] (2.,0.) -- (1.,2.);
\draw [->,line width=.4pt] (1.,-1.5) -- (1.,-0.5);
\draw [->,line width=.4pt] (1.5,-1.5) -- (1.5,-0.5);
\draw [->,line width=.4pt] (2.,-1.5) -- (2.,-0.5);
\draw (0.7,-1.4) node[anchor=north west] {\scriptsize $1$};
\draw (1.2,-1.4) node[anchor=north west] {\scriptsize $2$};
\draw (1.7,-1.4) node[anchor=north west] {\scriptsize $3$};
\draw [->,line width=.4pt] (1.,2.5) -- (1,3.5);
\draw [->,line width=.4pt] (2.,2.5) -- (2.,3.5);
\draw (0.7,4.1) node[anchor=north west] {\scriptsize $1$};
\draw (1.6,4.1) node[anchor=north west] {\scriptsize $2$};
\end{tikzpicture}
\substack{\displaystyle =\\ \vspace{3.5cm}}\begin{tikzpicture}[line cap=round,line join=round,>=triangle 45,x=0.7cm,y=0.7cm]
\clip(0.6,-2.1) rectangle (2.4,4.5);
\draw [line width=.4pt] (0.8,0.)-- (2.2,0.);
\draw [line width=.4pt] (2.2,0.)-- (2.2,-0.5);
\draw [line width=.4pt] (2.2,-0.5)-- (0.8,-0.5);
\draw [line width=.4pt] (0.8,-0.5)-- (0.8,0.);
\draw [line width=.4pt] (0.8,2.)-- (2.2,2.);
\draw [line width=.4pt] (2.2,2.)-- (2.2,2.5);
\draw [line width=.4pt] (2.2,2.5)-- (0.8,2.5);
\draw [line width=.4pt] (0.8,2.5)-- (0.8,2.);
\draw (1.2,0.05) node[anchor=north west] {\scriptsize $x$};
\draw (0.7,2.6) node[anchor=north west] {\scriptsize $y\cdot (12)$};
\draw [->,line width=.4pt] (1.,0.) -- (2.,2.);
\draw [->,line width=.4pt] (2.,0.) -- (1.,2.);
\draw [->,line width=.4pt] (1.,-1.5) -- (1.,-0.5);
\draw [->,line width=.4pt] (1.5,-1.5) -- (1.5,-0.5);
\draw [->,line width=.4pt] (2.,-1.5) -- (2.,-0.5);
\draw (0.7,-1.4) node[anchor=north west] {\scriptsize $1$};
\draw (1.2,-1.4) node[anchor=north west] {\scriptsize $2$};
\draw (1.7,-1.4) node[anchor=north west] {\scriptsize $3$};
\draw [->,line width=.4pt] (1.,2.5) -- (1,3.5);
\draw [->,line width=.4pt] (2.,2.5) -- (2.,3.5);
\draw (0.7,4.1) node[anchor=north west] {\scriptsize $1$};
\draw (1.6,4.1) node[anchor=north west] {\scriptsize $2$};
\end{tikzpicture}
\substack{\displaystyle =\\ \vspace{3.5cm}}\begin{tikzpicture}[line cap=round,line join=round,>=triangle 45,x=0.7cm,y=0.7cm]
\clip(0.6,-2.1) rectangle (2.4,4.5);
\draw [line width=.4pt] (0.8,0.)-- (2.2,0.);
\draw [line width=.4pt] (2.2,0.)-- (2.2,-0.5);
\draw [line width=.4pt] (2.2,-0.5)-- (0.8,-0.5);
\draw [line width=.4pt] (0.8,-0.5)-- (0.8,0.);
\draw [line width=.4pt] (0.8,2.)-- (2.2,2.);
\draw [line width=.4pt] (2.2,2.)-- (2.2,2.5);
\draw [line width=.4pt] (2.2,2.5)-- (0.8,2.5);
\draw [line width=.4pt] (0.8,2.5)-- (0.8,2.);
\draw (0.65,0.15) node[anchor=north west] {\scriptsize $(12)\cdot x$};
\draw (0.7,2.6) node[anchor=north west] {\scriptsize $y\cdot (12)$};
\draw [->,line width=.4pt] (1.,0.) -- (1.,2.);
\draw [->,line width=.4pt] (2.,0.) -- (2.,2.);
\draw [->,line width=.4pt] (1.,-1.5) -- (1.,-0.5);
\draw [->,line width=.4pt] (1.5,-1.5) -- (1.5,-0.5);
\draw [->,line width=.4pt] (2.,-1.5) -- (2.,-0.5);
\draw (0.7,-1.4) node[anchor=north west] {\scriptsize $1$};
\draw (1.2,-1.4) node[anchor=north west] {\scriptsize $2$};
\draw (1.7,-1.4) node[anchor=north west] {\scriptsize $3$};
\draw [->,line width=.4pt] (1.,2.5) -- (1,3.5);
\draw [->,line width=.4pt] (2.,2.5) -- (2.,3.5);
\draw (0.7,4.1) node[anchor=north west] {\scriptsize $1$};
\draw (1.6,4.1) node[anchor=north west] {\scriptsize $2$};
\end{tikzpicture}\substack{\displaystyle ,\\ \vspace{3.5cm}}&
\begin{tikzpicture}[line cap=round,line join=round,>=triangle 45,x=0.7cm,y=0.7cm]
\clip(0.6,-2.1) rectangle (2.4,4.5);
\draw [line width=.4pt] (0.8,0.)-- (2.2,0.);
\draw [line width=.4pt] (2.2,0.)-- (2.2,-0.5);
\draw [line width=.4pt] (2.2,-0.5)-- (0.8,-0.5);
\draw [line width=.4pt] (0.8,-0.5)-- (0.8,0.);
\draw [line width=.4pt] (0.8,2.)-- (2.2,2.);
\draw [line width=.4pt] (2.2,2.)-- (2.2,2.5);
\draw [line width=.4pt] (2.2,2.5)-- (0.8,2.5);
\draw [line width=.4pt] (0.8,2.5)-- (0.8,2.);
\draw (0.7,0.15) node[anchor=north west] {\scriptsize $x\cdot (12)$};
\draw (1.2,2.55) node[anchor=north west] {\scriptsize $y$};
\draw [->,line width=.4pt] (1.,0.) -- (1.,2.);
\draw [->,line width=.4pt] (2.,0.) -- (2.,2.);
\draw [->,line width=.4pt] (1.,-1.5) -- (1.,-0.5);
\draw [->,line width=.4pt] (1.5,-1.5) -- (1.5,-0.5);
\draw [->,line width=.4pt] (2.,-1.5) -- (2.,-0.5);
\draw (0.7,-1.4) node[anchor=north west] {\scriptsize $1$};
\draw (1.2,-1.4) node[anchor=north west] {\scriptsize $2$};
\draw (1.7,-1.4) node[anchor=north west] {\scriptsize $3$};
\draw [->,line width=.4pt] (1.,2.5) -- (1,3.5);
\draw [->,line width=.4pt] (2.,2.5) -- (2.,3.5);
\draw (0.7,4.1) node[anchor=north west] {\scriptsize $1$};
\draw (1.7,4.1) node[anchor=north west] {\scriptsize $2$};
\end{tikzpicture}&
\substack{\displaystyle =\\ \vspace{3.5cm}}\begin{tikzpicture}[line cap=round,line join=round,>=triangle 45,x=0.7cm,y=0.7cm]
\clip(0.6,-2.1) rectangle (2.4,4.5);
\draw [line width=.4pt] (0.8,0.)-- (2.2,0.);
\draw [line width=.4pt] (2.2,0.)-- (2.2,-0.5);
\draw [line width=.4pt] (2.2,-0.5)-- (0.8,-0.5);
\draw [line width=.4pt] (0.8,-0.5)-- (0.8,0.);
\draw [line width=.4pt] (0.8,2.)-- (2.2,2.);
\draw [line width=.4pt] (2.2,2.)-- (2.2,2.5);
\draw [line width=.4pt] (2.2,2.5)-- (0.8,2.5);
\draw [line width=.4pt] (0.8,2.5)-- (0.8,2.);
\draw (1.2,0.05) node[anchor=north west] {\scriptsize $x$};
\draw (1.2,2.55) node[anchor=north west] {\scriptsize $y$};
\draw [->,line width=.4pt] (1.,0.) -- (1.,2.);
\draw [->,line width=.4pt] (2.,0.) -- (2.,2.);
\draw [->,line width=.4pt] (1.,-1.5) -- (1.,-0.5);
\draw [->,line width=.4pt] (1.5,-1.5) -- (1.5,-0.5);
\draw [->,line width=.4pt] (2.,-1.5) -- (2.,-0.5);
\draw (0.7,-1.4) node[anchor=north west] {\scriptsize $2$};
\draw (1.2,-1.4) node[anchor=north west] {\scriptsize $1$};
\draw (1.7,-1.4) node[anchor=north west] {\scriptsize $3$};
\draw [->,line width=.4pt] (1.,2.5) -- (1,3.5);
\draw [->,line width=.4pt] (2.,2.5) -- (2.,3.5);
\draw (0.7,4.1) node[anchor=north west] {\scriptsize $1$};
\draw (1.7,4.1) node[anchor=north west] {\scriptsize $2$};
\end{tikzpicture}\substack{\displaystyle ,\\ \vspace{3.5cm}}&
\begin{tikzpicture}[line cap=round,line join=round,>=triangle 45,x=0.7cm,y=0.7cm]
\clip(0.6,-2.1) rectangle (2.4,4.5);
\draw [line width=.4pt] (0.8,0.)-- (2.2,0.);
\draw [line width=.4pt] (2.2,0.)-- (2.2,-0.5);
\draw [line width=.4pt] (2.2,-0.5)-- (0.8,-0.5);
\draw [line width=.4pt] (0.8,-0.5)-- (0.8,0.);
\draw [line width=.4pt] (0.8,2.)-- (2.2,2.);
\draw [line width=.4pt] (2.2,2.)-- (2.2,2.5);
\draw [line width=.4pt] (2.2,2.5)-- (0.8,2.5);
\draw [line width=.4pt] (0.8,2.5)-- (0.8,2.);
\draw (1.2,0.05) node[anchor=north west] {\scriptsize $x$};
\draw (0.7,2.6) node[anchor=north west] {\scriptsize $(12)\cdot y$};
\draw [->,line width=.4pt] (1.,0.) -- (1.,2.);
\draw [->,line width=.4pt] (2.,0.) -- (2.,2.);
\draw [->,line width=.4pt] (1.,-1.5) -- (1.,-0.5);
\draw [->,line width=.4pt] (1.5,-1.5) -- (1.5,-0.5);
\draw [->,line width=.4pt] (2.,-1.5) -- (2.,-0.5);
\draw (0.7,-1.4) node[anchor=north west] {\scriptsize $1$};
\draw (1.2,-1.4) node[anchor=north west] {\scriptsize $2$};
\draw (1.7,-1.4) node[anchor=north west] {\scriptsize $3$};
\draw [->,line width=.4pt] (1.,2.5) -- (1,3.5);
\draw [->,line width=.4pt] (2.,2.5) -- (2.,3.5);
\draw (0.7,4.1) node[anchor=north west] {\scriptsize $1$};
\draw (1.7,4.1) node[anchor=north west] {\scriptsize $2$};
\end{tikzpicture}&
\substack{\displaystyle =\\ \vspace{3.5cm}}\begin{tikzpicture}[line cap=round,line join=round,>=triangle 45,x=0.7cm,y=0.7cm]
\clip(0.6,-2.1) rectangle (2.4,4.5);
\draw [line width=.4pt] (0.8,0.)-- (2.2,0.);
\draw [line width=.4pt] (2.2,0.)-- (2.2,-0.5);
\draw [line width=.4pt] (2.2,-0.5)-- (0.8,-0.5);
\draw [line width=.4pt] (0.8,-0.5)-- (0.8,0.);
\draw [line width=.4pt] (0.8,2.)-- (2.2,2.);
\draw [line width=.4pt] (2.2,2.)-- (2.2,2.5);
\draw [line width=.4pt] (2.2,2.5)-- (0.8,2.5);
\draw [line width=.4pt] (0.8,2.5)-- (0.8,2.);
\draw (1.2,0.05) node[anchor=north west] {\scriptsize $x$};
\draw (1.2,2.55) node[anchor=north west] {\scriptsize $y$};
\draw [->,line width=.4pt] (1.,0.) -- (1.,2.);
\draw [->,line width=.4pt] (2.,0.) -- (2.,2.);
\draw [->,line width=.4pt] (1.,-1.5) -- (1.,-0.5);
\draw [->,line width=.4pt] (1.5,-1.5) -- (1.5,-0.5);
\draw [->,line width=.4pt] (2.,-1.5) -- (2.,-0.5);
\draw (0.7,-1.4) node[anchor=north west] {\scriptsize $2$};
\draw (1.2,-1.4) node[anchor=north west] {\scriptsize $1$};
\draw (1.7,-1.4) node[anchor=north west] {\scriptsize $3$};
\draw [->,line width=.4pt] (1.,2.5) -- (1,3.5);
\draw [->,line width=.4pt] (2.,2.5) -- (2.,3.5);
\draw (0.7,4.1) node[anchor=north west] {\scriptsize $2$};
\draw (1.7,4.1) node[anchor=north west] {\scriptsize $1$};
\end{tikzpicture}\substack{\displaystyle ,\\ \vspace{3.5cm}}
\end{align*}
\vspace{-2cm} 

\noindent where $x\in X_{3,2}$ and $y\in X_{2,2}$. \vskip 0,2cm

It is easy to show that the family of equivalences $\sim=(\sim_{k,l})_{(k,l)\in \N_0^2}$ 
{is compatible with} the action of $\sym\times \sym^{op}$, 
the partial trace maps and the horizontal concatenation in the sense of Lemma \ref{lem:relation_respection_TRAP}. The subsequent useful statement follows from  Lemma \ref{lem:relation_respection_TRAP}.
\begin{lemma}\label{lem:Gamma}  
Let $X$ be a $\sym\times\sym^{op}$-module. $\Gacirc{(X)}$ is a unitary TRAP
and $\rGacirc{(X)}$ is a TRAP. 
\end{lemma}
 Moreover, we know  from Theorem \ref{freetraps} that $\rPGr(X)$ is the free TRAP generated by the set $X$.
	Now, by its very definition $\rGacirc(X)$ is the free  TRAP
generated by the $\sym\times \sym^{op}$-module $X$: 
\begin{prop}\label{prop:freerGamma}
Let $X$ be $\sym\times \sym^{op}$ module, $P$ be a TRAP and $\phi:X\longrightarrow P$
be a morphism of $\sym\times \sym^{op}$-modules. There exists a {unique extension  $\phi$ to a} TRAP morphism $\Phi:\rGacirc(X)\longrightarrow P$. If moreover $P$ is unitary, this morphism   extends to $\Gacirc(X)$. 
\end{prop}

\begin{example}\label{ex:GammaGr}
If $X$ is a trivial $\sym\times \sym^{op}$-module, then $\Gacirc(X)=\Gr(X)$ and $\rGacirc(X)=\rGr(X)$ as TRAPs.
More generally, choosing for any graph $G$ a {corolla orientation} $\overline{G}$, we can prove that for any $(k,l)\in \N_0^2$,
the sets $\Gacirc(X)(k,l)$ and $\Gr(X)(k,l)$ are in bijection, as well as $\rGacirc(X)(k,l)$ and $\rGr(X)(k,l)$ 
(but not in a canonical way), through the map sending the equivalence class of $\overline{G}$ to $G$.
\end{example}
  The correspondence $P\to\Gacirc(P)$ defined above  induces an endofunctor in the category $\catssm$ of $\sym\times \sym^{op}$-modules which we now introduce.
  \begin{defi} \label{deficatssm}
   Let $\catssm$ denote the category of $\sym\times\sym^{op}$-modules: its objects are families $P=(P(k,l))_{(k,l)\in \N_0^2}$, such that for any $(k,l)\in \N_0^2$, $P(k,l)$ is a $\sym_l\times \sym_k^{op}$-module; a morphism $\phi:P\longrightarrow Q$ is a family $(\phi(k,l))_{(k,l)\in \N^2}$, where for any $(k,l)\in \N_0^2$, $\phi(k,l):P(k,l)\longrightarrow Q(k,l)$ is a morphism of $\sym_l\times \sym_k^{op}$-modules. 
 \end{defi}
 We can now state the main result of this subsection:
 \begin{prop} \label{prop:two_endofunctors} 
  The   correspondences  $\rGacirc,\Gacirc:\catssm\longrightarrow\catssm$ define endofunctors of the category   $\catssm$.
 \end{prop}
 \begin{proof}
  We have already seen that the correspondences $P\to \rGacirc(P)$ and $P\to\Gacirc(P)$ send $\sym\times\sym^{op}$-modules to $\sym\times\sym^{op}$-modules. So we only need to specify the action of the functors on morphisms of $\sym\times\sym^{op}$-modules:
 
  Let $X$ and $Y$ be two $\sym\times \sym^{op}$-modules and  $\varphi:X\longrightarrow Y$ a morphism of $\sym\times \sym^{op}$-modules and let $\rPGr(\varphi):\rPGr(P)\longrightarrow \rPGr(Q)$ (resp. $\PGr(\varphi):\PGr(P)\longrightarrow \PGr(Q)$) be its pullback defined by 
  \begin{equation}\label{eq:GRcpullback}
   \rPGr(\varphi)(G,d_G):=(G,\varphi \circ d_G)
  \end{equation} 
  (resp. $\PGr(\varphi)(G,d_G):=(G,\varphi \circ d_G)$) for any $G\in \rPGr(P)$ (resp. in $\PGr(P)$).
  
  Is is easy to check that the induced morphisms $ \rGacirc(\varphi):\rGacirc(P)\longrightarrow \rGacirc(Q)$ and $\Gacirc(\varphi):\Gacirc(P)\longrightarrow \Gacirc(Q)$ are indeed morphisms of $\sym\times\sym^{op}$-modules turning $\rGacirc$ and $\Gacirc$ {into endofunctors of $\catssm$}.
 \end{proof}

 \subsection{Monad of graphs}
 
 We now endow the functor $\Gacirc$  with a monad structure. 
 
 We first recall  basic definitions of category theory. In particular,  for two functors $F,G:\mathcal{C}\longrightarrow\mathcal{D}$, a 
 \textbf{natural transformation} $\eta:F\longrightarrow G$ between these two functors  is given by maps $\eta_X:F(X)\longmapsto G(X)$ for each object $X$ of $\mathcal{C}$ such that for any pair of objects $X,Y\in\Obj(\calC)$ and morphism $f:X\longmapsto Y\in \Mor(\calC)$ the following diagram commutes:
 \begin{align} \label{axiomes_natur_tsfm}
  & \xymatrix{F(X)\ar[r]^{F(f)}\ar[d]_{\eta_X} & F(Y)\ar[d]^{\eta_Y}\\
    G(X)\ar[r]_{G(f)}&G(Y).}
 \end{align}
Let us now introduce the structure of monad, a terminology we borrow from \cite{ML71}. A monad is the categorical equivalent of monoids.
\begin{defi} \label{defimonades}
 A \textbf{monad} (also called a triple) on a category $\mathcal C$ is   given by  an endofunctor $\Gamma\in \End(\cat)$ and two natural transformations $\mu:\Gamma\circ \Gamma\longrightarrow \Gamma$ and $\nu:\mathrm{Id}_{\cat}\longrightarrow\Gamma$ which form an associative and unital monoid $(\Gamma, \mu, \nu)$ in the unital monoid\footnote{The terminology monoid is used in this definition with an obvious abuse of vocabulary since $\Gamma$ and $\End(\cat)$ are not necessarily sets.} $\End(\cat)$ of endofunctors of $\cat$, where  are two natural transformations. This means that the multiplication $\mu:\Gamma\circ \Gamma\longrightarrow \Gamma$ and the unit morphism $\nu:\mathrm{Id}_\cat\longrightarrow \Gamma$ should satisfy the axioms given by commutativity of the diagrams below for any object $P$ of the category $\cat$. 
 \begin{align} \label{axiomesmonade}
  &\xymatrix{\Gamma\circ \Gamma\circ \Gamma(P)\ar[r]^{\Gamma(\mu_P)}\ar[d]_{\mu_{\Gamma(P)}}
  &\Gamma\circ \Gamma(P)\ar[d]^{\mu_P}\\
   \Gamma\circ \Gamma(P)\ar[r]_{\mu_P}&\Gamma(E)}&
   \xymatrix{\Gamma(P)\ar[r]^{\Gamma(\nu_P)}\ar[rd]_{\mathrm{Id}_\cat}&\Gamma\circ \Gamma(P)\ar[d]_{\mu_P}&
   \Gamma(P)\ar[l]_{\nu_{\Gamma(P)}}\ar[ld]^{\mathrm{Id}_\cat}\\
  &\Gamma(P)&}
 \end{align}
\end{defi}
We now recall the notion of  $\Gamma$-algebra  (see e.g. \cite[Definition 2.1.4]{Merkulov2009}).
\begin{defi} Let $\cat$ be a category. An   algebra over a monad  $\Gamma\in\End(\cat)$ or a \textbf{$\Gamma$-algebra} is an object $P$ of $\cat$ together with a
	structure morphism $\alpha: \Gamma(P)\to P$  such that
the following diagrams commute:
\begin{align}
\label{axiomesalgebres}
&\xymatrix{\Gamma\circ \Gamma(P)\ar[r]^{\Gamma(\alpha)} \ar[d]_{\mu_P}\ar[d]_{\mu_P}&\Gamma(P) \ar[d]^\alpha\\
\Gamma(P)\ar[r]_\alpha&P}&
&\xymatrix{P\ar[r]^{\nu_P}\ar[d]_{\mathrm{Id}}&\Gamma(P)\ar[ld]^{\alpha}\\P&}
\end{align}
Let $(P,\alpha)$ and $(Q,\beta)$ be two algebras over a fixed monad $\Gamma$. A morphism of $\Gamma$-algebras 
from $P$ to $Q$ is a morphism $\phi:P\longrightarrow Q$ in the category $\cat$ such that the following diagram commutes:
\begin{equation} \label{eq:defi_alg_monad}
 \xymatrix{\Gamma(P)\ar[d]_{\Gamma(\phi)} \ar[r]^\alpha&P\ar[d]^\phi\\
\Gamma(Q)\ar[r]_\beta&Q}
\end{equation}
\end{defi} 
 We now define the natural transformations $\nu$ and $\mu$   in the case  $\cat=\catssm$ and $\Gamma=\Gacirc$. In this case, for any $\sym\times \sym^{op}$-module $P$, elements of $\Gacirc\circ \Gacirc(P)$ are graphs $G$  whose vertices $v$ are decorated by graphs $G_v$, consistently with  the number of incoming and outgoing edges.
 \begin{defi}
  \begin{enumerate}
   \item For any $\sym\times \sym^{op}$-module $P$, let $\eta_P:P\longrightarrow\Gacirc(P)$ be the morphism of $\sym\times \sym^{op}$-modules which sends an element $p\in P(k,l)$ to the class of the graph {$G(k,l)(p)$} with one vertex $v$  decorated by $p$, $k$ incoming edges indexed from left to right by $1,\ldots,k$ and  $l$ outgoing edges indexed from left to right by $1,\ldots,l$.
   \item For any $\sym\times \sym^{op}$-module $P$, let $\mu_P:\Gacirc\circ \Gacirc(P)\longrightarrow \Gacirc(P)$ be the morphism $\sym\times \sym^{op}$-modules which sends a  graph $G\in\Gacirc\circ \Gacirc(P)$ to the graph $H\in\Gacirc(P)$  with $V(H)=\bigsqcup_{v\in V(G)} V(G_v)$ and whose edges are obtained by identifying, for any vertex $v$,   the $i$-th incoming edges of $v$ with the $i$-th incoming edge of $G_v$, and the $j$-th outgoing edge of $v$ with the $j$-th outgoing edge of $G_v$.
  \end{enumerate}
 \end{defi}
 In simpler words, the map $\eta_P$ sends an element $p\in P$ to the graph with one vertex, which is decorated by $p$ and has the same numbers of input and output edges as $p$ has inputs and outputs. In picture:
 \[\substack{\displaystyle \nu_P(p)=\\ \vspace{2cm}}
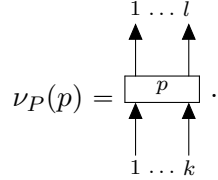
\begin{tikzpicture}[line cap=round,line join=round,>=triangle 45,x=0.7cm,y=0.7cm]
\clip(0.6,-2.1) rectangle (2.4,1.7);
\draw [line width=.4pt] (0.8,0.)-- (2.2,0.);
\draw [line width=.4pt] (2.2,0.)-- (2.2,-0.5);
\draw [line width=.4pt] (2.2,-0.5)-- (0.8,-0.5);
\draw [line width=.4pt] (0.8,-0.5)-- (0.8,0.);
\draw (1.2,0.1) node[anchor=north west] {\scriptsize $p$};
\draw [->,line width=.4pt] (1.,0.) -- (1.,1.);
\draw [->,line width=.4pt] (2.,0.) -- (2.,1.);
\draw [->,line width=.4pt] (1.,-1.5) -- (1,-0.5);
\draw [->,line width=.4pt] (2.,-1.5) -- (2.,-0.5);
\draw (0.7,-1.4) node[anchor=north west] {\scriptsize $1$};
\draw (1.05,-1.6) node[anchor=north west] {\scriptsize $\ldots$};
\draw (1.7,-1.4) node[anchor=north west] {\scriptsize $k$};
\draw (0.7,1.6) node[anchor=north west] {\scriptsize $1$};
\draw (1.05,1.4) node[anchor=north west] {\scriptsize $\ldots$};
\draw (1.7,1.6) node[anchor=north west] {\scriptsize $l$};
\end{tikzpicture}\substack{\displaystyle .\\ \vspace{2cm}}\]

Furthermore, the map $\mu_P$ replaces vertices decorated by graphs $(G,d_G)$ by (decorated) subgraphs. These subgraphs are exactly the graphs that were decorating the vertices of the original graph. To illustrate this graphically, we give an example in which $\mu_P$ sends the graph on the left to the graph on the right:
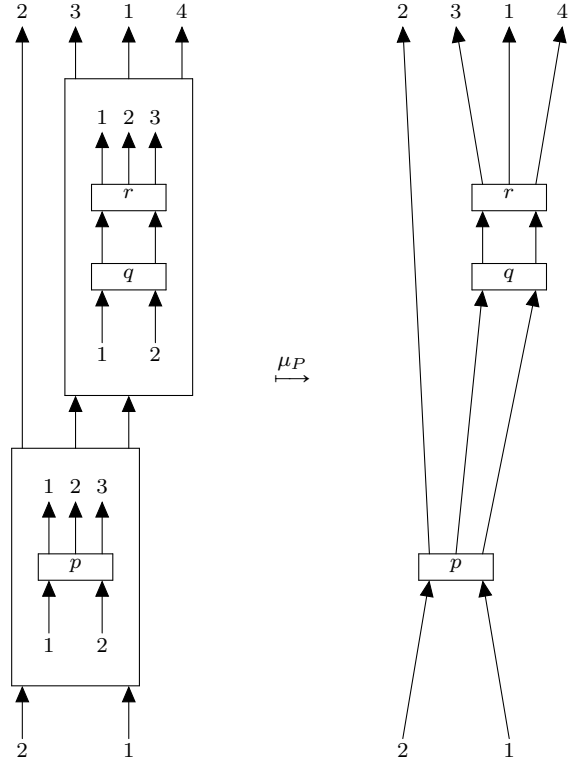
\begin{align*}
\substack{\hspace{5mm}\\ \begin{tikzpicture}[line cap=round,line join=round,>=triangle 45,x=0.7cm,y=0.7cm]
\clip(0.2,-4.1) rectangle (3.8,11.);
\draw [line width=.4pt] (0.8,0.)-- (2.2,0.);
\draw [line width=.4pt] (2.2,0.)-- (2.2,-0.5);
\draw [line width=.4pt] (2.2,-0.5)-- (0.8,-0.5);
\draw [line width=.4pt] (0.8,-0.5)-- (0.8,0.);
\draw (1.2,0.1) node[anchor=north west] {\scriptsize $p$};
\draw [->,line width=.4pt] (1.,0.) -- (1.,1.);
\draw [->,line width=.4pt] (1.5,0.) -- (1.5,1.);
\draw [->,line width=.4pt] (2.,0.) -- (2.,1.);
\draw [->,line width=.4pt] (1.,-1.5) -- (1.,-0.5);
\draw [->,line width=.4pt] (2.,-1.5) -- (2.,-0.5);
\draw (0.7,-1.4) node[anchor=north west] {\scriptsize $1$};
\draw (1.7,-1.4) node[anchor=north west] {\scriptsize $2$};
\draw (0.7,1.6) node[anchor=north west] {\scriptsize $1$};
\draw (1.2,1.6) node[anchor=north west] {\scriptsize $2$};
\draw (1.7,1.6) node[anchor=north west] {\scriptsize $3$};
\draw [line width=.4pt]  (0.3,-2.5) -- (2.7,-2.5);
\draw [line width=.4pt]  (2.7,-2.5) -- (2.7,2.);
\draw [line width=.4pt]  (2.7,2.) -- (0.3,2.);
\draw [line width=.4pt]  (0.3,2.) -- (0.3,-2.5);
\draw [->,line width=.4pt] (0.5,-3.5) -- (0.5,-2.5);
\draw [->,line width=.4pt] (2.5,-3.5) -- (2.5,-2.5);
\draw (0.2,-3.4) node[anchor=north west] {\scriptsize $2$};
\draw (2.2,-3.4) node[anchor=north west] {\scriptsize $1$};
\draw [->,line width=.4pt] (1.5,2.) -- (1.5,3.);
\draw [->,line width=.4pt] (2.5,2.) -- (2.5,3.);
\draw [line width=.4pt]  (1.3,3.) -- (3.7,3.); 
\draw [line width=.4pt]  (3.7,3.) -- (3.7,9.); 
\draw [line width=.4pt]  (3.7,9.) -- (1.3,9.); 
\draw [line width=.4pt]  (1.3,9.) -- (1.3,3.); 
\draw [->,line width=.4pt] (2.,4.) -- (2.,5.);
\draw [->,line width=.4pt] (1.5,9.) -- (1.5,10.);
\draw [->,line width=.4pt] (2.5,9.) -- (2.5,10.);
\draw [->,line width=.4pt] (3.5,9.) -- (3.5,10.);
\draw [->,line width=.4pt] (0.5,2.) -- (0.5,10.);
\draw (0.2,10.6) node[anchor=north west] {\scriptsize $2$};
\draw (1.2,10.6) node[anchor=north west] {\scriptsize $3$};
\draw (2.2,10.6) node[anchor=north west] {\scriptsize $1$};
\draw (3.2,10.6) node[anchor=north west] {\scriptsize $4$};
\draw (1.7,4.1) node[anchor=north west] {\scriptsize $1$};
\draw (2.7,4.1) node[anchor=north west] {\scriptsize $2$};
\draw [->,line width=.4pt] (3.,4.) -- (3.,5.);
\draw [line width=.4pt]  (1.8,5.) -- (3.2,5.);
\draw [line width=.4pt]  (3.2,5.) -- (3.2,5.5);
\draw [line width=.4pt]  (3.2,5.5) -- (1.8,5.5);
\draw [line width=.4pt]  (1.8,5.5) -- (1.8,5.);
\draw (2.2,5.6) node[anchor=north west] {\scriptsize $q$};
\draw [->,line width=.4pt] (2.,5.5) -- (2.,6.5);
\draw [->,line width=.4pt] (3.,5.5) -- (3.,6.5);
\draw [line width=.4pt]  (1.8,6.5) -- (3.2,6.5);
\draw [line width=.4pt]  (3.2,6.5) -- (3.2,7);
\draw [line width=.4pt]  (3.2,7) -- (1.8,7);
\draw [line width=.4pt]  (1.8,7) -- (1.8,6.5);
\draw (2.2,7.1) node[anchor=north west] {\scriptsize $r$};
\draw [->,line width=.4pt] (2.,7.) -- (2.,8.);
\draw [->,line width=.4pt] (2.5,7.) -- (2.5,8.);
\draw [->,line width=.4pt] (3.,7.) -- (3.,8.);
\draw (1.7,8.6) node[anchor=north west] {\scriptsize $1$};
\draw (2.2,8.6) node[anchor=north west] {\scriptsize $2$};
\draw (2.7,8.6) node[anchor=north west] {\scriptsize $3$};
\end{tikzpicture}}&\hspace{1cm}\substack{ \mu_P\\ \longmapsto}\hspace{1cm}
\substack{\hspace{5mm}\\ \begin{tikzpicture}[line cap=round,line join=round,>=triangle 45,x=0.7cm,y=0.7cm]
\clip(0.2,-4.1) rectangle (3.8,11.);
\draw [line width=.4pt] (0.8,0.)-- (2.2,0.);
\draw [line width=.4pt] (2.2,0.)-- (2.2,-0.5);
\draw [line width=.4pt] (2.2,-0.5)-- (0.8,-0.5);
\draw [line width=.4pt] (0.8,-0.5)-- (0.8,0.);
\draw (1.2,0.1) node[anchor=north west] {\scriptsize $p$};
\draw [->,line width=.4pt] (1.,0.) -- (0.5,10.); 
\draw [->,line width=.4pt] (1.5,0.) -- (2.,5.); 
\draw [->,line width=.4pt] (2.,0.) -- (3.,5.);  
\draw [->,line width=.4pt] (0.5,-3.5) -- (1.,-0.5); 
\draw [->,line width=.4pt] (2.5,-3.5) -- (2.,-0.5); 
\draw (0.2,-3.4) node[anchor=north west] {\scriptsize $2$};
\draw (2.2,-3.4) node[anchor=north west] {\scriptsize $1$};
\draw [->,line width=.4pt] (2.,7.) -- (1.5,10.);
\draw [->,line width=.4pt] (2.5,7.) -- (2.5,10.);
\draw [->,line width=.4pt] (3.,7.) -- (3.5,10.);
\draw (0.2,10.6) node[anchor=north west] {\scriptsize $2$};
\draw (1.2,10.6) node[anchor=north west] {\scriptsize $3$};
\draw (2.2,10.6) node[anchor=north west] {\scriptsize $1$};
\draw (3.2,10.6) node[anchor=north west] {\scriptsize $4$};
\draw [line width=.4pt]  (1.8,5.) -- (3.2,5.);
\draw [line width=.4pt]  (3.2,5.) -- (3.2,5.5);
\draw [line width=.4pt]  (3.2,5.5) -- (1.8,5.5);
\draw [line width=.4pt]  (1.8,5.5) -- (1.8,5.);
\draw (2.2,5.6) node[anchor=north west] {\scriptsize $q$};
\draw [->,line width=.4pt] (2.,5.5) -- (2.,6.5);
\draw [->,line width=.4pt] (3.,5.5) -- (3.,6.5);
\draw [line width=.4pt]  (1.8,6.5) -- (3.2,6.5);
\draw [line width=.4pt]  (3.2,6.5) -- (3.2,7);
\draw [line width=.4pt]  (3.2,7) -- (1.8,7);
\draw [line width=.4pt]  (1.8,7) -- (1.8,6.5);
\draw (2.2,7.1) node[anchor=north west] {\scriptsize $r$};
\end{tikzpicture}}
\end{align*}
where $p\in P(2,3)$, $q\in P(2,2)$ and $r\in P(2,3)$.\\

The families of morphisms $\eta_P$ and $\mu_P$ define two natural tranformations and we further obtain:
\begin{prop} \label{prop:monades_graphes}
 The triple $\Gacirc=(\Gacirc,\mu,\nu)$ is a monad in the category $\catssm$. 
Moreover, $\rGacirc=(\rGacirc,\mu_{\mid \rGacirc}, \mu_{\mid \rGacirc})$ is a sub-monad of $\Gacirc$.
\end{prop}

\subsection{TRAPs versus wheeled PROPs} \label{subsec:trap_vs_wPROP}

We can now state the main result  of this section, which relates TRAPs and various known objects.
\begin{theo} \label{thm:equivalence_TRAP_wPROP}
The categories of $\Gacirc$-algebras (i.e of wheeled PROPs) and of unitary TRAPs are isomorphic.
Similarly, the categories of $\rGacirc$-algebras and of TRAPs are isomorphic.
\end{theo}

\begin{remark}\label{rk:MMS09}
	Wheeled PROPs are defined (for example in \cite{Merkulov2009}) as $\Gacirc$-algebras. Thus Theorem \ref{thm:equivalence_TRAP_wPROP}    precisely says that wheeled PROPs and unitary TRAPs coincide, and that TRAPs can be viewed as non-unitary wheeled PROPs.
\end{remark}
\begin{proof}
Let us start with the TRAP case, i.e. the non-unitary case.

{From Proposition \ref{prop:freerGamma}, we know that   $\rGacirc(P)$ is the free  TRAP
generated by the $\sym\times \sym^{op}$-module $P$. 
Let $\nu_P$ be the canonical injection from $P$ to $\rGacirc(P)$}. If $P$ is a TRAP, then the canonical TRAP morphism
${\alpha_P}:\Gacirc(P)\longrightarrow P$ makes it a $\Gacirc$-algebra, thus defining  a functor from the category
of TRAPs to the category of $\rGacirc$-algebras.\\

Conversely, if $(P,\alpha)$ is a $\rGacirc$-algebra:
\begin{itemize}
\item For any $(p,p')\in P(k,l)\times P(k',l')$,we define  $p*p'$ by  applying  $\alpha$ to the following graph:
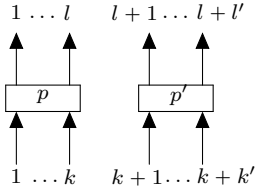
\begin{align*}
\begin{tikzpicture}[line cap=round,line join=round,>=triangle 45,x=0.7cm,y=0.7cm]
\clip(0.8,-2.1) rectangle (2.2,1.7);
\draw [line width=.4pt] (0.8,0.)-- (2.2,0.);
\draw [line width=.4pt] (2.2,0.)-- (2.2,-0.5);
\draw [line width=.4pt] (2.2,-0.5)-- (0.8,-0.5);
\draw [line width=.4pt] (0.8,-0.5)-- (0.8,0.);
\draw (1.2,0.1) node[anchor=north west] {\scriptsize $p$};
\draw [->,line width=.4pt] (1.,0.) -- (1.,1.);
\draw [->,line width=.4pt] (2.,0.) -- (2.,1.);
\draw [->,line width=.4pt] (1.,-1.5) -- (1,-0.5);
\draw [->,line width=.4pt] (2.,-1.5) -- (2.,-0.5);
\draw (0.7,-1.4) node[anchor=north west] {\scriptsize $1$};
\draw (1.1,-1.6) node[anchor=north west] {\scriptsize $\ldots$};
\draw (1.7,-1.4) node[anchor=north west] {\scriptsize $k$};
\draw (0.7,1.7) node[anchor=north west] {\scriptsize $1$};
\draw (1.05,1.5) node[anchor=north west] {\scriptsize $\ldots$};
\draw (1.7,1.7) node[anchor=north west] {\scriptsize $l$};
\end{tikzpicture}
\begin{tikzpicture}[line cap=round,line join=round,>=triangle 45,x=0.7cm,y=0.7cm]
\clip(-0.3,-2.1) rectangle (3.4,1.7);
\draw [line width=.4pt] (0.8,0.)-- (2.2,0.);
\draw [line width=.4pt] (2.2,0.)-- (2.2,-0.5);
\draw [line width=.4pt] (2.2,-0.5)-- (0.8,-0.5);
\draw [line width=.4pt] (0.8,-0.5)-- (0.8,0.);
\draw (1.2,0.2) node[anchor=north west] {\scriptsize $p'$};
\draw [->,line width=.4pt] (1.,0.) -- (1.,1.);
\draw [->,line width=.4pt] (2.,0.) -- (2.,1.);
\draw [->,line width=.4pt] (1.,-1.5) -- (1,-0.5);
\draw [->,line width=.4pt] (2.,-1.5) -- (2.,-0.5);
\draw (0.1,-1.4) node[anchor=north west] {\scriptsize $k+1$};
\draw (1.1,-1.6) node[anchor=north west] {\scriptsize $\ldots$};
\draw (1.7,-1.35) node[anchor=north west] {\scriptsize $k+k'$};
\draw (0.1,1.7) node[anchor=north west] {\scriptsize $l+1$};
\draw (1.05,1.5) node[anchor=north west] {\scriptsize $\ldots$};
\draw (1.7,1.75) node[anchor=north west] {\scriptsize $l+l'$};
\end{tikzpicture}\end{align*}
\item For any $p\in P(k,l)$, for any $(i,j)\in [k]\times [l]$, we define $t_{i,j}(p)$ by the application of $\alpha$ to the following graph:
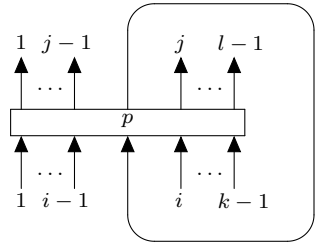
\begin{align*}
\begin{tikzpicture}[line cap=round,line join=round,>=triangle 45,x=0.7cm,y=0.7cm]
\clip(0.8,-2.6) rectangle (6.5,2.1);
\draw [line width=.4pt] (0.8,0.)-- (5.2,0.);
\draw [line width=.4pt] (5.2,0.)-- (5.2,-0.5);
\draw [line width=.4pt] (5.2,-0.5)-- (0.8,-0.5);
\draw [line width=.4pt] (0.8,-0.5)-- (0.8,0.);
\draw (2.7,0.1) node[anchor=north west] {\scriptsize $p$};
\draw [->,line width=.4pt] (1.,0.) -- (1.,1.);
\draw [->,line width=.4pt] (2.,0.) -- (2.,1.);
\draw [line width=.4pt] (3.,0.) -- (3.,1.5);
\draw [->,line width=.4pt] (4.,0.) -- (4.,1.);
\draw [->,line width=.4pt] (5.,0.) -- (5.,1.);
\draw [->,line width=.4pt] (1.,-1.5) -- (1,-0.5);
\draw [->,line width=.4pt] (2.,-1.5) -- (2.,-0.5);
\draw [->,line width=.4pt] (3.,-2) -- (3,-0.5);
\draw [->,line width=.4pt] (4.,-1.5) -- (4.,-0.5);
\draw [->,line width=.4pt] (5.,-1.5) -- (5.,-0.5);
\draw (0.7,-1.4) node[anchor=north west] {\scriptsize $1$};
\draw (1.1,-1.) node[anchor=north west] {\scriptsize $\ldots$};
\draw (1.2,-1.4) node[anchor=north west] {\scriptsize $i-1$};
\draw (3.7,-1.4) node[anchor=north west] {\scriptsize $i$};
\draw (4.1,-1.) node[anchor=north west] {\scriptsize $\ldots$};
\draw (4.5,-1.4) node[anchor=north west] {\scriptsize $k-1$};
\draw (0.7,1.6) node[anchor=north west] {\scriptsize $1$};
\draw (1.1,0.6) node[anchor=north west] {\scriptsize $\ldots$};
\draw (1.2,1.6) node[anchor=north west] {\scriptsize $j-1$};
\draw (3.7,1.6) node[anchor=north west] {\scriptsize $j$};
\draw (4.1,0.6) node[anchor=north west] {\scriptsize $\ldots$};
\draw (4.5,1.6) node[anchor=north west] {\scriptsize $l-1$};
\draw [shift={(3.5,-2.)},line width=0.4pt]  plot[domain=3.141592653589793:4.71238898,variable=\t]({1.*0.5*cos(\t r)+0.*0.5*sin(\t r)},{0.*0.5*cos(\t r)+1.*0.5*sin(\t r)});
\draw [line width=.4pt] (3.5,-2.5) -- (6.,-2.5);
\draw [shift={(6.,-2.)},line width=0.4pt]  plot[domain=4.71238898:6.283185307,variable=\t]({1.*0.5*cos(\t r)+0.*0.5*sin(\t r)},{0.*0.5*cos(\t r)+1.*0.5*sin(\t r)});
\draw [shift={(3.5,1.5)},line width=0.4pt]  plot[domain=1.570796327:3.141592653589793,variable=\t]({1.*0.5*cos(\t r)+0.*0.5*sin(\t r)},{0.*0.5*cos(\t r)+1.*0.5*sin(\t r)});
\draw [line width=.4pt] (3.5,2.) -- (6.,2.);
\draw [shift={(6,1.5)},line width=0.4pt]  plot[domain=0.:1.570796327,variable=\t]({1.*0.5*cos(\t r)+0.*0.5*sin(\t r)},{0.*0.5*cos(\t r)+1.*0.5*sin(\t r)});
\draw [line width=.4pt] (6.5,-2.) -- (6.5,1.5);
\end{tikzpicture}
\end{align*}
\end{itemize}
Let us prove some of the axioms of TRAPs for $P$. The others can be proved in the same way and are left to the reader.

1. holds by Proposition \ref{prop:two_endofunctors}.

2. (a):  let $(p,p',p'')\in P(k,l)\times P(k',l')\times P(k'',l'')$. Then $(p*p')*p''$ is obtained by the application
of $\alpha_P$ to the graph:
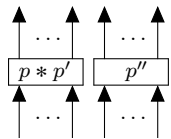
\begin{align*}
\begin{tikzpicture}[line cap=round,line join=round,>=triangle 45,x=0.7cm,y=0.7cm]
\clip(0.8,-1.5) rectangle (2.4,1.);
\draw [line width=.4pt] (0.8,0.)-- (2.2,0.);
\draw [line width=.4pt] (2.2,0.)-- (2.2,-0.5);
\draw [line width=.4pt] (2.2,-0.5)-- (0.8,-0.5);
\draw [line width=.4pt] (0.8,-0.5)-- (0.8,0.);
\draw (0.8,0.15) node[anchor=north west] {\scriptsize $p*p'$};
\draw [->,line width=.4pt] (1.,0.) -- (1.,1.);
\draw [->,line width=.4pt] (2.,0.) -- (2.,1.);
\draw [->,line width=.4pt] (1.,-1.5) -- (1,-0.5);
\draw [->,line width=.4pt] (2.,-1.5) -- (2.,-0.5);
\draw (1.1,-0.9) node[anchor=north west] {\scriptsize $\ldots$};
\draw (1.1,0.6) node[anchor=north west] {\scriptsize $\ldots$};
\end{tikzpicture}
\begin{tikzpicture}[line cap=round,line join=round,>=triangle 45,x=0.7cm,y=0.7cm]
\clip(0.8,-1.5) rectangle (2.4,1.);
\draw [line width=.4pt] (0.8,0.)-- (2.2,0.);
\draw [line width=.4pt] (2.2,0.)-- (2.2,-0.5);
\draw [line width=.4pt] (2.2,-0.5)-- (0.8,-0.5);
\draw [line width=.4pt] (0.8,-0.5)-- (0.8,0.);
\draw (1.2,0.15) node[anchor=north west] {\scriptsize $p''$};
\draw [->,line width=.4pt] (1.,0.) -- (1.,1.);
\draw [->,line width=.4pt] (2.,0.) -- (2.,1.);
\draw [->,line width=.4pt] (1.,-1.5) -- (1,-0.5);
\draw [->,line width=.4pt] (2.,-1.5) -- (2.,-0.5);
\draw (1.1,-0.9) node[anchor=north west] {\scriptsize $\ldots$};
\draw (1.1,0.6) node[anchor=north west] {\scriptsize $\ldots$};
\end{tikzpicture}
\end{align*}
(For the sake of simplicity, we delete the indices of the input and output edges of this graph: they are always indexed from left to right).
Hence, $(p*p')*p''$ is obtained by application of $\alpha\circ \Gacirc(\alpha)$ to the graph:
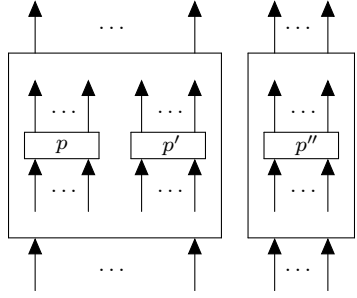
\begin{align*}
\begin{tikzpicture}[line cap=round,line join=round,>=triangle 45,x=0.7cm,y=0.7cm]
\clip(0.5,-3.) rectangle (7.,2.5);
\draw [line width=.4pt] (0.8,0.)-- (2.2,0.);
\draw [line width=.4pt] (2.2,0.)-- (2.2,-0.5);
\draw [line width=.4pt] (2.2,-0.5)-- (0.8,-0.5);
\draw [line width=.4pt] (0.8,-0.5)-- (0.8,0.);
\draw (1.2,0.05) node[anchor=north west] {\scriptsize $p$};
\draw [->,line width=.4pt] (1.,0.) -- (1.,1.);
\draw [->,line width=.4pt] (2.,0.) -- (2.,1.);
\draw [->,line width=.4pt] (1.,-1.5) -- (1,-0.5);
\draw [->,line width=.4pt] (2.,-1.5) -- (2.,-0.5);
\draw (1.1,-0.9) node[anchor=north west] {\scriptsize $\ldots$};
\draw (1.1,0.6) node[anchor=north west] {\scriptsize $\ldots$};
\draw [line width=.4pt] (2.8,0.)-- (4.2,0.);
\draw [line width=.4pt] (4.2,0.)-- (4.2,-0.5);
\draw [line width=.4pt] (4.2,-0.5)-- (2.8,-0.5);
\draw [line width=.4pt] (2.8,-0.5)-- (2.8,0.);
\draw (3.2,0.15) node[anchor=north west] {\scriptsize $p'$};
\draw [->,line width=.4pt] (3.,0.) -- (3.,1.);
\draw [->,line width=.4pt] (4.,0.) -- (4.,1.);
\draw [->,line width=.4pt] (3.,-1.5) -- (3,-0.5);
\draw [->,line width=.4pt] (4.,-1.5) -- (4.,-0.5);
\draw (3.1,-0.9) node[anchor=north west] {\scriptsize $\ldots$};
\draw (3.1,0.6) node[anchor=north west] {\scriptsize $\ldots$};
\draw [line width=.4pt] (5.3,0.)-- (6.7,0.);
\draw [line width=.4pt] (6.7,0.)-- (6.7,-0.5);
\draw [line width=.4pt] (6.7,-0.5)-- (5.3,-0.5);
\draw [line width=.4pt] (5.3,-0.5)-- (5.3,0.);
\draw (5.7,0.15) node[anchor=north west] {\scriptsize $p''$};
\draw [->,line width=.4pt] (5.5,0.) -- (5.5,1.);
\draw [->,line width=.4pt] (6.5,0.) -- (6.5,1.);
\draw [->,line width=.4pt] (5.5,-1.5) -- (5.5,-0.5);
\draw [->,line width=.4pt] (6.5,-1.5) -- (6.5,-0.5);
\draw (5.6,-0.9) node[anchor=north west] {\scriptsize $\ldots$};
\draw (5.6,0.6) node[anchor=north west] {\scriptsize $\ldots$};
\draw [line width=.4pt] (0.5,-2) -- (4.5,-2);
\draw [line width=.4pt] (4.5,-2) -- (4.5,1.5);
\draw [line width=.4pt] (4.5,1.5) -- (0.5,1.5);
\draw [line width=.4pt] (0.5,1.5) -- (0.5,-2.);
\draw [->,line width=.4pt] (1.,-3.) -- (1.,-2.);
\draw [->,line width=.4pt] (1.,1.5) -- (1.,2.5);
\draw [->,line width=.4pt] (4.,-3.) -- (4.,-2.);
\draw [->,line width=.4pt] (4.,1.5) -- (4.,2.5);
\draw (2.,-2.4) node[anchor=north west] {\scriptsize $\ldots$};
\draw (2.,2.2) node[anchor=north west] {\scriptsize $\ldots$};
\draw [line width=.4pt] (5,-2) -- (7,-2);
\draw [line width=.4pt] (7,-2) -- (7,1.5);
\draw [line width=.4pt] (7,1.5) -- (5,1.5);
\draw [line width=.4pt] (5,1.5) -- (5,-2);
\draw [->,line width=.4pt] (5.5,-3.) -- (5.5,-2.);
\draw [->,line width=.4pt] (6.5,1.5) -- (6.5,2.5);
\draw [->,line width=.4pt] (6.5,-3.) -- (6.5,-2.);
\draw [->,line width=.4pt] (5.5,1.5) -- (5.5,2.5);
\draw (5.5,-2.4) node[anchor=north west] {\scriptsize $\ldots$};
\draw (5.5,2.2) node[anchor=north west] {\scriptsize $\ldots$};
\end{tikzpicture}
\end{align*}
Note that for the second connected component of this graph, this comes from:
\[\alpha \circ \Gacirc(\alpha)\circ \Gacirc(\nu_P)(p'')=\alpha \circ \Gacirc(\alpha \circ \nu_P)(p'')
=\alpha \circ \Gacirc(\mathrm{Id}_P)(p'')=\alpha(p'').\]
As $\alpha \circ \Gacirc(\alpha)=\alpha \circ \mu_P$, $(p*p')*p''$ is obtained by applying $\alpha$ to the graph:
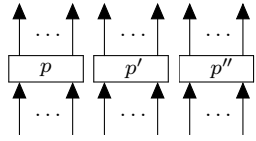
\begin{align*}
\begin{tikzpicture}[line cap=round,line join=round,>=triangle 45,x=0.7cm,y=0.7cm]
\clip(0.8,-1.5) rectangle (2.4,1.);
\draw [line width=.4pt] (0.8,0.)-- (2.2,0.);
\draw [line width=.4pt] (2.2,0.)-- (2.2,-0.5);
\draw [line width=.4pt] (2.2,-0.5)-- (0.8,-0.5);
\draw [line width=.4pt] (0.8,-0.5)-- (0.8,0.);
\draw (1.2,0.05) node[anchor=north west] {\scriptsize $p$};
\draw [->,line width=.4pt] (1.,0.) -- (1.,1.);
\draw [->,line width=.4pt] (2.,0.) -- (2.,1.);
\draw [->,line width=.4pt] (1.,-1.5) -- (1,-0.5);
\draw [->,line width=.4pt] (2.,-1.5) -- (2.,-0.5);
\draw (1.1,-0.9) node[anchor=north west] {\scriptsize $\ldots$};
\draw (1.1,0.6) node[anchor=north west] {\scriptsize $\ldots$};
\end{tikzpicture}
\begin{tikzpicture}[line cap=round,line join=round,>=triangle 45,x=0.7cm,y=0.7cm]
\clip(0.8,-1.5) rectangle (2.4,1.);
\draw [line width=.4pt] (0.8,0.)-- (2.2,0.);
\draw [line width=.4pt] (2.2,0.)-- (2.2,-0.5);
\draw [line width=.4pt] (2.2,-0.5)-- (0.8,-0.5);
\draw [line width=.4pt] (0.8,-0.5)-- (0.8,0.);
\draw (1.2,0.15) node[anchor=north west] {\scriptsize $p'$};
\draw [->,line width=.4pt] (1.,0.) -- (1.,1.);
\draw [->,line width=.4pt] (2.,0.) -- (2.,1.);
\draw [->,line width=.4pt] (1.,-1.5) -- (1,-0.5);
\draw [->,line width=.4pt] (2.,-1.5) -- (2.,-0.5);
\draw (1.1,-0.9) node[anchor=north west] {\scriptsize $\ldots$};
\draw (1.1,0.6) node[anchor=north west] {\scriptsize $\ldots$};
\end{tikzpicture}
\begin{tikzpicture}[line cap=round,line join=round,>=triangle 45,x=0.7cm,y=0.7cm]
\clip(0.8,-1.5) rectangle (2.4,1.);
\draw [line width=.4pt] (0.8,0.)-- (2.2,0.);
\draw [line width=.4pt] (2.2,0.)-- (2.2,-0.5);
\draw [line width=.4pt] (2.2,-0.5)-- (0.8,-0.5);
\draw [line width=.4pt] (0.8,-0.5)-- (0.8,0.);
\draw (1.2,0.15) node[anchor=north west] {\scriptsize $p''$};
\draw [->,line width=.4pt] (1.,0.) -- (1.,1.);
\draw [->,line width=.4pt] (2.,0.) -- (2.,1.);
\draw [->,line width=.4pt] (1.,-1.5) -- (1,-0.5);
\draw [->,line width=.4pt] (2.,-1.5) -- (2.,-0.5);
\draw (1.1,-0.9) node[anchor=north west] {\scriptsize $\ldots$};
\draw (1.1,0.6) node[anchor=north west] {\scriptsize $\ldots$};
\end{tikzpicture}
\end{align*}
The same computation can be carried out for $p*(p'*p'')$, which gives the associativity of $*$.

2. (b):  the unity $I_0$ of the concatenation product of graph is the empty graph, which is the image of the unity of $P$ for the horizontal concatenation under $\alpha$.

2. (c) holds trivially by definition of the horizontal concatenation product on $P$, the $\sym\times\sym^{op}$-module structure of $P$, and the fact that $\rPGr(X)$ is a TRAP.

3. (c): for any $k,l,k',l'\geq 1$, for any $i\in [k]$, $j\in [l]$, for any $p\in P(k,l)$, $p'\in P(k',l')$, $t_{i,j}(p*p')$ is the image under $\alpha_P$ of the graph
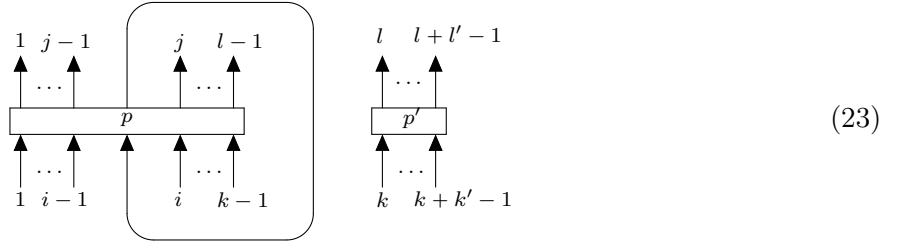
\begin{align} \label{graph:compatibility_tij_vert_conc1}
\substack{\hspace{1cm} \\ \begin{tikzpicture}[line cap=round,line join=round,>=triangle 45,x=0.7cm,y=0.7cm]
\clip(0.8,-2.6) rectangle (6.5,2.1);
\draw [line width=.4pt] (0.8,0)-- (5.2,0.);
\draw [line width=.4pt] (5.2,0.)-- (5.2,-0.5);
\draw [line width=.4pt] (5.2,-0.5)-- (0.8,-0.5);
\draw [line width=.4pt] (0.8,-0.5)-- (0.8,0.);
\draw (2.7,0.1) node[anchor=north west] {\scriptsize $p$};
\draw [->,line width=.4pt] (1.,0.) -- (1.,1.);
\draw [->,line width=.4pt] (2.,0.) -- (2.,1.);
\draw [line width=.4pt] (3.,0.) -- (3.,1.5);
\draw [->,line width=.4pt] (4.,0.) -- (4.,1.);
\draw [->,line width=.4pt] (5.,0.) -- (5.,1.);
\draw [->,line width=.4pt] (1.,-1.5) -- (1,-0.5);
\draw [->,line width=.4pt] (2.,-1.5) -- (2.,-0.5);
\draw [->,line width=.4pt] (3.,-2) -- (3,-0.5);
\draw [->,line width=.4pt] (4.,-1.5) -- (4.,-0.5);
\draw [->,line width=.4pt] (5.,-1.5) -- (5.,-0.5);
\draw (0.7,-1.4) node[anchor=north west] {\scriptsize $1$};
\draw (1.1,-1.) node[anchor=north west] {\scriptsize $\ldots$};
\draw (1.2,-1.4) node[anchor=north west] {\scriptsize $i-1$};
\draw (3.7,-1.4) node[anchor=north west] {\scriptsize $i$};
\draw (4.1,-1.) node[anchor=north west] {\scriptsize $\ldots$};
\draw (4.5,-1.4) node[anchor=north west] {\scriptsize $k-1$};
\draw (0.7,1.6) node[anchor=north west] {\scriptsize $1$};
\draw (1.1,0.6) node[anchor=north west] {\scriptsize $\ldots$};
\draw (1.2,1.6) node[anchor=north west] {\scriptsize $j-1$};
\draw (3.7,1.6) node[anchor=north west] {\scriptsize $j$};
\draw (4.1,0.6) node[anchor=north west] {\scriptsize $\ldots$};
\draw (4.5,1.6) node[anchor=north west] {\scriptsize $l-1$};
\draw [shift={(3.5,-2.)},line width=0.4pt]  plot[domain=3.141592653589793:4.71238898,variable=\t]({1.*0.5*cos(\t r)+0.*0.5*sin(\t r)},{0.*0.5*cos(\t r)+1.*0.5*sin(\t r)});
\draw [line width=.4pt] (3.5,-2.5) -- (6.,-2.5);
\draw [shift={(6.,-2.)},line width=0.4pt]  plot[domain=4.71238898:6.283185307,variable=\t]({1.*0.5*cos(\t r)+0.*0.5*sin(\t r)},{0.*0.5*cos(\t r)+1.*0.5*sin(\t r)});
\draw [shift={(3.5,1.5)},line width=0.4pt]  plot[domain=1.570796327:3.141592653589793,variable=\t]({1.*0.5*cos(\t r)+0.*0.5*sin(\t r)},{0.*0.5*cos(\t r)+1.*0.5*sin(\t r)});
\draw [line width=.4pt] (3.5,2.) -- (6.,2.);
\draw [shift={(6,1.5)},line width=0.4pt]  plot[domain=0.:1.570796327,variable=\t]({1.*0.5*cos(\t r)+0.*0.5*sin(\t r)},{0.*0.5*cos(\t r)+1.*0.5*sin(\t r)});
\draw [line width=.4pt] (6.5,-2.) -- (6.5,1.5);
\end{tikzpicture}}
\substack{\hspace{1cm} \\ \begin{tikzpicture}[line cap=round,line join=round,>=triangle 45,x=0.7cm,y=0.7cm]
\clip(-0.3,-2.1) rectangle (3.9,2.4);
\draw [line width=.4pt] (0.8,0.4)-- (2.2,0.4);
\draw [line width=.4pt] (2.2,0.4)-- (2.2,-0.1);
\draw [line width=.4pt] (2.2,-0.1)-- (0.8,-0.1);
\draw [line width=.4pt] (0.8,-0.1)-- (0.8,0.4);
\draw (1.2,0.6) node[anchor=north west] {\scriptsize $p'$};
\draw [->,line width=.4pt] (1.,0.4) -- (1.,1.4);
\draw [->,line width=.4pt] (2.,0.4) -- (2.,1.4);
\draw [->,line width=.4pt] (1.,-1.1) -- (1,-0.1);
\draw [->,line width=.4pt] (2.,-1.1) -- (2.,-0.1);
\draw (0.7,-1.0) node[anchor=north west] {\scriptsize $k$};
\draw (1.1,-0.6) node[anchor=north west] {\scriptsize $\ldots$};
\draw (1.4,-0.95) node[anchor=north west] {\scriptsize $k+k'-1$};
\draw (0.7,2.1) node[anchor=north west] {\scriptsize $l$};
\draw (1.05,1.1) node[anchor=north west] {\scriptsize $\ldots$};
\draw (1.4,2.15) node[anchor=north west] {\scriptsize $l+l'-1$};
\end{tikzpicture}}
\end{align}
On the other hand, $t_{i,j}(p)*p'$ is the image under $\alpha_P$ of the graph
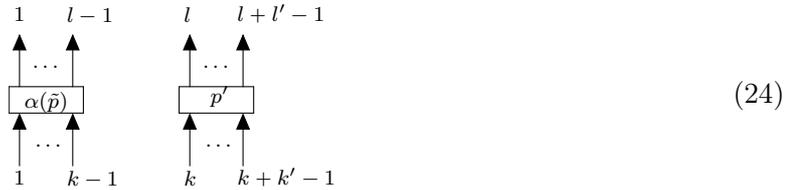
\begin{align} \label{graph:compatibility_tij_vert_conc2}
\substack{\hspace{1cm} \\ \begin{tikzpicture}[line cap=round,line join=round,>=triangle 45,x=0.7cm,y=0.7cm]
\clip(0.8,-2.1) rectangle (2.9,1.7);
\draw [line width=.4pt] (0.8,0.)-- (2.2,0.);
\draw [line width=.4pt] (2.2,0.)-- (2.2,-0.5);
\draw [line width=.4pt] (2.2,-0.5)-- (0.8,-0.5);
\draw [line width=.4pt] (0.8,-0.5)-- (0.8,0.);
\draw (0.9,0.1) node[anchor=north west] {\scriptsize $\alpha(\tilde p)$};
\draw [->,line width=.4pt] (1.,0.) -- (1.,1.);
\draw [->,line width=.4pt] (2.,0.) -- (2.,1.);
\draw [->,line width=.4pt] (1.,-1.5) -- (1,-0.5);
\draw [->,line width=.4pt] (2.,-1.5) -- (2.,-0.5);
\draw (0.7,-1.4) node[anchor=north west] {\scriptsize $1$};
\draw (1.1,-0.9) node[anchor=north west] {\scriptsize $\ldots$};
\draw (1.7,-1.4) node[anchor=north west] {\scriptsize $k-1$};
\draw (0.7,1.7) node[anchor=north west] {\scriptsize $1$};
\draw (1.05,0.6) node[anchor=north west] {\scriptsize $\ldots$};
\draw (1.7,1.7) node[anchor=north west] {\scriptsize $l-1$};
\end{tikzpicture}}
\substack{\hspace{1cm} \\ \begin{tikzpicture}[line cap=round,line join=round,>=triangle 45,x=0.7cm,y=0.7cm]
\clip(-0.3,-2.1) rectangle (3.9,1.7);
\draw [line width=.4pt] (0.8,0.)-- (2.2,0.);
\draw [line width=.4pt] (2.2,0.)-- (2.2,-0.5);
\draw [line width=.4pt] (2.2,-0.5)-- (0.8,-0.5);
\draw [line width=.4pt] (0.8,-0.5)-- (0.8,0.);
\draw (1.2,0.2) node[anchor=north west] {\scriptsize $p'$};
\draw [->,line width=.4pt] (1.,0.) -- (1.,1.);
\draw [->,line width=.4pt] (2.,0.) -- (2.,1.);
\draw [->,line width=.4pt] (1.,-1.5) -- (1,-0.5);
\draw [->,line width=.4pt] (2.,-1.5) -- (2.,-0.5);
\draw (0.7,-1.4) node[anchor=north west] {\scriptsize $k$};
\draw (1.1,-0.9) node[anchor=north west] {\scriptsize $\ldots$};
\draw (1.7,-1.35) node[anchor=north west] {\scriptsize $k+k'-1$};
\draw (0.7,1.7) node[anchor=north west] {\scriptsize $l$};
\draw (1.05,0.6) node[anchor=north west] {\scriptsize $\ldots$};
\draw (1.7,1.75) node[anchor=north west] {\scriptsize $l+l'-1$};
\end{tikzpicture}}\end{align}
with $\tilde p$ the image under $\alpha_P$ of the graph
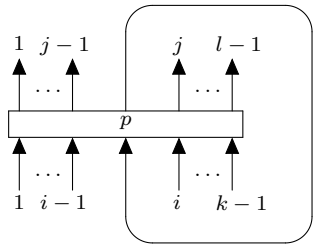
\begin{align*}
\begin{tikzpicture}[line cap=round,line join=round,>=triangle 45,x=0.7cm,y=0.7cm]
\clip(0.8,-2.6) rectangle (6.5,2.1);
\draw [line width=.4pt] (0.8,0.)-- (5.2,0.);
\draw [line width=.4pt] (5.2,0.)-- (5.2,-0.5);
\draw [line width=.4pt] (5.2,-0.5)-- (0.8,-0.5);
\draw [line width=.4pt] (0.8,-0.5)-- (0.8,0.);
\draw (2.7,0.1) node[anchor=north west] {\scriptsize $p$};
\draw [->,line width=.4pt] (1.,0.) -- (1.,1.);
\draw [->,line width=.4pt] (2.,0.) -- (2.,1.);
\draw [line width=.4pt] (3.,0.) -- (3.,1.5);
\draw [->,line width=.4pt] (4.,0.) -- (4.,1.);
\draw [->,line width=.4pt] (5.,0.) -- (5.,1.);
\draw [->,line width=.4pt] (1.,-1.5) -- (1,-0.5);
\draw [->,line width=.4pt] (2.,-1.5) -- (2.,-0.5);
\draw [->,line width=.4pt] (3.,-2) -- (3,-0.5);
\draw [->,line width=.4pt] (4.,-1.5) -- (4.,-0.5);
\draw [->,line width=.4pt] (5.,-1.5) -- (5.,-0.5);
\draw (0.7,-1.4) node[anchor=north west] {\scriptsize $1$};
\draw (1.1,-1.) node[anchor=north west] {\scriptsize $\ldots$};
\draw (1.2,-1.4) node[anchor=north west] {\scriptsize $i-1$};
\draw (3.7,-1.4) node[anchor=north west] {\scriptsize $i$};
\draw (4.1,-1.) node[anchor=north west] {\scriptsize $\ldots$};
\draw (4.5,-1.4) node[anchor=north west] {\scriptsize $k-1$};
\draw (0.7,1.6) node[anchor=north west] {\scriptsize $1$};
\draw (1.1,0.6) node[anchor=north west] {\scriptsize $\ldots$};
\draw (1.2,1.6) node[anchor=north west] {\scriptsize $j-1$};
\draw (3.7,1.6) node[anchor=north west] {\scriptsize $j$};
\draw (4.1,0.6) node[anchor=north west] {\scriptsize $\ldots$};
\draw (4.5,1.6) node[anchor=north west] {\scriptsize $l-1$};
\draw [shift={(3.5,-2.)},line width=0.4pt]  plot[domain=3.141592653589793:4.71238898,variable=\t]({1.*0.5*cos(\t r)+0.*0.5*sin(\t r)},{0.*0.5*cos(\t r)+1.*0.5*sin(\t r)});
\draw [line width=.4pt] (3.5,-2.5) -- (6.,-2.5);
\draw [shift={(6.,-2.)},line width=0.4pt]  plot[domain=4.71238898:6.283185307,variable=\t]({1.*0.5*cos(\t r)+0.*0.5*sin(\t r)},{0.*0.5*cos(\t r)+1.*0.5*sin(\t r)});
\draw [shift={(3.5,1.5)},line width=0.4pt]  plot[domain=1.570796327:3.141592653589793,variable=\t]({1.*0.5*cos(\t r)+0.*0.5*sin(\t r)},{0.*0.5*cos(\t r)+1.*0.5*sin(\t r)});
\draw [line width=.4pt] (3.5,2.) -- (6.,2.);
\draw [shift={(6,1.5)},line width=0.4pt]  plot[domain=0.:1.570796327,variable=\t]({1.*0.5*cos(\t r)+0.*0.5*sin(\t r)},{0.*0.5*cos(\t r)+1.*0.5*sin(\t r)});
\draw [line width=.4pt] (6.5,-2.) -- (6.5,1.5);
\end{tikzpicture}
\end{align*}
The images  of the graphs \eqref{graph:compatibility_tij_vert_conc1} and \eqref{graph:compatibility_tij_vert_conc2} under $\alpha_P$ are identical by the commutativity of the first diagram of \eqref{axiomesalgebres}. The case $l+1\leq j \leq l+l'$ and $k+1\leq i\leq k+k'$ holds by the same argument.

{Let us now focus on the unitary case.

First if $P$ is a unitary TRAP, then by Theorem \ref{freetraps}, $(P,\mu_P)$ is a $\Gacirc$-algebra with exactly the same argument as in the non unitary case.

Conversely, let $(P,\alpha_P)$ be a $\Gacirc$-algebra. We then set 
\begin{equation*}
 I:=\alpha_P(I_1)
\end{equation*}}
where $I_1$ is the graph with only one input-output edge.

Let $p\in P(k,l)$ and $2\leqslant j\leqslant l+1$. Then $t_{1,j}(I*p)$ is obtained by applying
  $\alpha \circ \Gacirc(\alpha)$ to the graph:
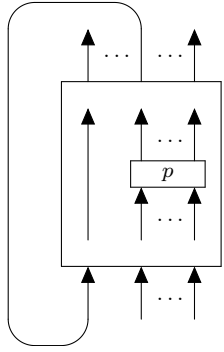
\begin{align*}
\begin{tikzpicture}[line cap=round,line join=round,>=triangle 45,x=0.7cm,y=0.7cm]
\clip(-1.7,-3.5) rectangle (2.7,3.2);
\draw [line width=.4pt] (0.8,0.)-- (2.2,0.);
\draw [line width=.4pt] (2.2,0.)-- (2.2,-0.5);
\draw [line width=.4pt] (2.2,-0.5)-- (0.8,-0.5);
\draw [line width=.4pt] (0.8,-0.5)-- (0.8,0.);
\draw (1.2,0.05) node[anchor=north west] {\scriptsize $p$};
\draw [->,line width=.4pt] (1.,0.) -- (1.,1.);
\draw [->,line width=.4pt] (2.,0.) -- (2.,1.);
\draw [->,line width=.4pt] (1.,-1.5) -- (1,-0.5);
\draw [->,line width=.4pt] (2.,-1.5) -- (2.,-0.5);
\draw (1.1,-0.9) node[anchor=north west] {\scriptsize $\ldots$};
\draw (1.1,0.6) node[anchor=north west] {\scriptsize $\ldots$};
\draw [->,line width=.4pt] (0.,-1.5) -- (0.,1.);
\draw [line width=.4pt] (-0.5,-2) -- (2.5,-2);
\draw [line width=.4pt] (2.5,-2) -- (2.5,1.5);
\draw [line width=.4pt] (2.5,1.5) -- (-0.5,1.5);
\draw [line width=.4pt] (-0.5,1.5) -- (-0.5,-2.);
\draw [->,line width=.4pt] (0.,-3.) -- (0.,-2.);
\draw [->,line width=.4pt] (1.,-3.) -- (1.,-2.);
\draw [->,line width=.4pt] (2.,-3.) -- (2.,-2.);
\draw (1.1,-2.4) node[anchor=north west] {\scriptsize $\ldots$};
\draw [->,line width=.4pt] (0.,1.5) -- (0.,2.5);
\draw [line width=.4pt] (1.,1.5) -- (1.,2.5);
\draw [->,line width=.4pt] (2.,1.5) -- (2.,2.5);
\draw (0.1,2.2) node[anchor=north west] {\scriptsize $\ldots$};
\draw (1.1,2.2) node[anchor=north west] {\scriptsize $\ldots$};
\draw [shift={(-0.5,-3)},line width=0.4pt]  plot[domain=4.71238898:6.283185307,variable=\t]({1.*0.5*cos(\t r)+0.*0.5*sin(\t r)},{0.*0.5*cos(\t r)+1.*0.5*sin(\t r)});
\draw [line width=.4pt] (-0.5,-3.5) -- (-1.,-3.5);
\draw [shift={(-1.,-3)},line width=0.4pt]  plot[domain=3.141592654:4.71238898,variable=\t]({1.*0.5*cos(\t r)+0.*0.5*sin(\t r)},{0.*0.5*cos(\t r)+1.*0.5*sin(\t r)});
\draw [line width=.4pt] (-1.5,-3.) -- (-1.5,2.5);
\draw [shift={(-1.,2.5)},line width=0.4pt]  plot[domain=1.570796327:3.141592654,variable=\t]({1.*0.5*cos(\t r)+0.*0.5*sin(\t r)},{0.*0.5*cos(\t r)+1.*0.5*sin(\t r)});
\draw [line width=.4pt] (-1.,3.) -- (0.5,3.);
\draw [shift={(0.5,2.5)},line width=0.4pt]  plot[domain=0.:1.570796327,variable=\t]({1.*0.5*cos(\t r)+0.*0.5*sin(\t r)},{0.*0.5*cos(\t r)+1.*0.5*sin(\t r)});
\end{tikzpicture}
\end{align*}
where the curved edge relate the first edge at the bottom to the $j$-th edge on the top.
As $\alpha \circ \Gacirc(\alpha)=\alpha \circ \mu_P$, $t_{1,j}(I*p)$ is obtained by application
of $\alpha$ to the graph:
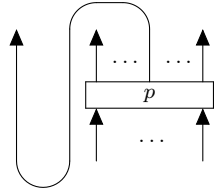
\begin{align*}
\begin{tikzpicture}[line cap=round,line join=round,>=triangle 45,x=0.7cm,y=0.7cm]
\clip(-0.7,-2) rectangle (3.2,1.5);
\draw [line width=.4pt] (0.8,0.)-- (3.2,0.);
\draw [line width=.4pt] (3.2,0.)-- (3.2,-0.5);
\draw [line width=.4pt] (3.2,-0.5)-- (0.8,-0.5);
\draw [line width=.4pt] (0.8,-0.5)-- (0.8,0.);
\draw (1.7,0.05) node[anchor=north west] {\scriptsize $p$};
\draw [->,line width=.4pt] (1.,0.) -- (1.,1.);
\draw [line width=.4pt] (2.,0.) -- (2.,1.);
\draw [->,line width=.4pt] (3.,0.) -- (3.,1.);
\draw [->,line width=.4pt] (1.,-1.5) -- (1,-0.5);
\draw [->,line width=.4pt] (3.,-1.5) -- (3.,-0.5);
\draw (1.6,-0.9) node[anchor=north west] {\scriptsize $\ldots$};
\draw (1.1,0.6) node[anchor=north west] {\scriptsize $\ldots$};
\draw (2.1,0.6) node[anchor=north west] {\scriptsize $\ldots$};
\draw [->,line width=.4pt] (-0.5,-1.5) -- (-0.5,1.);
\draw [shift={(0.,-1.5)},line width=0.4pt]  plot[domain=3.141592654:6.283185307,variable=\t]({1.*0.5*cos(\t r)+0.*0.5*sin(\t r)},{0.*0.5*cos(\t r)+1.*0.5*sin(\t r)});
\draw [line width=.4pt] (0.5,-1.5) -- (0.5,1);
\draw [shift={(1.,1.)},line width=0.4pt]  plot[domain=1.570796327:3.141592654,variable=\t]({1.*0.5*cos(\t r)+0.*0.5*sin(\t r)},{0.*0.5*cos(\t r)+1.*0.5*sin(\t r)});
\draw [line width=.4pt] (1.,1.5) -- (1.5,1.5);
\draw [shift={(1.5,1.)},line width=0.4pt]  plot[domain=0.:1.570796327,variable=\t]({1.*0.5*cos(\t r)+0.*0.5*sin(\t r)},{0.*0.5*cos(\t r)+1.*0.5*sin(\t r)});
\end{tikzpicture}
\end{align*}
where the curved edge relate the first edge on the bottom to the $j$-th edge on the top (note that this edge is also
the $(j-1)$-th outgoing the vertex decorated by $p$). As $\alpha$ is a $\sym\times \sym^{op}$ morphism,
we obtain that this is $(1,\ldots,j-1)\cdot \alpha \circ \nu_P(p)$, that is to say $(1,\ldots,j-1)\cdot p$. 

In this way, we define a functor from the category of $r\Gacirc$-algebras to the category of TRAPs.
In the same way, we define a functor from the category of $\Gacirc$-algebras to the category of unitary TRAPs.

{We obtain in this way two functors
\begin{align*}
\mathcal{F}&:\Trap\longrightarrow\rGacirc-\mathbf{Alg},&
\mathcal{G}&:\rGacirc-\mathbf{Alg}\longrightarrow\Trap.
\end{align*}
Let $P$ be a {TRAP} and $P'$ the {TRAP} $\mathcal{G}\circ \mathcal{F}(P)$, with concatenation $*'$
and trace operators $t'_{i,j}$. 
We set $\mathcal{F}(P):=(P,\alpha_P)$: in other words, $\alpha_P$ is the TRAP morphism from $\rGacirc(P)$ to $P$
which is the identity on $P$. For any $p,q\in P$:
\[p*'q=\alpha_P (\nu_P(p)*\nu_P(q))={p*q},\]
where in the middle term $*$ is the concatenation in the TRAP $\rGacirc(P)$ {and where we used that $\alpha_P$ is a TRAP morphism by Proposition \ref{prop:freerGamma}}. Therefore, $*=*'$.
If $p\in P(k,l)$, $(i,j)\in [k]\times [l]$, then $t'_{i,j}$ is obtained by the application of $\alpha_P$ to the graph:
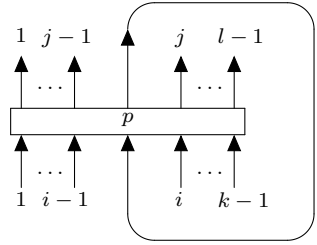
\[\begin{tikzpicture}[line cap=round,line join=round,>=triangle 45,x=0.7cm,y=0.7cm]
\clip(0.8,-2.6) rectangle (6.5,2.1);
\draw [line width=.4pt] (0.8,0.)-- (5.2,0.);
\draw [line width=.4pt] (5.2,0.)-- (5.2,-0.5);
\draw [line width=.4pt] (5.2,-0.5)-- (0.8,-0.5);
\draw [line width=.4pt] (0.8,-0.5)-- (0.8,0.);
\draw (2.7,0.1) node[anchor=north west] {\scriptsize $p$};
\draw [->,line width=.4pt] (1.,0.) -- (1.,1.);
\draw [->,line width=.4pt] (2.,0.) -- (2.,1.);
\draw [->,line width=.4pt] (3.,0.) -- (3.,1.5);
\draw [->,line width=.4pt] (4.,0.) -- (4.,1.);
\draw [->,line width=.4pt] (5.,0.) -- (5.,1.);
\draw [->,line width=.4pt] (1.,-1.5) -- (1,-0.5);
\draw [->,line width=.4pt] (2.,-1.5) -- (2.,-0.5);
\draw [->,line width=.4pt] (3.,-2) -- (3,-0.5);
\draw [->,line width=.4pt] (4.,-1.5) -- (4.,-0.5);
\draw [->,line width=.4pt] (5.,-1.5) -- (5.,-0.5);
\draw (0.7,-1.4) node[anchor=north west] {\scriptsize $1$};
\draw (1.1,-1.) node[anchor=north west] {\scriptsize $\ldots$};
\draw (1.2,-1.4) node[anchor=north west] {\scriptsize $i-1$};
\draw (3.7,-1.4) node[anchor=north west] {\scriptsize $i$};
\draw (4.1,-1.) node[anchor=north west] {\scriptsize $\ldots$};
\draw (4.5,-1.4) node[anchor=north west] {\scriptsize $k-1$};
\draw (0.7,1.7) node[anchor=north west] {\scriptsize $1$};
\draw (1.1,0.6) node[anchor=north west] {\scriptsize $\ldots$};
\draw (1.2,1.7) node[anchor=north west] {\scriptsize $j-1$};
\draw (3.7,1.7) node[anchor=north west] {\scriptsize $j$};
\draw (4.1,0.6) node[anchor=north west] {\scriptsize $\ldots$};
\draw (4.5,1.7) node[anchor=north west] {\scriptsize $l-1$};
\draw [shift={(3.5,-2.)},line width=0.4pt]  plot[domain=3.141592653589793:4.71238898,variable=\t]({1.*0.5*cos(\t r)+0.*0.5*sin(\t r)},{0.*0.5*cos(\t r)+1.*0.5*sin(\t r)});
\draw [line width=.4pt] (3.5,-2.5) -- (6.,-2.5);
\draw [shift={(6.,-2.)},line width=0.4pt]  plot[domain=4.71238898:6.283185307,variable=\t]({1.*0.5*cos(\t r)+0.*0.5*sin(\t r)},{0.*0.5*cos(\t r)+1.*0.5*sin(\t r)});
\draw [shift={(3.5,1.5)},line width=0.4pt]  plot[domain=1.570796327:3.141592653589793,variable=\t]({1.*0.5*cos(\t r)+0.*0.5*sin(\t r)},{0.*0.5*cos(\t r)+1.*0.5*sin(\t r)});
\draw [line width=.4pt] (3.5,2.) -- (6.,2.);
\draw [shift={(6,1.5)},line width=0.4pt]  plot[domain=0.:1.570796327,variable=\t]({1.*0.5*cos(\t r)+0.*0.5*sin(\t r)},{0.*0.5*cos(\t r)+1.*0.5*sin(\t r)});
\draw [line width=.4pt] (6.5,-2.) -- (6.5,1.5);
\end{tikzpicture}\]
which is $t_{i,j}(\nu_P(p))$, where here $t_{i,j}$ is the trace operator of $\rGacirc(P)$. As $\alpha_P$ is a TRAP morphism:
\[t'_{i,j}(p)=\alpha_P \circ t_{i,j}\circ \nu_P(p)=t_{i,j}\circ \alpha_P \circ \nu_P(p)=t_{i,j}(p),\]
so $P'=P$ and $\mathcal{G}\circ \mathcal{F}$ is the identity functor of $\Trap$.

Let now $(P,\alpha)$ be a $\rGacirc$-algebra and let us consider $(P',\alpha')$ be the $\rGacirc$-algebra
$ \mathcal{F}\circ  \mathcal{G}(P)$. Both $\alpha$ and $\alpha'$ are TRAP morphisms from
$\rGacirc(P)$ to $\mathcal{G}(P)$; for any $p\in P$,
\[\alpha \circ \nu_P(p)=\alpha'\circ \nu_P(p)=p.\]
{Since $\rGacirc(P)$  viewed as  a TRAP is generated}  by the elements $\nu_P(p)$, $\alpha=\alpha'$,
it follows that $\mathcal{F}\circ \mathcal{G}$ is the identity functor of $\rGacirc-\mathbf{Alg}$.\\

The proof is similar in the unitary case. }\end{proof}

Remark \ref{rk:MMS09} gives a straightforward Corollary of Theorems \ref{thm:equivalence_TRAP_wPROP} and \ref{freetraps},
 thus confirming previous statements.

\begin{cor}
 $\PGr(X)$ is the free wheeled PROP generated by $X$.
\end{cor}

\begin{remark}
The monad $\Gacirc$ contains an interesting sub-monad, formed by graphs without any oriented cycle (which includes also loops).
This submonad is denoted by $\Ga$. It is well-known that $\Ga$ is the monad of PROPs \cite{Markl}.
Hence, unitary  TRAPs are PROPS. In particular,   a unitary TRAP $P$ inherits a vertical composition denoted by $\circ$,
which is the one described in Proposition \ref{propverticalconcatenation} in the more general frame of (non unitary) TRAPs.
\end{remark}

\appendix

\section{Appendix: topologies on tensor products} \label{section:topologies_tens_prod}

Tensor products ot topological spaces can be equipped with various topologies. A first possibility is the so-called \textbf{$\epsilon$-topology}; \cite[Definition 43.1]{Treves67}. For two   topological vector spaces $E$ and $F$, one 
can show (\cite[Proposition 42.4]{Treves67}) the isomorphism  of vector spaces
$E\otimes F\simeq \mathcal{B}^c(E'_\sigma\times F'_\sigma,\K)$ where 
$\mathcal{B}^c(E'_\sigma\times F'_\sigma,\K)$ denotes the space of continuous bilinear maps from $E'_\sigma\times F'_\sigma$ to $\K$ and $E'_\sigma$ (resp. 
$F'_\sigma$)  the topological dual of $E$ (resp. $F$) for $\sigma$, the weak topology. 

Recall that a bilinear map $f:E\times F\longrightarrow K$ is called separately continuous if, 
for any pair $(x,y)\in E\times F$, the maps $z\to f(x,z)$ and $z\to f(z,y)$ are continuous. We then clearly have that continuous bilinear maps 
build a linear subspace of the space $\mathcal{B}^{sc}(E\times F,\K)$ of separately continuous bilinear maps.

The space  $\mathcal{B}^{sc}(E\times F,\K)$ can be equipped with the topology of uniform convergence on products of equicontinuous subsets of $E'_\sigma$ with 
equicontinuous subsets of $F'_\sigma$. Recall that, for a topological space $X$ and a topological vector space $G$, a set $S$ of maps from $X$ to $G$ is 
said to be equicontinuous at $x_0\in X$ if, for any $V\subseteq G$ neighbourhood of zero, there is some neighbourhood  $V(x_0)\subseteq X$   of $x_0$, such 
that 
\begin{equation*}
\forall f\in S,~x\in V(x_0) \Rightarrow f(x)-f(x_0)\in V.
\end{equation*}
In our case, $G$ is $\K$ and $X$ is $E_\sigma$ (resp. $F_\sigma$). This topology induces a topology on the subspace 
$\mathcal{B}^c(E'_\sigma\times F'_\sigma,\K)$ and thus on $E\otimes F$. We denote by $E\otimes_\epsilon F$ the topological vector space obtained by 
endowing $E\otimes F$ with this topology.

There is another  topology on $E\otimes F$   called the \textbf{projective topology}; 
\cite[Definition 43.2]{Treves67}.  The projective topology is defined as the strongest 
locally convex topology on $E\otimes F$ such that the canonical map $\phi:E\times F\longrightarrow E\otimes F$ is continuous. 
We write $E\otimes_\pi F$ the topological vector space obtained by 
endowing $E\otimes F$ with this topology.

The neighbourhoods of zero of the projective topology can be simply described in terms of neighbourhoods of zero in $E$ and $V$. A convex subset $S$ of 
$E\otimes F$ containing zero is a neighbourhood of zero if it exist a neighbourhood $U$ (resp. V) of zero in $E$ (resp. $F$) such that 
$U\otimes V:=\{u\otimes v|u\in U\wedge v\in V\}\subseteq S$.

Various  topologies can be defined on the vector space $E\otimes F$ for $E$ and $F$ two topological vector spaces. However the projective topology and the $\epsilon$-topology play an important special role since they allow to define nuclear spaces (see Definition \ref{defi:nuclear_spaces}).

\section{Appendix: definition of the partial trace maps on $\Gr$} \label{Appendix:B}

We give a rigorous definition of the partial trace maps on the space of graphs $\Gr$, which were only loosely defined in the bulk of the article.

Let $G\in \Gr(k,l)$ with $k,l\geqslant 1$, $i\in [k]$ and $j\in [l]$. We put $e_i=\alpha_G^{-1}(i)$
and $f_j=\beta_G^{-1}(j)$. We define the graph $G'=t_{i,j}(G)$ in the following way:
\begin{enumerate}
	\item If $e_i\in I(G)$ and $f_j\in O(G)$, then:
	\begin{align*}
	V(G')&=V(G),&E(G')&=E(G)\sqcup \{(e_i,f_j)\},\\
	I(G')&=I(G)\setminus\{e_i\},& O(G')&=O(G)\setminus\{f_j\},\\
	IO(G')&=IO(G),&L(G')&=L(G),\\
	s_{G'}(e)&=\begin{cases}
	s_G(f_j)\mbox{ if }e=(e_i,f_j),\\
	s_G(e)\mbox{ otherwise},
	\end{cases}&
	t_{G'}(e)&=\begin{cases}
	t_G(e_i)\mbox{ if }e=(e_i,f_j),\\
	t_G(e)\mbox{ otherwise},
	\end{cases}\\
	\alpha_{G'}(e)&=\begin{cases}
	\alpha_G(e)\mbox{ if }\alpha_G(e)<i,\\
	\alpha_G(e)-1\mbox{  if }\alpha_G(e)\geqslant i,
	\end{cases}&
	\beta_{G'}(e)&=\begin{cases}
	\beta_G(e)\mbox{ if }\beta_G(e)<j,\\
	\beta_G(e)-1\mbox{ if }\beta_G(e)\geqslant j.
	\end{cases}
	\end{align*}
	
	\item If $e_i\in IO(G)$ and $f_j\in O(G)$, then:
	\begin{align*}
	V(G')&=V(G),&E(G')&=E(G),\\
	I(G')&=I(G),& O(G')&=O(G)\setminus\{f_j\}\sqcup \{(e_i,f_j,)\},\\
	IO(G')&=IO(G)\setminus\{e_i\},&L(G')&=L(G),\\
	s_{G'}(e)&=\begin{cases}
	s_G(f_j)\mbox{ if }e=(e_i,f_j),\\
	s_G(e)\mbox{ otherwise},
	\end{cases}&
	t_{G'}(e)&=t_G(e),\\
	\alpha_{G'}(e)&=\begin{cases}
	\alpha_G(e)\mbox{ if }\alpha_G(e)<i,\\
	\alpha_G(e)-1\mbox{  if }\alpha_G(e)\geqslant i,
	\end{cases}&
	\beta_{G'}(e)&=\begin{cases}
	\beta_G(e_i)\mbox{ if }e=(e_i,f_j)\mbox{ and }\beta_G(e_i)<j,\\
	\beta_G(e_i)-1\mbox{ if }e=(e_i,f_j)\mbox{ and }\beta_G(e_i)\geqslant j,\\
	\beta_G(e)\mbox{ if }e\neq(e_i,f_j)\mbox{ and }\beta_G(e)<j,\\
	\beta_G(e)-1\mbox{ if }e\neq(e_i,f_j)\mbox{ and }\beta_G(e)\geqslant j.
	\end{cases}
	\end{align*}
	
	\item If $e_i\in I(G)$ and $f_j\in IO(G)$, then:
	\begin{align*}
	V(G')&=V(G),&E(G')&=E(G),\\
	I(G')&=I(G)\setminus\{e_i\}\sqcup \{(e_i,f_j)\},& O(G')&=O(G),\\
	IO(G')&=IO(G)\setminus\{f_j\},&L(G')&=L(G),\\
	s_{G'}(e)&=s_G(e),&
	t_{G'}(e)&=\begin{cases}
	t_G(e_i)\mbox{ if }e=(e_i,f_j),\\
	t_G(e)\mbox{ otherwise},
	\end{cases}\\
	\alpha_{G'}(e)&=\begin{cases}
	\alpha_G(f_i)\mbox{ if }e=(e_i,f_j)\mbox{ and }\alpha_G(f_j)<i,\\
	\alpha_G(f_i)-1\mbox{ if }e=(e_i,f_j)\mbox{ and }\alpha_G(f_j)\geqslant i,\\
	\alpha_G(e)\mbox{ if }e\neq(e_i,f_j)\mbox{ and }\alpha_G(e)<i,\\
	\alpha_G(e)-1\mbox{ if }e\neq(e_i,f_j)\mbox{ and }\alpha_G(e)\geqslant i,
	\end{cases}&
	\beta_{G'}(e)&=\begin{cases}
	\beta_G(e)\mbox{ if }\beta_G(e)<j,\\
	\beta_G(e)-1\mbox{ if }\beta_G(e)\geqslant j.
	\end{cases}
	\end{align*}

	\item If $e_i\in IO(G)$, $f_j\in IO(G)$ and $e_i\neq f_j$, then:
	\begin{align*}
	V(G')&=V(G),&E(G')&=E(G),\\
	I(G')&=I(G),& O(G')&=O(G),\\
	IO(G')&=\{(e_i,f_j)\}\sqcup IO(G)\setminus\{e_i,f_j\},&L(G')&=L(G),\\
	s_{G'}(e)&=s_G(e),&
	t_{G'}(e)&=t_G(e),\\
	\alpha_{G'}(e)&=\begin{cases}
	\alpha_G(f_i)\mbox{ if }e=(e_i,f_j)\mbox{ and }\alpha_G(f_j)<i,\\
	\alpha_G(f_i)-1\mbox{ if }e=(e_i,f_j)\mbox{ and }\alpha_G(f_j)\geqslant i,\\
	\alpha_G(e)\mbox{ if }e\neq(e_i,f_j)\mbox{ and }\alpha_G(e)<i,\\
	\beta_G(e)-1\mbox{ if }e\neq(e_i,f_j)\mbox{ and }\alpha_G(e)\geqslant i,
	\end{cases}\\
	\beta_{G'}(e)&=\begin{cases}
	\beta_G(e_i)\mbox{ if }e=(e_i,f_j)\mbox{ and }\beta_G(e_i)<j,\\
	\beta_G(e_i)-1\mbox{ if }e=(e_i,f_j)\mbox{ and }\beta_G(e_i)\geqslant j,\\
	\beta_G(e)\mbox{ if }e\neq(e_i,f_j)\mbox{ and }\beta_G(e)<j,\\
	\beta_G(e)-1\mbox{ if }e\neq(e_i,f_j)\mbox{ and }\beta_G(e)\geqslant j.
	\end{cases}
	\end{align*}

	\item If $e_i\in IO(G)$, $f_j\in IO(G)$ and $e_i=f_j$, then:
	\begin{align*}
	V(G')&=V(G),&E(G')&=E(G),\\
	I(G')&=I(G),& O(G')&=O(G),\\
	IO(G')&=IO(G)\setminus\{e_i,f_j\},&L(G')&=L(G)\sqcup \{(e_i,f_j)\},\\
	s_{G'}(e)&=s_G(e),&
	t_{G'}(e)&=t_G(e),\\
	\alpha_{G'}(e)&=\begin{cases}
	\alpha_G(e)\mbox{ if }\alpha_G(e)<i,\\
	\alpha_G(e)-1\mbox{  if }\alpha_G(e)\geqslant i,
	\end{cases}&
	\beta_{G'}(e)&=\begin{cases}
	\beta_G(e)\mbox{ if }\beta_G(e)<j,\\
	\beta_G(e)-1\mbox{ if }\beta_G(e)\geqslant j.
	\end{cases}
	\end{align*}
\end{enumerate}

\section{Appendix: Freeness of $\PGr(X)$}

We now give a detailed proof of Theorem \ref{freetraps}.
\begin{proof}
We simultaneously prove the unitary and non-unitary cases.

We first define $\Phi:\PGr(X)\longrightarrow P$ by assigning to
any graph   {$G\in \PGr(X)(k,l)$ (or $G\in\rPGr(X)(k,l)$ for the non-unitary case)
 such that  $\Phi(\sigma\cdot G\cdot \tau)=\sigma\cdot\Phi(G)\cdot \tau$ any $(\sigma,\tau)\in \sym_l\times \sym_k$}.  We proceed by induction on the number $N$  of internal edges of $G$. If $N=0$, then $G$ can be written (non uniquely) as
\[G=\grapheo^{*p}*\sigma\cdot(I^{*q}*G_{k_1,l_1}*\ldots *G_{k_r,l_r})\cdot \tau,\]
where $p,q,r\in  \N_0$ are unique, $(k_i,k_i) \in  \N_0^2$ for any $i$, unique up to a permutation, and $\sigma \in \sym_{q+k_1+\ldots+k_r}$, $\tau\in \sym_{q+l_1+\ldots+l_r}$. Notice that in the non-unitary case we necessarily have $p=q=0$, and we have set in this case $g^{*0}=I_0=\emptyset$, the empty graph, for any graph $g$.

We then put:
\[\Phi(G)=t_{1,1}(I)^{*p}*\sigma\cdot(I^{*q}*x_{k_1,l_1}*\ldots*x_{k_r,l_r})\cdot \tau,\]
{where as before $x^{*0}=I_0$ (the unit for horizontal concatenation in the image $P$ of $\Phi$) for any $x\in P$; and $I$ is now the unit of $P$ in the unitary case.}

Let us prove that this does not depend of the choice of the decomposition of $G$. Such a decomposition is determined  modulo a permutation of the vertices
and of the choice of $\sigma$ and $\tau$. Thus,  we can go from one decomposition of $G$ to any other one by means of a finite number of steps among the following two types:
\begin{enumerate}
\item We consider two decompositions of $G$ of the form
\begin{align*}
G&=\grapheo^{*p}*\sigma\cdot(I^{*q}*G_{k_1,l_1}*\ldots* G_{k_i,l_i}*G_{k_{i+1},l_{i+1}} *\ldots *G_{k_r,l_r})\cdot \tau,\\
G&=\grapheo^{*p}*\sigma'\cdot(I^{*q}*G_{k_1,l_1}*\ldots*G_{k_{i+1},l_{i+1}}* G_{k_i,l_i} *\ldots *G_{k_r,l_r})\cdot \tau',
\end{align*}
with
\begin{align*}
\sigma'&=\sigma(\mathrm{Id}_{q+l_1+\ldots+l_{i-1}}\otimes c_{l_i,l_{i+1}}\otimes \mathrm{Id}_{l_{i+2}+\ldots+l_r}),\\
\tau'&=(\mathrm{Id}_{q+k_1+\ldots+k_{i-1}}\otimes c_{k_{i+1},k_i}\otimes \mathrm{Id}_{k_{i+2}+\ldots+k_r})\tau.
\end{align*}
Then, by commutativity of $*$:
\begin{align*}
&\sigma'\cdot(I^{*q}*x_{k_1,l_1}*\ldots*x_{k_r,l_r})\cdot \tau'\\
&=\sigma\cdot (I^{*q}*x_{k_1,l_1}*\ldots * c_{l_i,l_{i+1}}
\cdot(x_{k_{i+1},l_{i+1}}*x_{k_i,l_i})\cdot c_{k_{i+1},k_i}*\ldots*x_{k_r,l_r})\cdot \tau\\
&=\sigma\cdot (I^{*q}*x_{k_1,l_1}*\ldots *x_{k_i,l_i}*x_{k_{i+1},l_{i+1}}*\ldots *x_{k_r,l_r})\cdot \tau.
\end{align*}
\item We consider two decompositions of $G$ of the form
\begin{align*}
G&=\grapheo^{*p}*\sigma\cdot(I^{*q}*G_{k_1,l_1}*\ldots*G_{k_r,l_r})\cdot \tau,\\
G&=\grapheo^{*p}*\sigma'\cdot(I^{*q}*G_{k_1,l_1}*\ldots*G_{k_r,l_r})\cdot \tau',
\end{align*}
with
\begin{align*}
\sigma'&=\sigma (\sigma_0\otimes \sigma_1\otimes \ldots \otimes \sigma_r),&
\tau'&=(\sigma_0^{-1}\otimes \tau_1\otimes \ldots \otimes \tau_r)\tau',
\end{align*}
with $\sigma_0\in \sym_q$, $\sigma_i\in \sym_{k_i}$ and $\tau_i\in \sym_{l_i}$ if $i\geqslant 1$.
Using the commutativity of $*$ and the invariance of the $x_{k,l}$, we find
\begin{align*}
&\sigma'\cdot (I^{*q}*x_{k_1,l_1}*\ldots*x_{k_r,l_r})\cdot \tau'\\
&=\sigma \cdot (\sigma_0\cdot I^{*q}\cdot \sigma_0^{-1}
*\sigma_1\cdot x_{k_1,l_1}\cdot \tau_1*\ldots*\sigma_r\cdot x_{k_r,l_r}\cdot \tau_r)\cdot \tau\\
&=\sigma\cdot (I^{*q}*x_{k_1,l_1}*\ldots*x_{k_r,l_r})\cdot \tau.
\end{align*}
\end{enumerate}
Notice that setting $p=q=0$ in these computations does not change the result.

Hence, $\Phi(G)$ is well-defined. Moreover, for any  $\tau' \in \sym_k$, $\sigma'\in \sym_l$,  a decomposition  of $G$ of the form
\[G=\grapheo^{*p}*\sigma\cdot(I^{*q}*G_{k_1,l_1}*\ldots* G_{k_i,l_i}*G_{k_{i+1},l_{i+1}} *\ldots *G_{k_r,l_r})\cdot \tau,\]
give rise to a decomposition of $G'=\sigma'\cdot G\cdot \tau'$ given by
\[\grapheo^{*p}*\sigma'\sigma\cdot(I^{*q}*G_{k_1,l_1}*\ldots*G_{k_r,l_r})\cdot \tau\tau',\]
and, by definition of $\Phi(G')$:
\begin{align*}
\Phi(G')&=t_{1,1}(I)^{*p}*\sigma'\sigma\cdot(I^{*q}*x_{k_1,l_1}*\ldots*x_{k_r,l_r})\cdot \tau\tau'\\
&=\sigma'\cdot( t_{1,1}(I)^*p*\sigma\cdot(I^{*q}*x_{k_1,l_1}*\ldots*x_{k_r,l_r})\cdot \tau)*\tau'\\
&=\sigma'\cdot \Phi(G)\cdot \tau'.
\end{align*}
{Here again, the  computations are valid in particular in the case $p=q=0$, so for the non-unitary case.}

Let us assume now that $\Phi(G')$ is defined for any graph with $N-1$ internal edges, for a given $N \geqslant 1$.
Let $G$ be a graph with $N$ internal edges and let $e$ be one of these edges. 
Let $G_e$ be a graph obtained by cutting this edge in two:
\begin{enumerate}
\item $V(G_e)=V(G)$.
\item $E(G_e)=E(G)\setminus \{e\}$, $I(G_e)=I(G)\sqcup \{e\}$, $O(G_e)=O(G)\sqcup \{e\}$, $IO(G_e)=IO(G)$, $L(G_e)=L(G)$.
\item $s_{G_e}=s_G$ and $t_{G_e}=t_G$.
\item For any $e'\in I(G_e)\sqcup IO(G_e)$, for any $f'\in O(G_e)\sqcup IO(G_e)$:
\begin{align*}
\alpha_{G_e}(e')&=\begin{cases}
1\mbox{ if }e'=e,\\
\alpha_{G}(e')+1\mbox{ if }e'\neq e,
\end{cases}
&\beta_{G_e}(f')&=\begin{cases}
1\mbox{ if }f'=e,\\
\beta_{G}(f')+1\mbox{ if }f'\neq e.
\end{cases}
\end{align*}
\end{enumerate}
{Notice that if $G\in\rPGr(X)$ (i.e. $IO(G)=L(G)=\emptyset$) then $G_e$ also lies in $\rPGr(X)$. Then, as before, we can treat the unitary and non-unitary cases simultaneously. In both cases we have} $G=t_{1,1}(G_e)$ and $G_e$ has $N-1$ internal edges. We then put:
\[\Phi(G)=t_{1,1}\circ \Phi(G_e).\]
Let us prove that this does not depend of the choice of $e$. If $e'$ is another internal edge of $G$,
then:
\[(G_e)_{e'}=(12)\cdot (G_{e'})_e\cdot (12),\]
which implies, by definition of $\Phi(G_e)$ and $\Phi(G_{e'})$:
\begin{align*}
t_{1,1}\circ \Phi(G_e)&=t_{1,1}\circ t_{1,1}\circ \Phi((G_e)_{e'})\\
&=t_{1,1}\circ t_{1,1} \circ ((12)\cdot \Phi((G_{e'})_e)\cdot (12))\\
&=t_{1,1}\circ t_{2,2}\circ \Phi((G_{e'})_e)\\
&=t_{1,1}\circ t_{1,1}\circ \Phi((G_{e'})_e)\\
&=t_{1,1}\circ \Phi(G_{e'}).
\end{align*}
So $\Phi(G)$ is well-defined. Let $\sigma \in \sym_k$ and $\tau\in \sym_l$. Then:
\[(\sigma\cdot G\cdot \tau)_e=((1)\otimes \sigma)\cdot (G_e)\cdot ((1)\otimes \tau),\]
so:
\begin{align*}
\Phi(\sigma \cdot G\cdot \tau)&=t_{1,1}\circ \Phi((\sigma\cdot G\cdot \tau)_e)\\
&=t_{1,1}((1)\otimes \sigma)\cdot \Phi(G_e)\cdot ((1)\otimes \tau)\\
&=((1)\otimes \sigma)_1\cdot t_{1,1}\circ \Phi(G_e)\cdot ((1)\otimes \tau)_1\\
&=\sigma\cdot \Phi(G)\cdot \tau.
\end{align*}
where, for $\sigma\in\sym_k$ we use $\sigma_i$ for the permutation in $\sym_{k-1}$ defined by
\begin{equation*}
  \sigma_i(j) = \begin{cases}
               & \sigma(j) \quad\text{if }j\leq i-1, \\
               & \sigma(j-1)\quad\text{if }j\geq i. 
              \end{cases}
\end{equation*}
where $((1)\otimes \tau)_1$ is defined by (\ref{defalphak}).

We have therefore defined a map {$\Phi:\PGr(X)\longrightarrow P$ (resp. $\Phi:\rPGr(X)\longrightarrow P$ in the non-unitary case)} compatible with the action of the symmetric groups. Let us prove that for any graphs $G$, $G'$,
\[\Phi(G*G')=\Phi(G)*\Phi(G').\]
We proceed by induction on the number $N$ of internal edges of $G*G'$. If $N=0$, we put:
\begin{align*}
G&=\grapheo^{*p}*\sigma\cdot (I^{*q}*G_{k_1,l_1}*\ldots*G_{k_r,l_r})\cdot \tau,\\
G'&=\grapheo^{*p'}*\sigma'\cdots (I^{*q'}*G_{k'_1,l'_1}*\ldots*G_{k'_{r'},l'_{r'}})\cdot \tau'.
\end{align*}
As before, if we are in the non-unitary case we set $p=q=p'=q'=0$ and the whole discussion still holds. We obtain:
\begin{align*}
G*G'&=\grapheo^{*(p+p')}*(\sigma \otimes \sigma')* (\mathrm{Id}_q\otimes c_{k_1+\ldots+k_r,q'}\otimes \mathrm{Id}_{k'_1+\ldots+k'_{r'}})\\
&\cdot (I^{q+q'}*G_{k_1,l_1}*\ldots*G_{k'_{r'},l'_{r'}})\cdot
(\mathrm{Id}_q\otimes c_{q',l_1+\ldots+l_r}\otimes \mathrm{Id}_{l'_1+\ldots+l'_{r'}}),
\end{align*}
which gives, by commutativity of $*$:
\begin{align*}
\Phi(G*G')&=t_{1,1}(I)^{*(p+p')}*(\sigma \otimes \sigma')* (\mathrm{Id}_q\otimes c_{l_1+\ldots+l_r,q'}
\otimes \mathrm{Id}_{l'_1+\ldots+l'_{r'}})\\
&\cdot (I^{q+q'}*x_{k_1,l_1}*\ldots*x_{k'_{r'},l'_{r'}})\cdot
(\mathrm{Id}_q\otimes c_{q',k_1+\ldots+k_r}\otimes \mathrm{Id}_{k'_1+\ldots+k'_{r'}})\\
&=t_{1,1}(I)^{*p}*\sigma\cdot(I^{*q}*x_{k_1,l_1}*\ldots*x_{k_r,l_r})\cdot \tau\\
&*t_{1,1}(I)^{*p'}*\sigma'\cdot(I^{*q'}*x_{k'_1,l'_1}*\ldots*x_{k'_{r'},l'_{r'}})\cdot \tau'\\
&=\Phi(G)*\Phi(G').
\end{align*}
In the non-unitary case, the TRAP $P$ has no unit $P$ and one simply removes the terms with the identity $I$ in the above computation and sets $p=q=p'=q=0$. In this case, the result is the same as in the unitary case: $\Phi(G*G')=\Phi(G)*\Phi(G')$.

If $N\geqslant 1$, let us take an internal edge $e$ of $G*G'$. If $e$ is an internal edge of $G$, then
$(G*G')_e=G_e*G'$, and:
\begin{align*}
 \Phi(G*G')&=t_{1,1}\circ \Phi((G*G')_e)\\
 &=t_{1,1}\circ \Phi(G_e*G')\\
 &=t_{1,1}(\Phi(G_e)*G')\\
 &=t_{1,1}\circ \Phi(G_e)*\Phi(G')=\Phi(G)*\Phi(G').
\end{align*}
If $e$ is an internal edge of $G'$, we obtain similarly that $\Phi(G'*G)=\Phi(G')*\Phi(G)$.
The result then 
follows from the commutativity of $*$ (axiom $2.(d)$ of Definition \ref{defi:Trap}).
So $\Phi$ is compatible with $*$.

We still need to prove the compatibility of $\Phi$ with the partial trace maps. By Lemma \ref{lemmemorphismes}, it is enough to prove that $\Phi$ is compatible with $t_{1,1}$. Let {$G\in \PGr(X)(k,l)$ (or $G\in\rPGr(X)(k,l)$ in the non-unitary case)} be a graph, $e_1=\alpha^{-1}(1)$, $f_1=\beta^{-1}(1)$. We set $G'=t_{1,1}(G)$ and $e=\{e_1,f_1\}$ to be the edge of $G'$ created in the process. {Notice that if $G\in\rPGr(X)$ then $G'\in\rPGr(X)$.} There are five different cases {(but only the first cases appear if $G\in\rPGr(X)$)}:
\begin{enumerate}
\item If $e_1\in I(G)$ and $f_1\in O(G)$, then $e\in E(G')$ and $G'_e=G$. By construction of $\Phi(G')$:
\[\Phi\circ t_{1,1}(G)=\Phi(G')=t_{1,1}\circ \Phi(G'_e)=t_{1,1}\circ \Phi(G).\]
\item If $e_1\in IO(G)$ and $f_1\in O(G)$, let us put $j=\beta(e_1)$. Then there exists a graph $H$
such that $(1,j)\cdot G=I*H$, hence
\begin{align*}
t_{1,1}(G)&=t_{1,1}((1,j)\cdot(I*H))=(1,\ldots,j)\cdot (t_{1,i}(I*H))=
(1,\ldots,j)\cdot H,
\end{align*}
so:
\begin{align*}
t_{1,1}\circ \Phi(G)&=t_{1,1}((1,j)\cdot (I*\Phi(H))\\
&=  (1,j)(1,\ldots,j-1)\cdot \Phi(H)\\
&=(1,\ldots,j)\cdot \Phi(H)\\
&=\Phi((1,\ldots,j)\cdot H)\\
&=\Phi\circ t_{1,1,}(G).
\end{align*}
\item If $e_1\in I(G)$ and $f_1\in IO(G)$ the  computation  is similar.
\item If $e_1,f_1\in IO(G)$, with $e_1\neq f_1$  the  computation  is similar.
\item If $e_1=f_1$ in $IO(G)$, then $G=I*H$ for a certain graph $G$ and $t_{1,1}(G)=\grapheo*H$.
Then:
\[\Phi\circ t_{1,1}(G)=\Phi(\grapheo)*\Phi(H)=t_{1,1}\circ \Phi(I)*\Phi(H)=t_{1,1}(\Phi(I)*\Phi(H)
=t_{1,1}\circ \Phi(G).\]
\end{enumerate}
So $\Phi$ is compatible with the partial trace maps{, both in the unitary and non-unitary cases}. 
\end{proof}
 
\bibliographystyle{alpha}
 \addcontentsline{toc}{section}{References}
\bibliography{biblio_new}
 
\end{document}